\documentclass[bezier,amstex, oneside, reqno]{amsart}
\usepackage{amsmath, amssymb, amsthm, amscd, verbatim, euscript}
\usepackage[dvips]{graphicx}
\usepackage{epsfig}
\usepackage{color}
\begingroup\makeatletter\ifx\SetFigFont\undefined%
\gdef\SetFigFont#1#2#3#4#5{%
  \reset@font\fontsize{#1}{#2pt}%
  \fontfamily{#3}\fontseries{#4}\fontshape{#5}%
  \selectfont}%
\fi\endgroup%

\def\ie{\emph{i.e., }}

\newtheorem*{theorem}{Theorem}
\newtheorem*{theoremA}{Theorem A}
\newtheorem*{theoremB}{Theorem B}
\newtheorem*{theoremC}{Theorem C}
\newtheorem*{theoremD}{Theorem D}

\newtheorem*{propertyA}{Property $\mathcal A$}
\newtheorem*{propertyB}{Property $\mathcal A'$}
\newtheorem{q}{Question}
\newtheorem{con}{Conjecture}
\newtheorem{lemma}{Lemma}[section]
\newtheorem{proposition}{Proposition}

\newtheorem*{reglemma}{Regularity Lemma}
\newtheorem{definition}{Definition}

\theoremstyle{remark}
\newtheorem*{remark}{Remark}

\begin{document}
\author{Andrey Gogolev}
\title[Smooth conjugacy of Anosov systems]{Smooth conjugacy of Anosov diffeomorphisms on higher dimensional tori}
\date{September 26, 2008}

\begin{abstract}
Let $L$ be a hyperbolic automorphism of\/ $\mathbb T^d$, $d\ge3$. We
study the smooth conjugacy problem in a small $C^1$-neighborhood
$\mathcal U$ of\/ $L$.

The main result establishes $C^{1+\nu}$ regularity of the conjugacy
between two Anosov systems with the same periodic eigenvalue data.
We assume that these systems are $C^1$-close to an irreducible
linear hyperbolic automorphism $L$ with simple real spectrum and
that they satisfy a natural transitivity assumption on certain
intermediate foliations.

We elaborate on the example of de la Llave of two Anosov systems on
$\mathbb T^4$ with the same constant periodic eigenvalue data that
are only H\"older conjugate. We show that these examples exhaust all
possible ways to perturb $C^{1+\nu}$ conjugacy class without
changing periodic eigenvalue data. Also we generalize these examples
to majority of reducible toral automorphisms as well as to certain
product diffeomorphisms of\/ $\mathbb T^4$ $C^1$-close to the
original example.
\end{abstract}

 \maketitle

\tableofcontents

\section{Introduction and statements}

Consider an Anosov diffeomorphism $f$ of a compact smooth manifold.
Structural stability asserts that if a diffeomorphism $g$ is $C^1$
close to $f$, then $f$ and $g$ are topologically conjugate, \ie $$h
\circ f = g \circ h.$$ The conjugacy $h$ is unique in the
neighborhood of identity. It is known that $h$ is
H\"older-continuous.

There are simple obstructions to the smoothness of\/ $h$. Namely,
if\/ $x$ is a periodic point of\/ $f$ with period $p$, that is,
$f^p(x)=x$, then $g^p(h(x))=h(x)$. If\/ $h$ were differentiable,
then
\begin{equation*}
Df^p(x)=\left(Dh(x)\right)^{-1}Dg^p(h(x))Dh(x),
\end{equation*}
\ie $Df^p(x)$ and $Dg^p(h(x))$ are conjugate. We see that every
periodic point carries a modulus of smooth conjugacy.

Suppose that for every periodic point $x$ of period $p$, the
differentials of the return maps $Df^p(x)$ and $Dg^p(h(x))$ are
conjugate. Then we say that the \emph{periodic data} (p.~d.) of\/
$f$ and $g$ coincide.

\begin{q}\label{question}
Suppose that the p.~d. coincide. Is then $h$ differentiable? If it
is, how smooth is it?
\end{q}

\subsection{Positive answers}
\label{positive} We describe situations when the p.~d. form a full
set of moduli of\/ $C^1$ conjugacy.

The only surface  that supports Anosov diffeomorphisms is the
two-di\-men\-si\-on\-al torus. For Anosov diffeomorphisms of\/
$\mathbb{T}^2$, the complete answer to Question \ref{question} was
given by de la Llave, Marco and Moriy\'on.

\begin{theorem}[\cite{LMM},~\cite{L}]
Let $f$ and $g$ be topologically conjugate $C^r$, $r>1$, Anosov
diffeomorphisms of \,$\mathbb{T}^2$ with coinciding p.~d. Then the
conjugacy $h$ is $C^{r-\varepsilon}$, where $\varepsilon>0$ is
arbitrarily small.
\end{theorem}

De la Llave~\cite{L} also observed that the answer is negative for
Anosov diffeomorphisms of\/ $\mathbb T^d$, $d\ge4$.  He constructed
two diffeomorphisms with the same p.~d. which are only H\"older
conjugate. We describe this example in Section~\ref{counterexample}.

In dimension three, the only manifold that supports Anosov
diffeomorphisms is the three-di\-men\-si\-on\-al torus. Moreover,
all Anosov diffeomorphisms of\/ $\mathbb T^3$ are topologically
conjugate to linear automorphisms of\/ $\mathbb T^3$. Nevertheless,
the answer to Question~\ref{question} is not known.

\begin{con}
Let $f$ and $g$ be topologically conjugate $C^r$, $r>1$, Anosov
diffeomorphisms of\, $\mathbb T^3$ with coinciding p.~d. Then the
conjugacy $h$ is at least $C^1$.
\end{con}

There are partial results that support this conjecture.

\begin{theorem}[\cite{GG}] Let $L$ be a hyperbolic automorphism of\/ $\mathbb T^3$ with real eigenvalues. Then there exists a $C^1$-neighborhood $\mathcal U$ of\/ $L$ such that any $f$ and $g$ in $\mathcal U$ having the same p.~d. are $C^{1+\nu}$ conjugate.
\end{theorem}

\begin{theorem}[\cite{KS}] Let $L$ be a hyperbolic automorphism of\/ $\mathbb T^3$ that has one real and two complex eigenvalues. Then any $f$ sufficiently $C^1$ close to $L$ that has the same p.~d. as $L$ is $C^\infty$ conjugate to $L$.
\end{theorem}

In higher dimensions, not much is known. In recent years, much
progress has been made
(see~\cite{Lconformal1,KS2,Lconformal2,F,S,KS}) in the case when the
stable and unstable foliations carry invariant conformal structures.
To ensure existence of these conformal structures one has to at
least assume that every periodic orbit has only one positive and one
negative Lyapunov exponent. This is a very restrictive assumption on
the p.~d.

In contrast to the above, we will study the smooth-conjugacy problem
in the proximity of a hyperbolic automorphism $L\colon\mathbb
T^d\to\mathbb T^d$ with a simple spectrum. Namely, with the
exception of Theorem B, we will always assume that the eigenvalues
of\/ $L$ are real and have different absolute values. For the sake
of notation we assume that the eigenvalues of\/ $L$ are positive.
This is not restrictive.

Let $l$ be the dimension of the stable subspace of\/ $L$ and $k$ be
the dimension of the unstable subspace of\/ $L$, so $k+l=d$.
Consider the $L$-invariant splitting
$$T\mathbb T^d=F_l\oplus F_{l-1}\oplus\ldots \oplus F_1\oplus E_1\oplus E_2\oplus\ldots \oplus E_k$$
along the eigendirections with corresponding eigenvalues
$$\mu_l<\mu_{l-1}<\ldots <\mu_1<1<\lambda_1<\lambda_2<\ldots <\lambda_k.$$
Let $\mathcal U$ be a $C^1$-neighborhood of\/ $L$. The precise
choice of\/ $\mathcal U$ is described in Section~\ref{scheme}. The
theory of partially hyperbolic dynamical systems guarantees that for
any $f$ in $\mathcal U$ the invariant splitting survives (e.~g.
see~\cite{P}); that is,
$$T\mathbb T^d=F_l^f\oplus F_{l-1}^f\oplus\ldots \oplus F_1^f\oplus E_1^f\oplus E_2^f\oplus\ldots \oplus E_k^f.$$
We will see in Section~\ref{scheme} that these
one-di\-men\-si\-on\-al invariant distributions integrate uniquely
to foliations $U_{l}^f$, $U_{l-1}^f$,... $U_1^f$, $V_1^f$,
$V_2^f$,... $V_k^f$.

Given a foliation $\mathcal F$ on $\mathbb T^d$ and an open set $B$,
define
$$
\mathcal F(B)=\bigcup_{y\in B}\mathcal F(y).
$$
We will assume that $f$ has the following property:
\begin{propertyA}
For every $x \in \mathbb T^d$ and every open ball $B \ni x$,
\begin{equation*}
\overline{ U_{l-1}^f(B)}=\overline{ U_{l-2}^f(B)}=\ldots= \overline{
U_{1}^f(B)}= \overline{ V_{1}^f(B)}=\overline{V_{2}^f(B)}=\ldots
=\overline{V_{k-1}^f(B)}=\mathbb T^d.
\end{equation*}
\end{propertyA}
We discuss this property in Section~\ref{minimality}.

\begin{theoremA}
Let $L$ be a hyperbolic automorphism of \,$\mathbb T^d, d\ge 3$,
with a simple real spectrum. Assume that the characteristic
polynomial of\/ $L$ is irreducible over $\mathbb Z$. There exists a
$C^1$-neighborhood $\mathcal U \subset \text{\upshape
Diff}^r(\mathbb T^d)$, $r\ge 2$, of\/ $L$ such that any $f \in
\mathcal U$ satisfying Property $\mathcal A$ and any $g \in \mathcal
U$ with the same p.~d. are $C^{1+\nu}$ conjugate.
\end{theoremA}
\newpage
\begin{remark}\mbox{}
\begin{enumerate}
\item We will see in Section~\ref{minimality} that irreducibility
of the characteristic polynomial of\/ $L$ is necessary for $f$ to
satisfy $\mathcal A$. Formally, we could have omitted the
irreducibility assumption above. Theorem B below shows that the
irreducibility of\/ $L$ is a necessary assumption for the conjugacy
to be $C^1$. We believe that Theorem A holds when $L$ is irreducible
without assuming that $f$ satisfies $\mathcal A$.
\item $\nu$ is a small positive number. It is possible to estimate
$\nu$ from below in terms of the eigenvalues of\/ $L$ and the size
of\/ $\mathcal U$.
\item  Obviously an analogous result holds on finite factors of tori. But we do not know how to prove it on
nilmanifolds. The problem is that for an algebraic Anosov
automorphism of a nilmanifold, various intermediate distributions
may happen to be nonintegrable.
\end{enumerate}
\end{remark}

Theorem A is a generalization of the theorem from~\cite{GG} quoted
above. Our method does not lead to higher regularity of the
conjugacy (see the last section of~\cite{GG} for an explanation).
Nevertheless we conjecture that the situation is the same as in
dimension two.

\begin{con}
In the context of Theorem A one can actually conclude that $f$ and
$g$ are $C^{r-\varepsilon}$ conjugate, where $\varepsilon$ is an
arbitrarily small positive     number.
\end{con}

Simple examples of diffeomorphisms that possess Property $\mathcal
A$ include $f=L$ and any $f\in\mathcal U$ when $\max(k,l)\le 2$ (see
Section~\ref{minimality}). In addition, we construct a $C^1$-open
set of Anosov diffeomorphisms of\/ $\mathbb T^5$ and $\mathbb T^6$
close to $L$ that have Property $\mathcal A$. It seems that this
construction can be extended to any dimension.

We describe this open set when $l=2$ and $k=3$. Given $f\in\mathcal
U$, denote by $D_f^{wu}$ the derivative of\/ $f$ along $V_1^f$.
Choose $f\in\mathcal U$ in such a way that
\begin{equation*}
\forall x\neq x_0,\;\;\;  D_f^{wu}(x)>D_f^{wu}(x_0),
\end{equation*}
where $x_0$ is a fixed point of\/ $f$. Then any diffeomorphism
sufficiently $C^1$ close to $f$ satisfies Property $\mathcal A$.

\subsection{When the coincidence of periodic data is not sufficient}
\label{negative} First let us briefly describe the counterexample of
de la Llave.

Let $L\colon\mathbb T^4\to\mathbb T^4$ be an automorphism of product
type,
\begin{equation}
\label{L} L(x,y)=(Ax, By), \;\;(x,y)\in\mathbb T^2\times\mathbb T^2,
\end{equation}
where $A$ and $B$ are Anosov automorphisms. Let $\lambda$,
$\lambda^{-1}$ be the eigenvalues of\/ $A$ and $\mu$, $\mu^{-1}$ the
eigenvalues of\/ $B$. We assume that $\mu>\lambda>1$. Consider
perturbations of the form
\begin{equation}
\label{tildeL} \tilde L=(Ax+\vec \varphi(y), By),
\end{equation}
where $\vec\varphi\colon\mathbb T^2\to\mathbb R^2$ is a $C^1$-small
$C^r$-function, $r>1$. Obviously the p.~d. of\/ $L$ and $\tilde L$
coincide. We will see in Section~\ref{counterexample} that the
majority of perturbations~(\ref{tildeL}) are only H\"older conjugate
to $L$. The following theorem is a simple generalization of this
counterexample.

\begin{theoremB}
Let  $L\colon\mathbb T^d\to \mathbb T^d$  be a hyperbolic
automorphism such that the characteristic polynomial of\/ $L$
factors over\, $\mathbb Q$. Then there exist
$C^\infty$-diffeomorphisms $\tilde L\colon\mathbb T^d\to\mathbb T^d$
and $\hat L\colon\mathbb T^d\to\mathbb T^d$ arbitrarily\,
$C^1$-close to $L$ with the same p.~d. such that the conjugacy
between $\tilde L$ and $\hat L$ is not Lipschitz.
\end{theoremB}

\begin{remark}
In the majority of cases, one can take $\hat L=L$.  The need to take
$\tilde L$ and $\hat L$ both different from $L$ appears, for
instance, when $L(x,y)=(Ax,Ay)$. It was shown in~\cite{Lconformal1}
that the p.~d. form a complete set of moduli for the
smooth-conjugacy problem to $L$. This is a remarkable phenomenon due
to the invariance of conformal structures on the stable and unstable
foliations. Nevertheless we still have a counterexample if we move a
little bit away from $L$.
\end{remark}

Next we study the smooth conjugacy problem in the neighborhood
of~(\ref{L}) assuming that $\mu>\lambda>1$. We show that the
perturbations~(\ref{tildeL}) exhaust all possibilities. Before
formulating the result precisely let us move to a slightly more
general setting. Let $A$ and $B$ be as in~(\ref{L}) with
$\mu>\lambda>1$. Consider the Anosov diffeomorphism
\begin{equation}
\label{Lgeneral} L(x,y)=(Ax, g(y)), \;\;(x,y)\in\mathbb
T^2\times\mathbb T^2,
\end{equation}
where $g$ is an Anosov diffeomorphism sufficiently $C^1$-close to
$B$, so $L$ can be treated as a partially hyperbolic diffeomorphism
with the automorphism $A$ acting in the central direction. Consider
perturbations of the form
\begin{equation}
\label{tildeLgeneral} \tilde L=(Ax+\vec \varphi(y), g(y)).
\end{equation}
As before, it is obvious that the p.~d. of\/ $L$ and $\tilde L$
coincide. In Section~\ref{moduli} we will see that $L$ and $\tilde
L$ with nonlinear $g$ also provide a counterexample to Question~1.

\begin{theoremC} Given $L$ as in~(\ref{Lgeneral}) with $\mu>\lambda>1$, there exists a $C^1$-neighborhood $\mathcal U \subset \text{\upshape Diff}^r(\mathbb T^4)$, $r\ge 2$, of\/ $L$ such that any $f \in \mathcal U$ that has the same p.~d. as $L$ is $C^{1+\nu}$-conjugate, $\nu > 0$, to a diffeomorphism $\tilde L$ of type~(\ref{tildeLgeneral}).
\end{theoremC}

\subsection{Additional moduli of\/ $C^1$ conjugacy in the neighborhood of the counterexample of de la Llave}
Let $L$ be given by~(\ref{L}) with $\mu>\lambda>1$ and let $\mathcal
U$ be a small $C^1$-neighborhood of\/ $L$. It is useful to think of
diffeomorphisms from $\mathcal U$ as partially hyperbolic
diffeomorphisms with two-dimensional central foliations. Consider
$f, g\in \mathcal U$, $h\circ f=g\circ h$. According to the
celebrated theorem of Hirsch, Pugh and Shub~\cite{HPS}, the
conjugacy $h$ maps the central foliation of\/ $f$ into the central
foliation of\/ $g$.

Assume that the p.~d. of\/ $f$ and $g$ are the same. We will show
that $h$ is $C^{1+\nu}$ along the central foliation. As described
above, it can still happen that $h$ is not a $C^1$-diffeomorphism.
This means that the conjugacy is not differentiable in the direction
transverse to the central foliation. The geometric reason for this
is a mismatch between the strong stable (unstable) foliations of\/
$f$ and $g$ --- the conjugacy $h$ does not map the strong stable
(unstable) foliation of\/ $f$ into the strong stable (unstable)
foliation of\/ $g$.

Motivated by this observation, we now introduce additional moduli
of\/ $C^1$-differen\-tiable conjugacy. Roughly speaking, these
moduli measure the tilt of the strong stable (unstable) leaves when
compared to the model~(\ref{L}).

We define these moduli precisely. Let $W_L^{ss}$, $W_L^{ws}$,
$W_L^{wu}$ and $W_L^{su}$ be the foliations by straight lines along
the eigendirections with eigenvalues $\mu^{-1}$, $\lambda^{-1}$,
$\lambda$ and $\mu$ respectively. For any $f\in\mathcal U$ these
invariant foliations survive. We denote them by $W_f^{ss}$,
$W_f^{ws}$, $W_f^{wu}$ and $W_f^{su}$. We will also write $W_f^s$
and $W_f^u$ for two-dimensional stable and unstable foliations.

Let $h_f$ be the conjugacy to the linear model, $h_f\circ f=L\circ
h_f$. Then
\begin{equation}
\label{h_f_mathches} h_f(W_f^\sigma)=W_L^\sigma,\;\;\; \sigma=s, u,
ws, wu.
\end{equation}

Fix orientations of\/ $W_L^\sigma$, $\sigma=ss, ws, wu, su$. Then
for every $x\in\mathbb T^4$ there exists a unique
orientation-preserving isometry $\EuScript I^\sigma(x)\colon
W_L^\sigma(x)\to\mathbb R$, $\EuScript I^\sigma(x)(x)=0$,
$\sigma=ss, ws, wu, su$.

Define $\Phi_f^u\colon\mathbb T^4\times \mathbb R\to\mathbb R$ by
the formula
$$
\Phi_f^u(x,t)=\EuScript I^{wu}\big(\EuScript
I^{su}(x)^{-1}(t)\big)\big(h_f(W_f^{su}(h_f^{-1}(x))\big)\cap
W_L^{wu}(\EuScript I^{su}(x)^{-1}(t))\big).
$$
The geometric meaning is transparent and illustrated on
Figure~\ref{1_moduli}. The image of the strong unstable manifold
$h_f(W_f^{su}(h_f^{-1}(x)))$ can be viewed as a graph of the
function $\Phi_f^u(x,\cdot)$ over $W_L^{su}(x)$. Analogously we
define $\Phi_f^s\colon\mathbb T^4\times \mathbb R\to\mathbb R$.


\begin{figure}[htbp]
\begin{center}
\begin{picture}(0,0)%
\includegraphics{1_moduli.pstex}%
\end{picture}%
\setlength{\unitlength}{3947sp}%
\begingroup\makeatletter\ifx\SetFigFont\undefined%
\gdef\SetFigFont#1#2#3#4#5{%
\reset@font\fontsize{#1}{#2pt}%
\fontfamily{#3}\fontseries{#4}\fontshape{#5}%
\selectfont}%
\fi\endgroup%
\begin{picture}(6896,2313)(3739,-2518)
\put(7651,-1411){\makebox(0,0)[lb]{\smash{{\SetFigFont{12}{14.4}{\rmdefault}{\mddefault}{\updefault}{\color[rgb]{0,0,0}$\Phi_f^u(x,t)$}%
}}}}
\put(6826,-361){\makebox(0,0)[lb]{\smash{{\SetFigFont{12}{14.4}{\rmdefault}{\mddefault}{\updefault}{\color[rgb]{0,0,0}$W_L^{wu}(\tilde x)$}%
}}}}
\put(4051,-1786){\makebox(0,0)[lb]{\smash{{\SetFigFont{12}{14.4}{\rmdefault}{\mddefault}{\updefault}{\color[rgb]{0,0,0}$x$}%
}}}}
\put(8251,-811){\makebox(0,0)[lb]{\smash{{\SetFigFont{12}{14.4}{\rmdefault}{\mddefault}{\updefault}{\color[rgb]{0,0,0}$h_f(W_f^{su}(h_f^{-1}(x))$}%
}}}}
\put(8026,-2161){\makebox(0,0)[lb]{\smash{{\SetFigFont{12}{14.4}{\rmdefault}{\mddefault}{\updefault}{\color[rgb]{0,0,0}$W_L^{su}(x)$}%
}}}}
\put(6601,-1861){\makebox(0,0)[lb]{\smash{{\SetFigFont{12}{14.4}{\rmdefault}{\mddefault}{\updefault}{\color[rgb]{0,0,0}$\tilde x$}%
}}}}
\end{picture}%

\end{center}
\caption{The geometric meaning of\/ $\Phi_f^u$. Here $\tilde
x=\EuScript I^{su}(x)^{-1}(t)$.}\label{1_moduli}
\end{figure}


Clearly, $\Phi_f^{s/u}$ are moduli of\/ $C^1$-conjugacy. Indeed,
assume that $f$ and $g$ are $C^1$ conjugate by $h$. Then
$h(W_f^{su})=h(W_g^{su})$ and $h(W_f^{ss})=h(W_g^{ss})$ since strong
stable and unstable foliations are characterized by the speed of
convergence which is preserved by $C^1$-conjugacy. Hence
$\Phi_f^{s/u}=\Phi_g^{s/u}$.

It is possible to choose a subfamily of these moduli in an efficient
way. We say that $f$ and $g$ from $\mathcal U$ have the same
\emph{strong unstable foliation moduli} if
\begin{equation}
\label{moduli1} \exists t\neq 0 \;\;\;\text{such that}\;\;\;\forall
x\in\mathbb T^4,\; \;\;\; \Phi_f^u(x,t)=\Phi_g^u(x,t)
\end{equation}
or
\begin{equation}\label{moduli2}
\exists x\in \mathbb T^4\;\;\;\text{and}\;\;\exists
I=(a,b)\subset\mathbb R\;\;\;\text{such that}\;\;\;\forall t\in
I\;\;\; \Phi_f^u(x,t)=\Phi_g^u(x,t).
\end{equation}
The definition of the \emph{strong stable foliation moduli} is
analogous.

\begin{theoremD}
Given $L$ as in~(\ref{L}) with $\mu>\lambda>1$, there exists a
$C^1$-neighborhood $\mathcal U \subset \text{\upshape
Diff}^r(\mathbb T^4)$, $r\ge 2$, of\/ $L$ such that  if\/ $f,
g\in\mathcal U$ have the same p.~d. and the same strong unstable and
strong stable foliation moduli, then $f$ and $g$ are $C^{1+\nu}$
conjugate.
\end{theoremD}
\begin{remark}
In this case $C^{1+\nu}$-differentiability is in fact the optimal
regularity.
\end{remark}

\subsection{Organization of the paper and a remark on terminology}
In Section~2 we describe the counterexample of de la Llave in a way
that allows us to generalize it to Theorem~B in Section~3.
Sections~2 and~3 are independent of the rest of the paper.

In Sections~4 and~5 we discuss Property $\mathcal A$ and construct
examples of diffeomorphisms that satisfy Property $\mathcal A$.
These sections are self-contained.

Section 6 is devoted to the proof of our main result, Theorem A. It
is self-contained but in number of places we refer to~\cite{GG},
where the three-dimensional version of Theorem A was established.

Theorem C is proved in Section 7. It is independent of the rest of
the paper with the exception of a reference to
Proposition~\ref{central_smoothness}.

The proof of Theorem D appears in Section 8 and relies on some
technical results from~\cite{GG}.

Throughout the paper we will prove that various maps are
$C^{1+\nu}$-differentiable. This should be understood in the usual
way: the map is $C^1$-differentiable and the derivative is
H\"older-continuous with some positive exponent $\nu$. The number
$\nu$ is not the same in different statements.

When we say that a map is $C^{1+\nu}$-differentiable along a
foliation $\mathcal F$, we mean that restrictions of the map to the
leaves of\/ $\mathcal F$ are $C^{1+\nu}$-differentiable and the
derivative is a H\"older-continuous function on the manifold, not
only on the leaf.

\subsection {Acknowledgements}
The author is grateful to Anatole Katok for numerous discussions,
advice, and for introducing him to this problem. Many thanks go to
Misha Guysinsky and Dmitry Scheglov for useful discussions. The
author also would like to thank the referees for providing helpful
suggestions and pointing out errors. It was pointed out that tubular
minimality of a foliation is equivalent to its transitivity. All
these suggestions led to a better exposition.

\section{The counterexample on $\mathbb T^4$}
\label{counterexample}

Here we describe the example of de la Llave of two Anosov
diffeomorphisms of\/ $\mathbb T^4$ with the same p.~d. that are only
H\"older conjugate. Understanding of the example is important for
the proof of Theorem B.

Recall that we start with an automorphism $L\colon\mathbb
T^4\to\mathbb T^4$ such that
$$L(x,y)=(Ax, By), \;\;(x,y)\in\mathbb T^2\times\mathbb T^2,$$
where $A$ and $B$ are Anosov automorphisms, $Av=\lambda v$, $A\tilde
v=\lambda^{-1}\tilde v$, $Bu=\mu u$, $B\tilde u=\mu^{-1}\tilde u$.
We assume that $\mu\ge\lambda>1$.

To simplify computations we consider a special perturbation of the
form
$$\tilde L=(Ax+\varphi(y)v, By).$$
We look for the conjugacy $h$ of the form
\begin{equation}
\label{conj} h(x, y)=(x+\psi(y)v, y).
\end{equation}

The conjugacy equation $h\circ\tilde L=L\circ h$ transforms into a
cohomological equation on $\psi$
\begin{equation}
\label{psi} \varphi(y)+\psi(By)=\lambda \psi(y).
\end{equation}
Let us solve for $\psi$ using the recurrent formula
$$\psi(y)=\lambda^{-1}\varphi(y)+\lambda^{-1}\psi(By).$$
We get a continuous solution to~(\ref{psi}),
\begin{equation}
\label{psisolution} \psi(y)=\lambda^{-1}\sum_{k\ge 0}
\lambda^{-k}\varphi(B^k y).
\end{equation}
Hence the conjugacy is indeed given by the formula~(\ref{conj}).

In the following proposition we denote by the subscript $u$ the
partial derivative in the direction of\/ $u$.

\begin{proposition}
\label{proposition1} Assume that $\mu>\lambda>1$. Then function
$\psi$ is Lipschitz in the direction of\/ $u$ if and only if
\begin{equation}
\label{condition} \sum_{k\in\mathbb Z}
\left(\frac\mu\lambda\right)^k \varphi_u(B^k y)=0,
\end{equation}
\ie the series on the left converges in the sense of distribution
convergence and the limit is equal to zero.
\end{proposition}

\begin{proof}
First assume~(\ref{condition}). Let us consider the
series~(\ref{psisolution}) as a series of distributions that
converge to $\psi$. Then, as a distribution, $\psi_u$ is obtained by
differentiating~(\ref{psisolution}) termwise:
\begin{equation}
\label{psiu1} \psi_u=\lambda^{-1}\sum_{k\ge 0} \lambda^{-k}\mu^k
\varphi_u(B^k).
\end{equation}
Applying~(\ref{condition}), we get
$$\psi_u=\lambda^{-1}\sum_{k<0} \lambda^{-k}\mu^k \varphi_u(B^k).$$
Since $\mu>\lambda$ the above series converges and the distribution
is regular. Hence $\psi$ is differentiable in the direction of\/
$u$.

Now assume that $\psi$ is $u$-Lipschitz. By
differentiating~(\ref{psi}), we get a cohomological equation on
$\psi_u$,
$$\varphi_u(x)+\mu\psi_u(By)=\lambda\psi_u(y),$$
that is satisfied on a $B$-invariant set of full measure. We solve
it using the recurrent formula
$$\psi_u(y)=-\frac1\mu \varphi_u(B^{-1}y)+\frac\lambda\mu\psi_u(B^{-1}y).$$
Hence
\begin{equation}
\label{psiu2}
\psi_u=\lambda^{-1}\sum_{k<0}\lambda^{-k}\mu^k\varphi_u(B^k).
\end{equation}
On the other hand we know that as a distribution $\psi_u$ is given
by~(\ref{psiu1}). Combining~(\ref{psiu1}) and~(\ref{psiu2}) we get
the desired equality~(\ref{condition}).
\end{proof}

If\/ $\mu=\lambda$ then the argument above works only in one
direction. We will see that in this case $L$ and $\tilde L$ do not
provide a counterexample since the p.~d. are different.

\begin{proposition}
\label{proposition2} Assume that $\mu=\lambda$.
Then~(\ref{condition}) is a necessary assumption for $\psi$ to be
Lipschitz in the direction of\/ $u$.
\end{proposition}

\begin{proof}
As in the proof of Proposition 1, viewed as distribution, $\psi_u$
is given by
\begin{equation}
\label{2psiu1} \psi_u=\lambda^{-1}\sum_{k\ge 0}  \varphi_u(B^k).
\end{equation}
Assume that $\psi$ is $u$-Lipschitz. Then, analogously
to~(\ref{psiu2}), we get
\begin{equation}
\label{2psiu2} \psi_u=\lambda^{-1}\sum_{-N\le
k<0}\varphi_u(B^k)+\psi(B^N).
\end{equation}
Note that in the sense of distributions, $\psi(B^N)\to 0$ as
$N\to\infty$ since $B$ is mixing. Hence, as a distibution, $\psi_u$
is given by
\begin{equation}
\label{2psiu3} \psi_u=\lambda^{-1}\sum_{k<0}\varphi_u(B^k).
\end{equation}
Combining~(\ref{2psiu1}) and~(\ref{2psiu3}), we
get~(\ref{condition}).
\end{proof}

By rewriting condition~(\ref{condition}) in terms of Fourier
coefficients of\/ $\varphi$, one can see that it is an infinite
codimension condition.  Moreover, one can easily construct functions
that do not satisfy~(\ref{condition}); one only needs to make sure
that some Fourier coefficients of the sum~(\ref{condition}) are
nonzero. For instance, for any $\varepsilon>0$ and positive integer
$p$, the function
\begin{equation}
\label{phiexample} \varphi(y)=\varphi(y_1, y_2)=\varepsilon\sin(p
\pi y_1)
\end{equation}
works. Thus the corresponding $\tilde L$ is not $C^1$-conjugate to
$L$. Note that $\tilde L$ may be chosen arbitrarily close to $L$.

\begin{remark}\mbox{}
\begin{enumerate}
\item Perturbations of the general type~(\ref{tildeL}) can be treated analogously by decomposing $\vec\phi=\phi_1 v+\phi_2 \tilde  v$.
\item The assumption $\mu\ge\lambda>1$ is crucial in this construction.
\item By choosing appropriate $\lambda$ and $\mu$, one can get any desired regularity of the conjugacy (see~\cite{L} for details). For example, if\/ $\mu^2>\lambda>\mu>1$, the conjugacy is $C^1$ but not $C^2$.
\end{enumerate}
\end{remark}

From now on let us assume that $\mu=\lambda$. As we remarked in the
introduction, $L$ and $\tilde L$ do not provide a counterexample.
Indeed, the derivative of\/ $\tilde L$ in the basis $\{v, u, \tilde
v,\tilde u\}$ is
$$
\begin{pmatrix}
\lambda & \varphi_u & 0 & \varphi_{\tilde u} \\
0 & \lambda & 0 & 0\\
0 & 0 &  \lambda^{-1} & 0\\
0 & 0 & 0 &  \lambda^{-1}
\end{pmatrix}.
$$
Let $x$ be a periodic point, $\tilde L^p(x)=x$. Then the derivative
of the return map at $x$ is
\begin{equation}
\label{Jordan_block}
\begin{pmatrix}
\lambda^p & \lambda^{p-1}\sum_{y\in\mathcal O(x)}\varphi_u(y) & 0 & * \\
0 & \lambda^p & 0 & 0\\
0 & 0 &  \lambda^{-p} & 0\\
0 & 0 & 0 &  \lambda^{-p}
\end{pmatrix}.
\end{equation}
We see that it is likely to have a Jordan block while $L$ is
diagonalizable. Hence $L$ and $\tilde L$ have different p.~d.

It is still easy to construct a counterexample in a neighborhood
of\/ $L$. Let
$$
\hat L=(Ax+\xi(y)v, By)
$$
and let
$$
h(x, y)=(x+\psi(y)v, y)
$$
be the conjugacy between $\tilde L$ and $\hat L$
\begin{proposition}
The condition
\begin{equation*}
\sum_{k\in\mathbb Z} (\xi-\varphi)_u(B^k y)=0,
\end{equation*}
is necessary for $\phi$ to be Lipschitz in the direction of $u$.
 \label{propos3}
\end{proposition}
The proof of Proposition~\ref{propos3} is exactly the same as the
one of Proposition~\ref{proposition2}.

Now take $\varphi$ that does not satisfy~(\ref{condition}) as before
and take $\xi=2\varphi$. Then obviously the condition of
Proposition~\ref{propos3} is not satisfied. Hence $h$ is not
Lipschitz. By looking at~(\ref{Jordan_block}) it is obvious that our
choice of\/ $\xi$ guarantees that the Jordan normal forms of the
derivatives of the return maps at periodic points of\/ $\tilde L$
and $\hat L$ are the same.

\begin{remark}
Due to the special choice of\/ $\xi$ it was easy to ensure that the
p.~d. of\/ $\tilde L$ and $\hat L$ are the same. We could have taken
a different and somewhat more general approach. It is possible to
show that for many choices of\/ $\varphi$, the sum that appears over
the diagonal in~(\ref{Jordan_block}) is nonzero for every periodic
point $x$. All of the corresponding diffeomorphisms will have the
same p.~d. with a Jordan block at every periodic point.
\end{remark}

\section{Proof of Theorem B}

Here we consider $L\colon\mathbb T^d\to\mathbb T^d$ with a reducible
characteristic polynomial. We show how to construct $\tilde L$  and
$\hat L$ with the same p.~d. which are not Lipschitz conjugate.

Assume that all real eigenvalues of\/ $L$ are positive. Otherwise we
may consider $L^2$. Let $M\colon\mathbb R^d\to \mathbb R^d$ be the
lift of\/ $L$ and let $\{e_1, e_2, \ldots e_d\}$ be the canonical
basis, so $\mathbb T^d=\mathbb R^d/\text{span}_\mathbb Z\{e_1,
e_2,\ldots e_d\}$.

It is well known that the characteristic polynomial of\/ $M$ factors
over $\mathbb Z$ into the product of polynomials irreducible over
$\mathbb Q$:
$$P(x)=P_1(x)P_2(x)\ldots P_r(x),\;\; r\ge 2.$$
Let $\lambda$ be the eigenvalue of\/ $M$ with the smallest absolute
value which is greater than one. Without loss of generality we
assume that $P_1(\lambda)=0$.

Let $V_i$ be the invariant subspace that corresponds to the roots
of\/ $P_i$. Then $\dim V_i=\deg P_i$ and it is easy to show that
$$V_i=\text{Ker}(P_i(M)).$$

Matrices of\/ $P_i(M)$ have integer entries. Hence there is a basis
$\{\tilde e_1,\tilde e_2,\ldots \tilde e_d\}$, $\tilde e_i\in
\text{span}_{\mathbb Z}\{e_1, e_2,\ldots e_d\}, i=1,\ldots d$, such
that the matrix of\/ $M$ in this basis has integer entries and is of
block diagonal form with blocks corresponding to the invariant
subspaces $V_i, i=1,\ldots r$.

We consider projection of\/ $M$ to $\tilde {\mathbb T}^d=\mathbb
R^d/\text{span}_{\mathbb Z}\{\tilde e_1,\tilde e_2,\ldots \tilde
e_d\}$. Denote by $N$ the induced map on $\tilde {\mathbb T}^d$. We
have the following commutative diagram, where $\pi$ is a
finite-to-one projection.

$$
\begin{CD}
\mathbb R^d @>M>> \mathbb R^d\\
@VVV @VVV\\
\tilde {\mathbb T}^d @>N>> \tilde {\mathbb T}^d\\
@V\pi VV @V\pi VV\\
\mathbb T^d @>L>> \mathbb T^d
\end{CD}
\medskip
$$
Notice that $N$ has the form $N(x,y)=(Ax, By)$, $(x, y)\in \mathbb
T^{\deg P_1}\times\mathbb T^{d-\deg P_1}$. Let $\mu$ be an
eigenvalue of\/ $B$. By construction, $\lambda$,
$|\lambda|\le|\mu|$, is an eigenvalue of\/ $A$.

With certain care, the  construction of Section~\ref{counterexample}
can be applied to $N$. We have to distinguish the following cases:
\begin{enumerate}
\item $\lambda$ and $\mu$ are real.
\item $\lambda$ is real and $\mu$ is complex.
\item $\lambda$ is complex and $\mu$ is real.
\item $\lambda$ and $\mu$ are complex.
\end{enumerate}

Assume that $|\lambda|<|\mu|$. Then we take $\hat L=L$.

In the first case the construction of Section~\ref{counterexample}
applies straightforwardly. We use a function of the
type~(\ref{phiexample}) to produce  $\tilde N$. Now we only need to
make sure that $\tilde N$ can be projected to a map $\tilde
L\colon\mathbb T^d\to\mathbb T^d$. Since $\pi$ is a finite-to-one
covering map this can be achieved by choosing suitable $p$
in~(\ref{phiexample}).

Other cases require heavier calculations but follow the same scheme
as Proposition~1. We outline the construction in case 4, which can
appear, for instance, if\/ $A$ and $B$ are hyperbolic automorphisms
of four-dimensional tori without real eigenvalues.

Let $V_A=\text{span}\{v_1, v_2\}$ be the two-dimensional
$A$-invariant subspace corresponding to $\lambda$ and
$V_B=\text{span}\{u_1, u_2\}$ be the two-dimensional $B$-invariant
subspace corresponding to $\mu$. Then $A$ acts on $V_A$ by
multiplication by $|\lambda| R_A$ and $B$ acts on $V_B$ by
multiplication by $|\mu|R_B$, where $R_A$ and $R_B$ are rotation
matrices expressed in the bases $\{v_1, v_2\}$ and $\{u_1, u_2\}$,
respectively.

We are following the construction from the previous section. Let
$$
\tilde N(x,y)=(Ax+\vec\varphi(y)\vec v,
By)\stackrel{\mathrm{def}}{=}(Ax+\varphi_1(y)v_1+\varphi_2(y)v_2,
By).
$$
Then we look for a conjugacy of the form
$$
h(x,y)=(x+\vec\psi(y)\vec v,
y)\stackrel{\mathrm{def}}{=}(x+\psi_1(y)v_1+\psi_2(y)v_2, y).
$$
The conjugacy equation $h\circ\tilde N=N\circ h$ transforms into
\begin{equation}
\label{coh_eq} \vec\varphi(y)\vec v+\vec\psi(By)\vec
v=|\lambda|R_A\vec\psi(y).
\end{equation}
Solving for $\vec\psi$ gives
$$
\vec\psi(y)=\sum_{k\ge0}|\lambda|^{-k-1}R_A^{-k-1}\vec\varphi(B^ky),
$$
which we would like to differentiate along the directions $u_1$ and
$u_2$. We use the formula
$$
\vec\varphi(By)_{\vec u}=
\begin{pmatrix}
\varphi_1(By)_{u_1} & \varphi_1(By)_{u_2} \\
\varphi_2(By)_{u_1} & \varphi_2(By)_{u_2}
\end{pmatrix}
=|\mu|
\begin{pmatrix}
(\varphi_1)_{u_1} & (\varphi_1)_{u_2}\\
(\varphi_2)_{u_1} & (\varphi_2)_{u_2}
\end{pmatrix}
(By)R_B=\vec\varphi_{\vec u}(By)R_B
$$
to get that, as a distribution,
$$
\vec\psi_{\vec u}=\sum_{k\ge
0}|\lambda|^{-k-1}|\mu|^kR_A^{-k-1}\vec\varphi_{\vec u}(B^k)R_B^k.
$$

Now we assume that $\vec\psi$ is Lipschitz and we
differentiate~(\ref{coh_eq}) along the directions $u_1$ and $u_2$:
$$
\vec\varphi_{\vec u}(y)+|\mu|\vec\psi_{\vec
u}(By)R_B=|\lambda|R_A\vec\psi_{\vec u}(y).
$$
Hence, by the recurrent formula,
$$
\vec\psi_{\vec u}=\sum_{k<
0}|\lambda|^{-k-1}|\mu|^kR_A^{-k-1}\vec\varphi_{\vec u}(B^k)R_B^k.
$$
Combining the expressions for $\vec\psi_{\vec u}$, we get
$$
\sum_{k\in\mathbb Z}|\lambda|^{-k}|\mu|^kR_A^{-k}\vec\varphi_{\vec
u}(B^k)R_B^k=0.
$$
Using Fourier decompositions, one can find functions $\vec\varphi$
that do not satisfy the condition above. One also needs to make sure
that the choice of\/ $\vec\varphi$ allows one to project $\tilde N$
down to $\tilde L$. We omit this analysis since it is routine.

This is a contradiction and therefore $\vec\psi$ (and hence $h$) is
not Lipschitz.

If\/ $|\lambda|=|\mu|$ but $\lambda\neq\mu$, then the scheme above
still works. Obviously, extra Jordan blocks do not appear in the
normal forms at periodic points of\/ $\tilde L$.

Finally, the case $\lambda=\mu$ must be treated separately. We use
the same trick as in Section~2 to find $\tilde L$ and $\hat L$ with
the same p.~d. that are only H\"older conjugate. This trick also
works well in the case of complex eigenvalues; we omit the details.

\section{On the Property $\mathcal A$}
\subsection{Transitivity versus minimality}
\label{minimality} Here we discuss Property $\mathcal A$. Let
$\mathcal F$ be a foliation of a compact manifold $M$. As usually
$\mathcal F(x)$ stands for the leaf of\/ $\mathcal F$ that contains
$x$ and $\mathcal F(x, R)$ stands for the ball of radius $R$
centered at $x$ inside of\/ $\mathcal F(x)$.

\begin{definition}
The foliation $\mathcal F$ is called \emph{minimal} if every leaf
of\/ $\mathcal F$ is dense in $M$.
\end{definition}
\begin{definition}
The foliation $\mathcal F$ is called \emph{transitive} if there
exists a leaf of\/ $\mathcal F$ that is dense in $M$.
\end{definition}
\begin{definition}
The foliation $\mathcal F$ is called \emph{tubularly minimal} if for
every $x$ and every open ball $B\ni x$,
$$
\overline{\bigcup_{y\in B} \mathcal F(y)}=M.
$$
\end{definition}
Property $\mathcal A$ simply requires the foliations $U_{l-1}^f$,
$U_{l-2}^f$,... $U_1^f$, $V_1^f$, $V_2^f$,... $V_{k-1}^f$ to be
tubularly minimal. We introduce the following related property:
\begin{propertyB}
The foliations $U_{l-1}^f$, $U_{l-2}^f$,... $U_1^f$, $V_1^f$,
$V_2^f$,... $V_{k-1}^f$ are minimal.
\end{propertyB}

\begin{proposition}
\label{tranz} The foliation $\mathcal F$ is transitive if and only
if it is tubularly minimal.\footnote{We would like to thank the
referee for pointing out this fact.}
\end{proposition}
\begin{proof}
Transitivity obviously implies tubular minimality.

Assume that $\mathcal F$ is tubularly minimal. Let $\{B_n, n\ge 1\}$
be a countable basis for the topology of\/ $M$. By the definition of
tubular minimality, the sets $\mathcal F(B_n)$ are open and dense in
$M$. Hence by the Baire Category Theorem, the set
$$
B=\bigcap_{n\ge 1}\mathcal F(B_n)
$$
is nonempty and for every $x\in B$ the leaf\/ $\mathcal F(x)$ is
dense in $M$.
\end{proof}

\begin{remark}
We define Property $\mathcal A$ in terms of tubular minimality
rather than transitivity since we need denseness of the tubes for
the proof of Theorem A.
\end{remark}

A priori, transitivity is weaker than minimality. Hence, a priori,
Property $\mathcal A$ is weaker than Property $\mathcal A'$.

If, in Theorem A, we had required $f$ to satisfy Property $\mathcal
A'$ instead of Property $\mathcal A$, then the induction procedure
that we use (the first induction step) is much simpler.  The proof
of this step, assuming only Property $\mathcal A$, requires a much
more lengthy and delicate argument. It is not clear to us what the
relationship is between Properties $\mathcal A$ and $\mathcal A'$;
they may be equivalent. Thus, we will first provide a proof of
Theorem A assuming that $f$ has Property $\mathcal A'$, then we will
present a separate proof of this first induction step (namely
Lemma~\ref{step1}) that uses only Property $\mathcal A$.

Minimality of a foliation can be characterized similarly to tubular
transitivity.
\begin{proposition}
\label{minimality_proposition} The foliation $\mathcal F$ is minimal
if and only if for every $x$ and every open ball $B\ni x$,
$$
{\bigcup_{y\in B} \mathcal F(y)}=M.
$$
\end{proposition}
The proof is simple, so we omit it. As a corollary  the foliation
$\mathcal F$ is minimal if and only if for every $x$ and every open
ball $B\ni x$, there exists a number $R$ such that
\begin{equation}
\label{finite_A} {\bigcup_{y\in B} \mathcal F(y, R)}=M.
\end{equation}
This is the property which we will actually use in the proof of the
induction step~1.

\subsection{Examples of diffeomorphisms that satisfy Property $\mathcal A$}
\begin{proposition}
\label{denseness} Assume that $L$ is irreducible. Then the
foliations $U_j^L$, $V_i^L$, $j=1\ldots l$, $i=1\ldots k$, are
minimal.
\end{proposition}
\begin{proof}
Denote by $\mathcal F$ one of the foliations under consideration.
Since $\mathcal F$ is a foliation by straight lines, the closure of
a leaf\/ $\mathcal F(x)$ is a subtorus of\/ $\mathbb T^d$. This
subtorus lifts to a rational invariant subspace of\/ $\mathbb R^d$.
The invariant subspace corresponds to a rational factor of the
characteristic polynomial of\/ $L$, but we assumed that it is
irreducible over $\mathbb Q$. Hence the invariant subspace is the
whole of\/ $\mathbb R^d$ and the subtorus is the whole $\mathbb
T^d$.
\end{proof}
So we can see that the conclusion of Theorem A holds at least for
$f=L$.

We will see in Section~\ref{scheme} that for any $f\in\mathcal U$,
the foliations $U_1^f$ and $V_1^f$ are minimal. Hence the conclusion
of Theorem A holds for any $f\in\mathcal U$ if\/ $\max(k,l)\le 2$.

It is easy to construct $f\neq L$ that satisfies $\mathcal A$ when
$k=3$ and $l=2$ since we only have to worry about the foliation
$V_2^f$. We let $f=s\circ L$ where $s$ is any small shift along
$V_2^f$. Clearly $V_2^f=V_2^L$ and hence $f$ satisfies $\mathcal A$.

Questions about robust minimality of the foliations $U_{l-1}^f$,
$U_{l-2}^f$,... $U_1^f$, $V_1^f$, $V_2^f$,... $V_{k-1}^f$ arise
naturally. Robust minimality of strong stable and strong unstable
foliations of partially hyperbolic systems has received some
attention in the literature due to its intimate connection with
robust transitivity; see~\cite{Ma} and the more recent
papers~\cite{BDU,PujalS}, where robust minimality of the \emph{full}
expanding foliation is established under some assumptions. We do not
have this luxury in our setting: the expanding foliations that we
are intrested in subfoliate the full unstable foliation. A
representative problem here is the following.
\begin{q} Let $L\colon\mathbb T^3\to \mathbb T^3$ be a hyperbolic
linear automorphism with real spectrum
$\lambda_1<1<\lambda_2<\lambda_3$. Consider the one-dimensional
strong unstable foliation. Is it true that this foliation is
robustly minimal? In other words, is it true that for any $f$
sufficiently $C^1$-close to $L$ the strong unstable foliation of\/
$f$ is minimal?
\end{q}

In addition to the simple examples above, in the next section we
construct a $C^1$-open set of diffeomorphisms that possess Property
$\mathcal A$. The following statement can be obtained by applying
the construction and the arguments of the next section in the
context of Question 2.

\begin{proposition}
Let $L$ be as in Question 2. Then there exists a $C^1$-open set
$\mathcal U$ $C^1$-close to $L$ such that for every $f\in\mathcal U$
the strong unstable foliation of\/ $f$ is transitive.
\end{proposition}

\section{An example of an open set of diffeomorphisms with Property $\mathcal A$}
\label{example}

Let $L\colon\mathbb T^5\to \mathbb T^5$ be a hyperbolic automorphism
as in Theorem A, $l=2, k=3$, and let $\mathcal U$ be a
$C^1$-neighborhood of\/ $L$ chosen as in Section~\ref{scheme}.

Recall that $D_f^{wu}$ stands for the derivative of\/ $f\in\mathcal
U$ along $V_1^f$. Choose $f\in\mathcal U$ in such a way that
\begin{equation}
\label{derivative_condition} \forall x\neq x_0\;\;\;
D_f^{wu}(x)>D_f^{wu}(x_0),
\end{equation}
where $x_0$ is a fixed point of\/ $f$.
\begin{proposition}
\label{example_prop} There exists a $C^1$-neighborhood
$\tilde{\mathcal U}$ of\/ $f$ such that any diffeomorphism $g\in
\tilde{\mathcal  U}$ has Property $\mathcal A$.
\end{proposition}

\begin{remark}
A similar example can be constructed on $\mathbb T^6$ with $l=3$,
$k=3$. We only need to do the trick described below for both the
stable and unstable manifolds of the fixed point $x_0$.
\end{remark}

Before proving the proposition let us briefly explain the idea
behind the proof. We know that $U_1^g$ and $V_1^g$ are minimal.
Hence we only need to show that the foliation $V_2^g$ is tubularly
minimal, \ie for every $x\in \mathbb T^5$  and every open ball $B\ni
x$
\begin{equation}
\label{to_prove} \overline{\bigcup_{y\in B} V_2^g(y)}=\mathbb T^5.
\end{equation}
To illustrate the idea we take $g=f$ and $x=x_0$. We work on the
universal cover $\mathbb R^5$ with lifted foliations. Let
\begin{equation}
\label{tube} \EuScript T\stackrel{\mathrm{def}}{=} \bigcup_{y\in
B}V_2^f(y)\subset\mathbb R^5,
\end{equation}
which is an open tube.

We show that $\EuScript T$ contains arbitrarily long connected
pieces of the leaves of\/ $V_1^f$ as shown on Figure~\ref{2_tube}.
It will then follow that $\EuScript T$ is dense in $\mathbb T^5$.
Indeed, the foliation $V_1^f$ is not just minimal but uniformly
minimal: for any $\varepsilon>0$ there exists $R>0$ such that
$\forall z\in\mathbb T^5$ $V_1^f(z, R)$ is $\varepsilon$-dense in
$\mathbb T^5$. This property follows from the fact that $V_1^f$ is
conjugate to the linear foliation $V_1^L$.

Pick $y_0\in B\cap V_1^f(x_0)$ close to $x_0$. Let $x\in V_2^f(x_0)$
be a point far away in the tube $\EuScript T$ and $y=V_1^f(x)\cap
V_2^f(y_0)$. To show that $\EuScript T$ contains arbitrarily long
pieces of leaves of\/ $V_1^f$ we prove that $d_1^f(x,y)$ (recall
that $d_i^f$ is the Riemannian distance along $V_i^f$) is an
unbounded function of\/ $x$.


\begin{figure}[htbp]
\begin{center}

\begin{picture}(0,0)%
\includegraphics{2_tube.pstex}%
\end{picture}%
\setlength{\unitlength}{3947sp}%
\begingroup\makeatletter\ifx\SetFigFont\undefined%
\gdef\SetFigFont#1#2#3#4#5{%
\reset@font\fontsize{#1}{#2pt}%
\fontfamily{#3}\fontseries{#4}\fontshape{#5}%
\selectfont}%
\fi\endgroup%
\begin{picture}(6175,2755)(2764,-3260)
\put(3226,-3211){\makebox(0,0)[lb]{\smash{{\SetFigFont{12}{14.4}{\rmdefault}{\mddefault}{\updefault}{\color[rgb]{0,0,0}$B$}%
}}}}
\put(2851,-2641){\makebox(0,0)[lb]{\smash{{\SetFigFont{12}{14.4}{\rmdefault}{\mddefault}{\updefault}{\color[rgb]{0,0,0}$y_0$}%
}}}}
\put(3301,-2911){\makebox(0,0)[lb]{\smash{{\SetFigFont{12}{14.4}{\rmdefault}{\mddefault}{\updefault}{\color[rgb]{0,0,0}$x_0$}%
}}}}
\put(7126,-811){\makebox(0,0)[lb]{\smash{{\SetFigFont{12}{14.4}{\rmdefault}{\mddefault}{\updefault}{\color[rgb]{0,0,0}$x$}%
}}}}
\put(6451,-1861){\makebox(0,0)[lb]{\smash{{\SetFigFont{12}{14.4}{\rmdefault}{\mddefault}{\updefault}{\color[rgb]{0,0,0}$\EuScript T$}%
}}}}
\put(3901,-661){\makebox(0,0)[lb]{\smash{{\SetFigFont{12}{14.4}{\rmdefault}{\mddefault}{\updefault}{\color[rgb]{0,0,0}$V_1^f(x)$}%
}}}}
\put(4726,-2761){\makebox(0,0)[lb]{\smash{{\SetFigFont{12}{14.4}{\rmdefault}{\mddefault}{\updefault}{\color[rgb]{0,0,0}$V_1^f(x_0)$}%
}}}}
\put(5761,-1636){\makebox(0,0)[lb]{\smash{{\SetFigFont{12}{14.4}{\rmdefault}{\mddefault}{\updefault}{\color[rgb]{0,0,0}$V_2^f(x_0)$}%
}}}}
\put(4789,-1411){\makebox(0,0)[lb]{\smash{{\SetFigFont{12}{14.4}{\rmdefault}{\mddefault}{\updefault}{\color[rgb]{0,0,0}$V_2^f(y_0)$}%
}}}}
\put(5176,-781){\makebox(0,0)[lb]{\smash{{\SetFigFont{12}{14.4}{\rmdefault}{\mddefault}{\updefault}{\color[rgb]{0,0,0}$y$}%
}}}}
\end{picture}%

\end{center}
\caption{The tube $\EuScript T$ contains arbitrarily long pieces of
leaves of\/ $V_1^f$.}\label{2_tube}
\end{figure}

We make use of the affine structure on $V_1^f$. We refer
to~\cite{GG} for the definition of affine distance-like function
$\tilde d_1$.  Recall the following crucial properties of\/ $\tilde
d_1$:
\begin{itemize}
\item[(D1)] $\tilde d_1(x,y)=d_1^f(x,y)+o(d_1^f(x,y))$,
\item[(D2)] $\tilde d_1(f(x),f(y))=D_f^{wu}(x)\tilde d_1(x,y)$,
\item[(D3)] $\forall K>0 \;\exists C>0$ such that
\begin{equation*}
\frac1C\tilde d_1(x,y)\le d_1^f(x,y)\le C\tilde d_1(x,y)
\end{equation*}
whenever $d_1(x, y)<K$.
\end{itemize}

Using Property (D3), we can see that it is enough to show that
$\tilde d_1(x,y)$ is unbounded.

Given $x$ as above, pick $N$ large enough such that the ratio
$$\tilde d_1(f^{-N}(x),f^{-N}(y))/\tilde d_1(x_0,f^{-N}(y_0))$$ is
close to $1$ as shown in Figure~\ref{3_tube_expanding}. This is
possible since $V_2^f$ contracts exponentially faster than $V_1^f$
under the action of\/ $f^{-1}$. It is not hard to see that, given a
large number $n$, we can pick $x$ (and $N$ correspondingly) far
enough from $x_0$ such that at least $n$ points from the orbit $\{x,
f^{-1}(x),\ldots f^{-N}(x)\}$ lie outside of\/ $B$. For such a point
$z=f^{-i}(x)$ that is not in $B$,
$$
D_f^{wu}(z)\ge D_f^{wu}(x_0)+\delta,
$$
where $\delta>0$ depends only on the size of\/ $B$.


\begin{figure}[htbp]
\begin{center}

\begin{picture}(0,0)%
\includegraphics{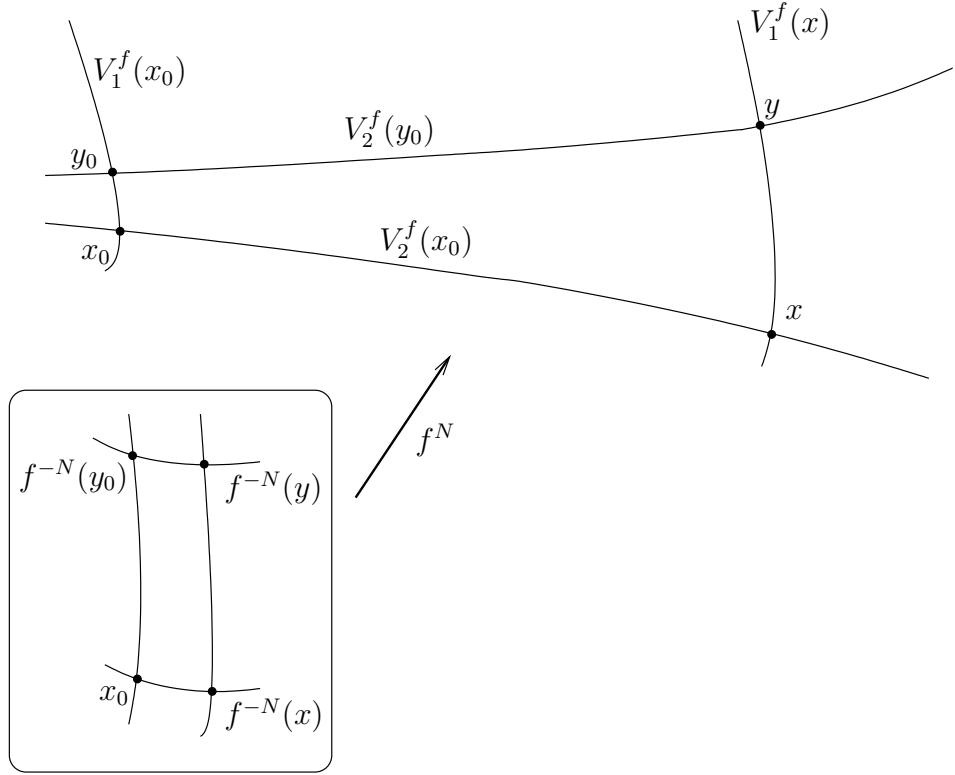}%
\end{picture}%
\setlength{\unitlength}{3947sp}%
\begingroup\makeatletter\ifx\SetFigFont\undefined%
\gdef\SetFigFont#1#2#3#4#5{%
\reset@font\fontsize{#1}{#2pt}%
\fontfamily{#3}\fontseries{#4}\fontshape{#5}%
\selectfont}%
\fi\endgroup%
\begin{picture}(5949,4818)(1189,-4348)
\put(1585,-511){\makebox(0,0)[lb]{\smash{{\SetFigFont{12}{14.4}{\rmdefault}{\mddefault}{\updefault}{\color[rgb]{0,0,0}$y_0$}%
}}}}
\put(1261,-2536){\makebox(0,0)[lb]{\smash{{\SetFigFont{12}{14.4}{\rmdefault}{\mddefault}{\updefault}{\color[rgb]{0,0,0}$f^{-N}(y_0)$}%
}}}}
\put(2551,-2611){\makebox(0,0)[lb]{\smash{{\SetFigFont{12}{14.4}{\rmdefault}{\mddefault}{\updefault}{\color[rgb]{0,0,0}$f^{-N}(y)$}%
}}}}
\put(1765,-3886){\makebox(0,0)[lb]{\smash{{\SetFigFont{12}{14.4}{\rmdefault}{\mddefault}{\updefault}{\color[rgb]{0,0,0}$x_0$}%
}}}}
\put(2551,-4036){\makebox(0,0)[lb]{\smash{{\SetFigFont{12}{14.4}{\rmdefault}{\mddefault}{\updefault}{\color[rgb]{0,0,0}$f^{-N}(x)$}%
}}}}
\put(3751,-2311){\makebox(0,0)[lb]{\smash{{\SetFigFont{12}{14.4}{\rmdefault}{\mddefault}{\updefault}{\color[rgb]{0,0,0}$f^N$}%
}}}}
\put(1651,-1111){\makebox(0,0)[lb]{\smash{{\SetFigFont{12}{14.4}{\rmdefault}{\mddefault}{\updefault}{\color[rgb]{0,0,0}$x_0$}%
}}}}
\put(1726, 14){\makebox(0,0)[lb]{\smash{{\SetFigFont{12}{14.4}{\rmdefault}{\mddefault}{\updefault}{\color[rgb]{0,0,0}$V_1^f(x_0)$}%
}}}}
\put(5851,314){\makebox(0,0)[lb]{\smash{{\SetFigFont{12}{14.4}{\rmdefault}{\mddefault}{\updefault}{\color[rgb]{0,0,0}$V_1^f(x)$}%
}}}}
\put(3526,-1057){\makebox(0,0)[lb]{\smash{{\SetFigFont{12}{14.4}{\rmdefault}{\mddefault}{\updefault}{\color[rgb]{0,0,0}$V_2^f(x_0)$}%
}}}}
\put(3301,-361){\makebox(0,0)[lb]{\smash{{\SetFigFont{12}{14.4}{\rmdefault}{\mddefault}{\updefault}{\color[rgb]{0,0,0}$V_2^f(y_0)$}%
}}}}
\put(6076,-1486){\makebox(0,0)[lb]{\smash{{\SetFigFont{12}{14.4}{\rmdefault}{\mddefault}{\updefault}{\color[rgb]{0,0,0}$x$}%
}}}}
\put(5941,-193){\makebox(0,0)[lb]{\smash{{\SetFigFont{12}{14.4}{\rmdefault}{\mddefault}{\updefault}{\color[rgb]{0,0,0}$y$}%
}}}}
\end{picture}%

\end{center}
\caption{Illustration to the argument. Quadrilateral in the box is
much smaller then the one outside.}\label{3_tube_expanding}
\end{figure}


Using (D2), we get
\begin{align}
\label{example_est} \frac{\tilde d_1(x,y)}{\tilde d_1(x_0,
y_0)}&=\prod_{i=1}^N\frac{D_f^{wu}(f^{-i}(x))}{D_f^{wu}(x_0)}\cdot\frac{\tilde
d_1(f^{-N}(x),f^{-N}(y))}{\tilde
d_1(x_0, f^{-N}(y_0))}\\
&\ge \left(\frac{D_f^{wu}(x_0)+\delta}{D_f^{wu}(x_0)}\right)^n
\cdot\frac{\tilde d_1(f^{-N}(x),f^{-N}(y))}{\tilde d_1(x_0,
f^{-N}(y_0))},\nonumber
\end{align}
which is an arbitrarily large number. Hence $\tilde d_1(x,y)$ is
arbitrarily large and we are done.

\begin{remark}
Although Proposition~\ref{example_prop} deals with a pretty special
situation we believe that the picture on Figure~\ref{2_tube} is
generic. To be more precise, we think that for any $g\in \mathcal U$
the following alternative holds. Either $V_2^g$ is conjugate to the
linear foliation $V_2^L$ or there exists a dense set $\Lambda$ such
that for any $x\in\Lambda$ and any $B\ni x$ the tube
$$
\bigcup_{y\in B}V_2^f(y)\subset\mathbb R^5
$$
contains arbitrarily long connected pieces of the leaves of\/
$V_1^g$.
\end{remark}

\begin{proof}[Proof of Proposition~\ref{example_prop}]
The argument is more delicate than the one presented above since we
do not know that the minimum of the derivative is achieved at $x_0$.

Let $B_0$ be a small ball around $x_0$ and $B_1\supset B_0$ a bigger
ball. Condition~(\ref{derivative_condition}) guarantees that we can
choose them in such a way that
$$
m_0<D_f^{wu}(x_0)< \sup_{x\in B_0}D_f^{wu}(x)<m_1<M<\min_{x\notin
B_1}D_f^{wu}(x),
$$
with $m_0$, $m_1$ and $M$ satisfying
\begin{equation}
\label{jacobian_ratio} \frac{Mm_0^{q-1}}{m_1^q} >1,
\end{equation}
where $q$ is an integer that depends only on the size of\/ $\mathcal
U$ and the size of\/ $B_1$. After that we choose $\tilde{\mathcal
U}\subset\mathcal U$ so the fixed point of\/ $g$ (that corresponds
to $x_0$) is inside of\/ $B_0$ and the property above persists.
Namely,
\begin{equation}
\label{jacobians} \forall g\in\tilde{\mathcal U}\;\;\;
m_0<\inf_{x\in B_0}D_g^{wu}(x)<\sup_{x\in
B_0}D_g^{wu}(x)<m_1<M<\min_{x\notin B_1}D_f^{wu}(x).
\end{equation}

Note that provided that $f$ is sufficiently $C^1$-close to $L$ and
the ball $B_1$ is small enough, any piece of a leaf of\/ $V_2^g$
outside of\/ $B_1$ that starts and ends on the boundary of\/ $B_1$
cannot be homotoped to a point keeping the endpoints on the
boundary. This is a minor technical detail that makes sure that the
picture shown on Figure~\ref{4_ab}a does not occur. Thus there is a
lower bound $R$ on the lengths of pieces of leaves of\/ $V_2^g$
outside of\/ $B_1$ with endpoints on the boundary of\/ $B_1$.
Obviously, there is also an upper bound $r$ on the lengths of pieces
of leaves of\/ $V_2^g$ inside $B_1$.


\begin{figure}[htbp]
\begin{center}

\begin{picture}(0,0)%
\includegraphics{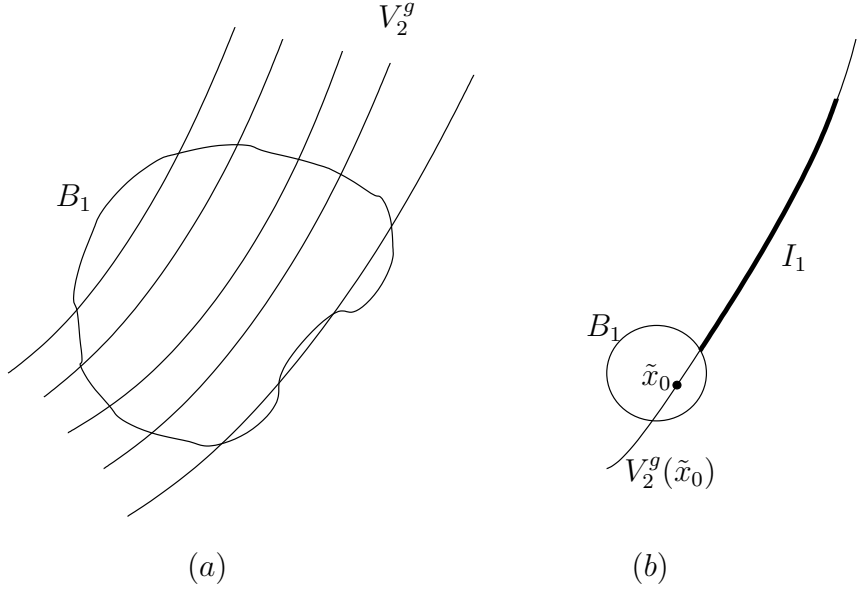}%
\end{picture}%
\setlength{\unitlength}{3947sp}%
\begingroup\makeatletter\ifx\SetFigFont\undefined%
\gdef\SetFigFont#1#2#3#4#5{%
\reset@font\fontsize{#1}{#2pt}%
\fontfamily{#3}\fontseries{#4}\fontshape{#5}%
\selectfont}%
\fi\endgroup%
\begin{picture}(5349,3667)(-161,-2222)
\put(151,164){\makebox(0,0)[lb]{\smash{{\SetFigFont{12}{14.4}{\rmdefault}{\mddefault}{\updefault}{\color[rgb]{0,0,0}$B_1$}%
}}}}
\put(2176,1289){\makebox(0,0)[lb]{\smash{{\SetFigFont{12}{14.4}{\rmdefault}{\mddefault}{\updefault}{\color[rgb]{0,0,0}$V_2^g$}%
}}}}
\put(976,-2161){\makebox(0,0)[lb]{\smash{{\SetFigFont{12}{14.4}{\rmdefault}{\mddefault}{\updefault}{\color[rgb]{0,0,0}$(a)$}%
}}}}
\put(4706,-211){\makebox(0,0)[lb]{\smash{{\SetFigFont{12}{14.4}{\rmdefault}{\mddefault}{\updefault}{\color[rgb]{0,0,0}$I_1$}%
}}}}
\put(3767,-2161){\makebox(0,0)[lb]{\smash{{\SetFigFont{12}{14.4}{\rmdefault}{\mddefault}{\updefault}{\color[rgb]{0,0,0}$(b)$}%
}}}}
\put(3721,-1561){\makebox(0,0)[lb]{\smash{{\SetFigFont{12}{14.4}{\rmdefault}{\mddefault}{\updefault}{\color[rgb]{0,0,0}$V_2^g(\tilde x_0)$}%
}}}}
\put(3481,-649){\makebox(0,0)[lb]{\smash{{\SetFigFont{12}{14.4}{\rmdefault}{\mddefault}{\updefault}{\color[rgb]{0,0,0}$B_1$}%
}}}}
\put(3829,-961){\makebox(0,0)[lb]{\smash{{\SetFigFont{12}{14.4}{\rmdefault}{\mddefault}{\updefault}{\color[rgb]{0,0,0}$\tilde x_0$}%
}}}}
\end{picture}%

\end{center}
\caption{(a) does not occur if\/ $B$ is sufficiently small;  (b)
choice of\/ $I_1$.}\label{4_ab}
\end{figure}

It is enough to check~(\ref{to_prove}) for a dense set $\Lambda$ of
points $x\in\mathbb T^5$. We take $\Lambda$ to be a subset of the
set of periodic points of\/ $g$
\begin{equation}
\label{Lambda} \Lambda=\{p:\;\ D_{f^{n(p)}}^{wu}(p)\le m_1^{n(p)}\},
\end{equation}
where $n(p)$ stands for the period of\/ $p$. The set $\Lambda$
consists of periodic points that spend a large but fixed percentage
of time inside of\/ $B_0$. It is fairly easy to show that $\Lambda$
is dense in $\mathbb T^5$. The proof is a trivial corollary of the
specification property (e.~g. see~\cite{KH}).

So we fix $\tilde x_0\in\Lambda$, a small ball $B$ centered at
$\tilde x_0$ and $y_0\in B\cap V_1^g(x_0)$ close to $\tilde x_0$.
Our goal now is to find $x\in V_2^g(\tilde x_0)$ far in the tube
$\EuScript T$ defined by~(\ref{tube}) for which we can carry out
estimates similar to~(\ref{example_est}).

We will be working with pieces of leaves of\/ $V_2^g$. Given a piece
$I$ with endpoints $z_1$ and $z_2$ let $|I|=d_2^g(z_1, z_2)$. Let
$q$ be a number such that for any piece $I$, $|I|=R$, we have
\begin{equation}
\label{q} |g^q(I)|>2R+r.
\end{equation}
Notice that $q$ can be chosen to be independent of\/ $g$ and depends
only on $\tilde\beta_2$, $R$ and~$r$.

Pick $I_1\subset V_2^g(\tilde x_0)$, $|I_1|=R$, $I_1\cap
B_1=\emptyset$, as close to $\tilde x_0$ as possible if\/ $\tilde
x_0\in B_1$ (see Figure~\ref{4_ab}b) or passing through $\tilde x_0$
if\/ $\tilde x_0\notin B_1$. Given $I_i$, $i\ge 1$ we choose
$I_{i+1}\subset f^q(I_i)$, $|I_{i+1}|=R$, $I_{i+1}\cap
B_1=\emptyset$. Condition~(\ref{q}) guarantees that such choice is
possible.

We fix $N$ large and take $x\in I_{Nq}\subset V_2^g(\tilde x_0)$.
Let $y=V_1^g(x)\cap V_2^g(y_0)$ as before. The construction of the
sequence $\{I_i,i\ge 1\}$ ensures that the points $f^{-qi}(x)$,
$i=0,\ldots N-1$, are outside $B_1$. This fact together
with~(\ref{jacobians}) and~(\ref{Lambda}) allows to carry out the
following estimate:
\begin{align*}
\frac{\tilde d_1(x,y)}{\tilde d_1(\tilde x_0,
y_0)}&=\prod_{i=1}^{Nq}\frac{D_g^{wu}(g^{-i}(x))}{D_g^{wu}(g^{-i}(\tilde
x_0))}\cdot\frac{\tilde d_1(f^{-Nq}(x),f^{-Nq}(y))}{\tilde
d_1(\tilde x_0, f^{-Nq}(y_0))}\\
&\ge \frac{M^Nm_0^{N(q-1)}}{m_1^{Nq}} \cdot\frac{\tilde
d_1(f^{-Nq}(x),f^{-Nq}(y))}{\tilde d_1(\tilde x_0, f^{-Nq}(y_0))}.
\end{align*}
The affine-like distance ratio on the right is bounded away from $0$
independently of\/ $N$ since $f^{-Nq}(x)\in I_1$, while the
coefficient in front of it is arbitrarily large according
to~(\ref{jacobian_ratio}). Hence $\tilde d_1^g(x,y)$ is arbitrarily
large and the projection of the tube $\EuScript T$ is dense in
$\mathbb T^5$.
\end{proof}

\section{Proof of Theorem A}

For reasons explained in Section 4 we first prove Theorem A assuming
that $f$ has Property $\mathcal A'$. The only place where we use
Property $\mathcal A'$ is in the proof of Lemma~\ref{step1}. In
Section~\ref{step1_revisited} we give another proof of
Lemma~\ref{step1} that uses Property $\mathcal A$ only.
\subsection{Scheme of the proof of Theorem A}
\label{scheme}

Recall the notation from~\ref{positive} for the $L$-invariant
splitting
$$T\mathbb T^d=F_l\oplus F_{l-1}\oplus\ldots \oplus F_1\oplus E_1\oplus E_2\oplus\ldots \oplus E_k$$
along the eigendirections with corresponding eigenvalues
$$\mu_l<\mu_{l-1}<\ldots <\mu_1<1<\lambda_1<\lambda_2<\ldots <\lambda_k.$$
We choose a neighborhood $\mathcal U$ in such a way that, for any
$f$ in $\mathcal U$, the invariant splitting survives:
$$T\mathbb T^d=F_l^f\oplus F_{l-1}^f\oplus\ldots \oplus F_1^f\oplus E_1^f\oplus E_2^f\oplus\ldots \oplus E_k^f,$$
with
\begin{equation}
\label{angles} \measuredangle(F_i,F_i^f)<\frac\pi2,
\;\measuredangle(E_j,E_j^f)<\frac\pi2,\;\; i=1,\ldots l,\;
j=1,\ldots k
\end{equation}
and $f$ is partially hyperbolic in the strongest sense; that is,
there exist $C>0$ and constants
$$\alpha_l<\tilde\alpha_{l-1}<\alpha_{l-1}<\ldots<\tilde\alpha_1<\alpha_1<1<\tilde\beta_1<\beta_1<\ldots<\tilde\beta_k$$
independent of the choice of\/ $f$ in $\mathcal U$ such that for
$n>0$
\begin{multline}
\label{phd}
\|D(f^n)(x)(v)\|\le C\alpha_l^n\|v\|,\;\;\; v\in F_l^f(x),\\
\shoveleft{\frac1C\tilde\alpha_{l-1}^n\|v\|\le\|D(f^n)(x)(v)\|\le C\alpha_{l-1}^n\|v\|,\;\;\; v\in F_{l-1}^f(x),}\\
\ldots\\
\shoveleft{\frac1C\tilde\alpha_{1}^n\|v\|\le\|D(f^n)(x)(v)\|\le C\alpha_{1}^n\|v\|,\;\;\; v\in F_{1}^f(x),}\\
\shoveleft{\frac1C\tilde\beta_{1}^n\|v\|\le\|D(f^n)(x)(v)\|\le C\beta_{1}^n\|v\|,\;\;\; v\in E_{1}^f(x),}\\
\ldots\\
\shoveleft{\frac1C\tilde\beta_{k}^n\|v\|\le\|D(f^n)(x)(v)\|,\;\;\;
v\in E_k^{f}(x).\hfill}
\end{multline} Equivalently, the Mather spectrum
of\/ $f$ does not contain $1$ and has $d$ connected components.

Such a choice is possible --- see Theorem~1 in~\cite{JPL}. This
theorem also guarantees that $C^1$-size of $\mathcal U$ is rather
large.

We show that the choice of\/ $\mathcal U$ guarantees unique
integrability of intermediate distributions. From now on, for the
sake of concreteness, we work with unstable distributions and
foliations.

For a given $f\in \mathcal U$ let $E^f(i,j)=E_i^f\oplus
E_{i+1}^f\oplus\ldots\oplus E_j^f$, $i\le j$.
\begin{lemma}
\label{integr1} For any $f$ in $\mathcal U$ distribution $E^f(1,1),
E^f(1,2),\ldots E^f(1,k)$ are uniquely integrable.
\end{lemma}

Let $W_1^f\subset W_2^f\subset\ldots \subset W_k^f$ be the
corresponding flag of weak unstable foliations. The last foliation
in the flag is the unstable foliation $W^f=W_k^f$.

\begin{lemma}
\label{integr2} For any $f$ in $\mathcal U$ and $i\le j$, the
distribution $E(i, j)$ is uniquely integrable.
\end{lemma}
Denote by $W^f(i,j)$, $i\le j$, the integral foliation of\/
$E^f(i,j)$. Also recall that we denote by $V_1^f, V_2^f,\ldots
V_k^f$ the integral foliations of\/ $E_1^f, E_2^f,\ldots E_k^f$
correspondingly. Notice that $V_i^f=W^f(i,i)$ and $W_i^f=W^f(1,i)$,
$i=1,\ldots k$.

Now we consider $f$ and $g$ as in Theorem A, $h\circ f=g\circ h$.
The conjugacy $h$ maps the unstable (stable) foliation of\/ $f$ into
the unstable (stable) foliation of\/ $g$. Moreover, $h$ preserves
the whole flag of weak unstable (stable) foliations.

\begin{lemma}
\label{weak_match} Fix an $i=1,\ldots k$. Then $h(W_i^f)=W_i^g$.
\end{lemma}

\begin{remark} The proof of this lemma does not use the assumption on the p.~d. We only need $f$ and $g$ to be in
$\mathcal U$.
\end{remark}

Lemmas~\ref{integr1}, \ref{integr2} and \ref{weak_match} can be
proved under a milder assumption. Instead of requiring $f$ and $g$
to be in $\mathcal U$
we can require the following.\\
{\bfseries Alternative Assumption:} $f$ and $g$ are partially
hyperbolic in the sense of (\ref{phd}) with the rate constants
satisfying
$$
\mu_l<\alpha_l<\tilde \alpha_{l-1}<\mu_{l-1}<\alpha_{l-1}<\ldots <
\tilde
\beta_{k-1}<\lambda_{k-1}<\beta_{k-1}<\tilde\beta_k<\lambda_k. \eqno
(\star)
$$
We think that $(\star)$ is actually automatic from~(\ref{phd}).

\begin{remark}
To carry out proofs of the Lemmas above under the Alternative
Assumption one needs to transfer the picture to the linear model by
the conjugacy and use the inequalities $(\star)$ for growth
arguments. This way one uses quasi-isometric foliations by straight
lines of the linear model instead of foliations of\/ $f$ which are a
priori not known to be quasi-isometric.
\end{remark}

\begin{con} Suppose that $f$ is homotopic to $L$ and partially hyperbolic in the strongest sense~(\ref{phd}).
Then the rate constants satisfy $(\star)$.
\end{con}

\begin{remark}
The proofs of Lemmas ~\ref{integr1}, \ref{integr2} and
\ref{weak_match} are the only places where we really need $f$ and
$g$ to be in $\mathcal U$. So, in Theorem A, the assumption that $f,
g\in\mathcal U$ can be substituted by the alternative assumption.
\end{remark}

\begin{lemma}\label{dense_leaf}
A leaf\/ $\;W_1^f(x)$ is dense in $\mathbb T^d$.
\end{lemma}

\begin{proof} By Lemma~\ref{weak_match}  the conjugacy between $L$ and $f$ takes the foliation
$W_1^L$ into the foliation $W_1^f$. According to
Proposition~\ref{denseness}, leaves of\/ $W_1^L$ are dense. Hence
leaves of\/ $W_1^f$ are dense.
\end{proof}

Next we describe the inductive procedure which leads to the
smoothness of\/ $h$ along the unstable foliation.
\\
\\
\emph{Induction base.} We know that $h$ takes $W_1^f$ into $W_1^g$.
\begin{lemma}
\label{conj_is_C1} The conjugacy $h$ is $C^{1+\nu}$-differentiable
along $W_1^f$, \ie the restrictions of\/ $h$ to the leaves of\/
$W_1^f$ are differentiable and the derivative is a $C^\nu$-function
on $\mathbb T^d$.
\end{lemma}
Provided that we have Lemma~\ref{dense_leaf}, the proof of
Lemma~\ref{conj_is_C1} is the same as the proof of Lemma 5
from~\cite{GG}.
\\
\\
\emph{Induction step.} The induction procedure is based on the
following lemmas.
\begin{lemma}
\label{step1} Assume that $h$ is $C^{1+\nu}$-differentiable along
$W_{m-1}^f$ and $h(V_i^f)=h(V_i^g)$, $i=1,\ldots  \;m-1$, $1<m\le
k$. Then $h(V_{m}^f)=V_{m}^g$.
\end{lemma}

\begin{lemma}
\label{step2} Assume that $h(V_{m}^f)=V_{m}^g$ for some $m=1,\ldots
\; k$. Then $h$ is $C^{1+\nu}$-differentiable along $V_m^f$.
\end{lemma}

We also use a regularity result due to Journ\'e.

\begin{reglemma}[\cite{J}]  Let $M_j$ be a manifold and $W_j^s$, $W_j^u$ be continuous transverse foliations
with uniformly smooth leaves, $j=1, 2$. Suppose that $h\colon M_1\to
M_2$ is a homeomorphism that maps $W_1^s$  into $W_2^s$ and $W_1^u$
into $W_2^u$. Moreover, assume that the restrictions of\/ $h$ to the
leaves of these foliations are uniformly $C^{r+\nu}, r\in \mathbb
N,\; 0<\nu<1$. Then $h$ is $C^{r+\nu}$.
\end{reglemma}

\begin{remark}
There are two more methods of proving analytical results of this
flavor besides Journ\'e's. One is due to de la Llave,  Marco,
Moriy\'on and the other one is due to Hurder and Katok
(see~\cite{KN} for a detailed discussion and proofs).  We remark
that we really need Journ\'e's result since the alternative
approaches require foliations to be absolutely continuous while we
apply the Regularity Lemma to various foliations that do not have to
be absolutely continuous.
\end{remark}

Now the inductive scheme can be described as follows. Assume that
$h$ is $C^{1+\nu}$ along $W_{m-1}^f$ for some $m\le k$ and
$h(V_i^f)=h(V_i^g)$, $i=1,\ldots  \;m-1$. By Lemma~\ref{step1} we
have that $h(V_{m}^f)=V_{m}^g$ and by Lemma~\ref{step2} $h$ is
$C^{1+\nu}$ along $V_{m}^f$. Fix a leaf\/ $W_{m}^f(x)$. Leaves of\/
$W_{m-1}^f$ and $V_{m}^f$ subfoliate $W_{m}^f(x)$ and it is clear
that the Regularity Lemma can be applied for $h\colon W_{m}^f(x)\to
W_{m}^g(h(x))$. Hence  $h$ is $C^{1+\nu}$ on every leaf of\/
$W_{m}^f$. H\"older-continuity of the derivative of\/ $h$ in the
direction transverse to $W_{m}^f$ is a direct consequence of
H\"older continuity of the derivatives along $W_{m-1}^f$ and
$V_{m}^f$. We conclude that $h$ is $C^{1+\nu}$-differentiable along
$W_{m}^f$.
\\

By induction  $h$ is $C^{1+\nu}$-differentiable along the unstable
foliation and analogously along the stable foliation. We finish the
proof of Theorem A by applying the Regularity Lemma to stable and
unstable foliations.

\subsection{Proof of the integrability lemmas}
\label{integrability}

In the proofs of Lemmas~\ref{integr1} and~\ref{integr2}, we work
with lifts of maps, distributions and foliations to $\mathbb R^d$.
We use the same notation for lifts as for the objects themselves.

\begin{proof}[Proof of Lemma~\ref{integr1}]
Fix $i < k$. We assume that the distribution $E^f(1,i)$ is not
integrable or it is integrable but not uniquely. In any case it
follows that we can find distinct points $a_0, a_1,\ldots a_m$ such
that
\begin{enumerate}
\item $\{a_1, a_2,\ldots a_m\} \subset W^f(a_0)$,
\item there are smooth curves $\tau_{j}\colon [0, 1]\to W^f(a_0)$, $j=1,\ldots m$, such that $\tau_{j}(0)=a_{j-1}$,
$\tau_{j}(1)=a_j$ and $\dot\tau_{j} \subset E_{p(j)}^{f}$, where
$p(j)\le i$,
\item there are smooth curves $\omega_{j}\colon [0, 1]\to W^f(a_0)$, $j=1,\ldots \tilde m$,
such that $\omega_1(0)=a_0$, $\omega_{\tilde m}(1)=a_m$,
$\omega_j(1)=\omega_{j+1}(0)$ for $j=1\ldots \tilde m-1$ and
$\dot\omega_{j} \subset E_{q(j)}^{f}$ with $q(j)>i$; $q(j_1)\neq
q(j_2)$ if $j_1\neq j_2$.
\end{enumerate}

Assume that $\tilde m=1$ and let $\omega=\omega_1$, $q=q(1)$. The
general case can be established in the same way by working with
$\omega=\omega_{\tilde j}$ where $\tilde j$ is chosen so that
$q(\tilde j)>q(j)$ for $j\neq\tilde j$.

Let $\tilde \tau$ be a piecewise smooth curve obtained by
concatenating $\tau_1$, $\tau_2$,... $\tau_{m-1}$ and $\tau_m$. From
the second property above and~(\ref{phd}), we get the following
rough estimate:
\begin{equation}
\label{es1} \forall n\ge 0\;\;\; \text{length}(f^n(\tilde \tau))\le
\beta_i^n  \text{length}(\tilde \tau).
\end{equation}
Similarly,
\begin{equation}
\label{es2} \forall n\ge 0\;\;\; \text{length}(f^n(\omega))\ge
\tilde \beta_{i+1}^n  \text{length}( \omega).
\end{equation}
Denote by $d(\cdot, \cdot)$ the usual distance in $\mathbb R^d$. It
follows from the assumption~(\ref{angles}) that any curve
$\gamma\colon [0,1]\to \mathbb R^d$ tangent to the distribution
$E_q^f$ is quasi-isometric:
$$
\exists c>0\;\text{such that }\;\;\; \text{length}(\gamma)\le c
\,d(\gamma(0),\gamma(1)).
$$
In particular,
\begin{equation}
\label{es3} \forall n\ge 0\;\;\; d(f^n(a_0),f^n(a_m))\ge\frac1c
\,\text{length}(f^n(\omega)).
\end{equation}
The inequalities~(\ref{es1}),~(\ref{es2}) and~(\ref{es3}) sum up to
a contradiction.
\end{proof}

\begin{proof}[Proof of Lemma~\ref{integr2}]
The theory of partial hyperbolicity guarantees that the
distributions $E^f(i,k)$, $i=1,\ldots k$, integrate uniquely to
foliations $W^f(i,k)$. Let us fix $i$ and $j$, $i<j$, and define
$W^f(i,j)=W^f(1,j)\cap W^f(i,k)$. Obviously $W^f(i,j)$ is an
integral foliation for $E^f(i,j)$. Unique integrability of\/
$E^f(i,j)$ is a direct consequence of the unique integrability of\/
$E^f(1,j)$ and $E^f(i,k)$.
\end{proof}

\subsection{Weak unstable flag is preserved: proof of Lemma~\ref{weak_match}}
\begin{proof}
We continue working on the universal cover. Pick two points $a$ and
$b$, $a \in W_i^f(b)$. Since
\begin{equation}
\label{lift} h(x+\vec m)=h(x)+\vec m,\;\;\vec m\in \mathbb Z^d
\end{equation}
we have that $d(h(x),h(y))\le c_1d(x,y)$ for any $x$ and $y$ such
that $d(x,y)\ge 1$. Hence, for any $n>0$,
\begin{equation*}
d(g^n(h(a)), g^n(h(b)))=d(h(f^n(a)), h(f^n(b)))\le c_2d(f^n(a),
f^n(b))\le c_2c_3\beta_i^n,
\end{equation*}
where $c_2$ and $c_3$ depend on $d(a, b)$. This inequality
guarantees $h(a)\in W_i^g(h(b))$. Since the choice of\/ $a$ and $b$
was arbitrary we conclude that $h(W_i^f)=W_i^g$.
\end{proof}

\subsection{Induction step 1: the conjugacy preserves the foliation $V_m$}
\label{proof_step1} We now prove Lemma~\ref{step1}, which is the key
ingredient in the proof of Theorem A. The proof is based on our idea
from~\cite{GG} but we take a rather different approach in order to
deal with the high dimension of\/ $W^f$. We provide a complete proof
almost without referring to~\cite{GG}. Nevertheless we strongly
encourage the reader to read Section 4.4 of~\cite{GG} first.

The goal is to prove that $h(V_m^f)=V_m^g$. So we will consider the
foliation $U=h^{-1}(V_m^g)$. As in the case for usual foliation,
$U(x)$ stands for the leaf of\/ $U$ passing through $x$ and $U(x,R)$
stands for the local leaf of size $R$. A priori, the leaves of\/ $U$
are just H\"older-continuous curves. Hence the local leaf needs to
be defined with a certain care. One way is to consider the lift of\/
$U$ and define the lift of the local leaf\/ $U(x, R)$ as a connected
component of\/ $x$ of the intersection $U(x)\cap B(x, R)$.

We prove Lemma~\ref{step1} by induction.
\subsubsection*{Induction base.}

We will be working on $m$-dimensional leaves of\/ $W_m^f$. By
Lemma~\ref{weak_match}, $U$ subfoliates $W_m^f$. In other words, for
any $x\in\mathbb T^d$, $U(x)\subset W_m^f(x)$.
\subsubsection*{Induction step.}

\emph{Suppose that  $U$ subfoliates $W^f(i,m)$ for some $i<m$. Then
$U$ subfoliates $W^f(i+1,m)$. }

By induction  $U$ subfoliate $W^f(m,m)=V_m^f$. Hence $U=V_m$.
\medskip

First let us prove several auxiliary claims. Note that all of the
foliations that we are dealing with are oriented and the orientation
is preserved under the dynamics. Denote by $d_j^f$ and $d_j^g$ the
induced distances on the leaves of\/ $V_j^f$ and $V_j^g$
correspondingly, $j=1,\ldots k$.

\begin{lemma}
\label{abcd} Consider a point $a\in \mathbb T^d$. Pick a point $b\in
U(a)$ and let $\tilde b=V_i^f(b)\cap W^f(i+1,m)(a)$. Assume that
$\tilde b\neq b$. Pick a point  $c\in V_i^f(a)$ and let $d=U(c)\cap
W^f(i, m-1)(b)$, $\tilde d=V_i^f(d)\cap W^f(i+1, m)(c)$. Then
$\tilde d\neq d$ and the orientations  of the pairs $(b, \tilde b)$
and $(d, \tilde d)$ in $V_i^f$ are the same.
\end{lemma}
The statement of the lemma when $i=1$ and $m=3$ is illustrated on
Figure~\ref{5_abcd}.

\begin{remark} Since by the induction hypothesis, $h(W^f(i, m-1))=W^g(i, m-1)$, we see that the leaf\/ $U(a)$ intersects each leaf\/ $W^f(i, m-1)(x)$, $x \in W^f(i,m)(a)$ exactly once.
\end{remark}

\begin{proof}
Let $e=V_i^f(b)\cap W^f(i+1,m)(d)$ and $\tilde e=V_i^f(b)\cap
W^f(i+1,m)(\tilde d)$. Obviously $(e, \tilde e)$ has the same
orientation as $(d,\tilde d)$ and also has the advantage of lying on
the leaf\/ $V_i^f(b)$. Therefore, we forget about $(d,\tilde d)$ and
work with $(e, \tilde e)$.


\begin{figure}[htbp]
\begin{center}

\begin{picture}(0,0)%
\includegraphics{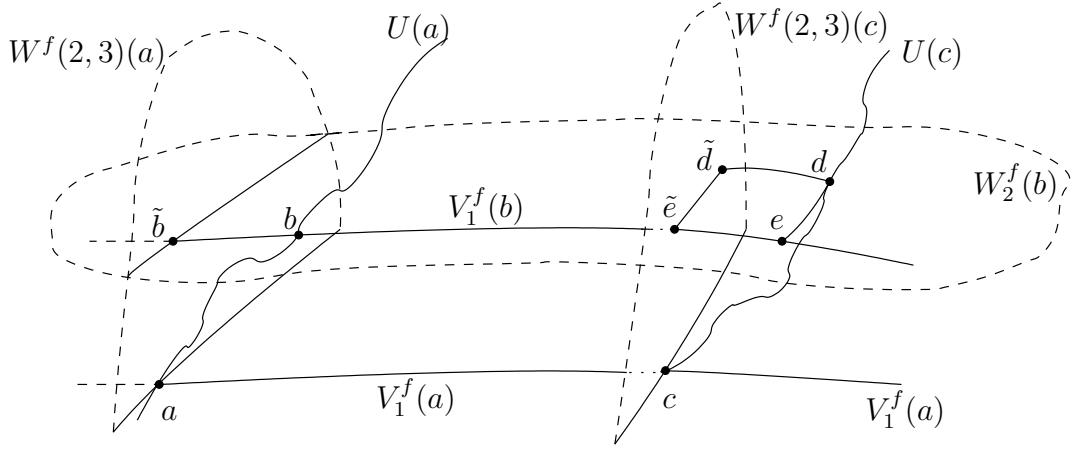}%
\end{picture}%
\setlength{\unitlength}{3947sp}%
\begingroup\makeatletter\ifx\SetFigFont\undefined%
\gdef\SetFigFont#1#2#3#4#5{%
\reset@font\fontsize{#1}{#2pt}%
\fontfamily{#3}\fontseries{#4}\fontshape{#5}%
\selectfont}%
\fi\endgroup%
\begin{picture}(7013,2796)(-359,-1423)
\put(5026,-1261){\makebox(0,0)[lb]{\smash{{\SetFigFont{12}{14.4}{\rmdefault}{\mddefault}{\updefault}{\color[rgb]{0,0,0}$V_1^f(a)$}%
}}}}
\put(5701,164){\makebox(0,0)[lb]{\smash{{\SetFigFont{12}{14.4}{\rmdefault}{\mddefault}{\updefault}{\color[rgb]{0,0,0}$W_2^f(b)$}%
}}}}
\put(5251,989){\makebox(0,0)[lb]{\smash{{\SetFigFont{12}{14.4}{\rmdefault}{\mddefault}{\updefault}{\color[rgb]{0,0,0}$U(c)$}%
}}}}
\put(2017,1139){\makebox(0,0)[lb]{\smash{{\SetFigFont{12}{14.4}{\rmdefault}{\mddefault}{\updefault}{\color[rgb]{0,0,0}$U(a)$}%
}}}}
\put(3751,-1186){\makebox(0,0)[lb]{\smash{{\SetFigFont{12}{14.4}{\rmdefault}{\mddefault}{\updefault}{\color[rgb]{0,0,0}$c$}%
}}}}
\put(4426,-61){\makebox(0,0)[lb]{\smash{{\SetFigFont{12}{14.4}{\rmdefault}{\mddefault}{\updefault}{\color[rgb]{0,0,0}$e$}%
}}}}
\put(3751, 14){\makebox(0,0)[lb]{\smash{{\SetFigFont{12}{14.4}{\rmdefault}{\mddefault}{\updefault}{\color[rgb]{0,0,0}$\tilde e$}%
}}}}
\put(4681,299){\makebox(0,0)[lb]{\smash{{\SetFigFont{12}{14.4}{\rmdefault}{\mddefault}{\updefault}{\color[rgb]{0,0,0}$d$}%
}}}}
\put(4201,1163){\makebox(0,0)[lb]{\smash{{\SetFigFont{12}{14.4}{\rmdefault}{\mddefault}{\updefault}{\color[rgb]{0,0,0}$W^f(2,3)(c)$}%
}}}}
\put(1951,-1186){\makebox(0,0)[lb]{\smash{{\SetFigFont{12}{14.4}{\rmdefault}{\mddefault}{\updefault}{\color[rgb]{0,0,0}$V_1^f(a)$}%
}}}}
\put(2401, 14){\makebox(0,0)[lb]{\smash{{\SetFigFont{12}{14.4}{\rmdefault}{\mddefault}{\updefault}{\color[rgb]{0,0,0}$V_1^f(b)$}%
}}}}
\put(601,-1261){\makebox(0,0)[lb]{\smash{{\SetFigFont{12}{14.4}{\rmdefault}{\mddefault}{\updefault}{\color[rgb]{0,0,0}$a$}%
}}}}
\put(3961,314){\makebox(0,0)[lb]{\smash{{\SetFigFont{12}{14.4}{\rmdefault}{\mddefault}{\updefault}{\color[rgb]{0,0,0}$\tilde d$}%
}}}}
\put(-359,989){\makebox(0,0)[lb]{\smash{{\SetFigFont{12}{14.4}{\rmdefault}{\mddefault}{\updefault}{\color[rgb]{0,0,0}$W^f(2,3)(a)$}%
}}}}
\put(541,-109){\makebox(0,0)[lb]{\smash{{\SetFigFont{12}{14.4}{\rmdefault}{\mddefault}{\updefault}{\color[rgb]{0,0,0}$\tilde b$}%
}}}}
\put(1369,-61){\makebox(0,0)[lb]{\smash{{\SetFigFont{12}{14.4}{\rmdefault}{\mddefault}{\updefault}{\color[rgb]{0,0,0}$b$}%
}}}}
\end{picture}%

\end{center}
\caption{Illustration to Lemma~\ref{abcd} when $i=1$ and
$m=3$.}\label{5_abcd}
\end{figure}

We use an affine structure on the expanding foliation $V_i^f$.
Namely we work with the affine distance-like function $\tilde d_i$.
We refer to~\cite{GG} for the definition. There we define the affine
distance-like function on the weak unstable foliation. The
definition for the foliation $V_i^f$ is the same with obvious
modifications. Recall the crucial properties of\/ $\tilde d_i$:
\begin{itemize}
\item[(D1)] $\tilde d_i(x,y)=d_i^f(x,y)+o(d_i^f(x,y))$,
\item[(D2)] $\tilde d_i(f(x),f(y))=D_f^i(x)\tilde d_i(x,y)$, where $D_f^i$ is the derivative of\/ $f$ along
$V_i^f$.
\item[(D3)] $\forall K>0 \;\exists C>0$ such that
\begin{equation*}
\frac1C\tilde d_i(x,y)\le d_i^f(x,y)\le C\tilde d_i(x,y)
\end{equation*}
whenever $d_i(x, y)<K$.
\end{itemize}

Assume that $(e, \tilde e)$ has orientation opposite to $(b,\tilde
b)$ or $e=\tilde e$. For the sake of concreteness we assume that
these points lie on $V_i^f(b)$ in the order $b, \tilde b, \tilde e,
e$. All other cases can be treated similarly. Then
$$
\tilde d_i(b,e)\ge \tilde d_i(b,\tilde e)>\tilde d_i(b,\tilde
e)-\tilde d_i(b, \tilde b).
$$
\begin{remark}
Notice that $\tilde d_i(b,\tilde e)-\tilde d_i(b, \tilde b)\neq
\tilde d_i(\tilde b,\tilde e)$ since $\tilde d_i$ is neither
symmetric nor additive. The distance $\tilde d_i$ is given by an
integral of a certain density with normalization defined by the
first argument. As long as the first argument (point $b$ in the
above inequality) is the same, all natural inequalities hold.
\end{remark}

Applying (D2), we get that
$$
\forall n>0\;\;\; \frac{\tilde d_i(f^{-n}(b),f^{-n}(e))}{\tilde
d_i(f^{-n}(b),f^{-n}(\tilde e))-\tilde d_i(f^{-n}(b),f^{-n}(\tilde
b))}=c_1>1,
$$
where $c_1$ does not depend on $n$. By property (D1) we can switch
to the usual distance:
\begin{equation}
\label{qwe4} \exists N:\; \forall n>N\;\;\;
\frac{d_i^f(f^{-n}(b),f^{-n}(e))}{d_i^f(f^{-n}(\tilde b),
f^{-n}(\tilde e))}>c_2>1,
\end{equation}
where $c_2$ does not depend on $n$.

Under the action of\/ $f^{-1}$, strong unstable leaves of\/
$W^f(i+1,m)$ contract exponentially faster then weak unstable leaves
of\/ $V_i^f$. Thus
\begin{equation}
\label{qwe3} \forall \varepsilon>0 \;\; \exists N: \;\forall
n>N\;\;\; \left|
\frac{d_i^f(f^{-n}(a),f^{-n}(c))}{d_i^f(f^{-n}(\tilde
b),f^{-n}(\tilde e))}-1 \right|<\varepsilon.
\end{equation}

We have that $h(e)\in W^g(i+1,m)(h(c))$. Indeed, notice that
$${e=V_i^f(b)\cap W^f(i+1,m)(d)}=V_i^f(b)\cap W^f(i+1,m-1)(d)$$ (if\/ $i=m-1$ than we have $e=d$). Thus
\begin{align*}
h(e)&= h(V_i^f(b)\cap W^f(i+1,m-1)(d))=V_i^g(h(b))\cap W^g(i+1,m-1)(h(d))\\
&=V_i^g(h(b))\cap W^g(i+1,m)(h(d))=V_i^g(h(b))\cap W^g(i+1,m)(h(c)),
\end{align*}
where the last equality is justified by the fact that $h(d)\in
V_m^g(h(c))$.  We know also that $h(b)\in W^g(i+1,m)(h(a))$. Hence,
analogously to~(\ref{qwe3}), we have
\begin{equation}
\label{qwe1} \forall \varepsilon>0 \;\; \exists N: \;\forall
n>N\;\;\; \left|
\frac{d_i^g\left(g^{-n}(h(a)),g^{-n}(h(c))\right)}{d_i^g\left(g^{-n}(h(b)),g^{-n}(h(e))\right)}-1
\right|<\varepsilon.
\end{equation}
On the other hand, we also know that $h$ is continuously
differentiable along $V_i^f$. Hence

\begin{align}
\label{qwe2} \forall \varepsilon>0\;\; \exists N:\; \forall
n>N\;\;\;\;\;
\left| \frac{d_i^g\left(g^{-n}(h(a)), g^{-n}(h(c))\right)}{d_i^f(f^{-n}(a), f^{-n}(c))}-D_h^{i}(f^{-n}(a))\right|&<\varepsilon\\
\intertext{and} \;\; \left| \frac{d_i^g\left(g^{-n}(h(b)),
g^{-n}(h(e))\right)}{d_i^g(f^{-n}(b),
f^{-n}(e))}-D_h^{i}(f^{-n}(a))\right|&<\varepsilon.\nonumber
\end{align}

Therefore from~(\ref{qwe1}) and~(\ref{qwe2}) we have
$$
\forall \varepsilon>0 \;\; \exists N: \;\forall n>N\;\;\; \left|
\frac{d_i^f(f^{-n}(a),f^{-n}(c))}{d_i^f(f^{-n}(b),f^{-n}(e))}-1
\right| <\varepsilon,
$$
which we combine with~(\ref{qwe3}) to get
$$
\forall \varepsilon>0 \;\; \exists N: \;\forall n>N\;\;\; \left|
\frac{d_i^f(f^{-n}(b),f^{-n}(e))}{d_i^f(f^{-n}(\tilde
b),f^{-n}(\tilde e))}-1 \right| <\varepsilon.
$$

We have reached a contradiction with~(\ref{qwe4})
\end{proof}

\begin{remark}
By the same argument one can prove that if\/ $b=\tilde b$ then
$d=\tilde d$.
\end{remark}

\begin{lemma}
\label{distance} Consider a weak unstable leaf\/ $W_{m-1}^f(a)$ and
$b\in V_m^f(a)$, $b\ne a$. For any $y\in W_{m-1}^f(a)$, let
$y'=W_{m-1}^f(b)\cap V_m^f(y)$. Then there exist $c_1, c_2>0$ such
that $\forall y\in W_{m-1}^f(a)$ $\;\;\;c_1>d_m^f(y,y')>c_2$.
\end{lemma}

\begin{proof}
We will be working on the universal cover $\mathbb R^d$. We abuse
notation slightly by using the same notation for the lifted objects.
Note that the leaves on $\mathbb R^d$ are connected components of
preimages by the projection map of the leaves on $\mathbb T^d$.

Let $h_f$ be the conjugacy with the linear model, $h_f\circ f=L\circ
h_f$. Lemma~\ref{weak_match} holds for $h_f$:
$h_f(W_{m-1}^f)=W_{m-1}^L$. The leaves $W_{m-1}^L(h_f(a))$ and
$W_{m-1}^L(h_f(b))$ are parallel hyperplanes. Thus the lower bound
follows from the uniform continuity of\/~$h_f$.

It follows from~(\ref{lift}) that $h_f^{-1}-Id$ is bounded. Hence we
can find positive $R$ that depends only on size of\/ $\mathcal U$
such that
$$W_{m-1}^f(a)\subset \text{Tube}_a\stackrel{\mathrm{def}}{=}\bigcup_{x\in B(a, R)}W_{m-1}^L(x)$$
and
$$W_{m-1}^f(b)\subset \text{Tube}_b\stackrel{\mathrm{def}}{=}\bigcup_{x\in B(b, R)}W_{m-1}^L(x).$$

Then, obviously,
$$
d_m^f(y, y')\le \sup\{d_m^f(x,x')\;|\;x\in \text{Tube}_a, x'\in
\text{Tube}_b\cap V_m^f(x)\}.
$$
Assumption~(\ref{angles}) guarantees that $E_m^f$ is uniformly
transverse to $TW_{m-1}^L=E_1^L\oplus E_2^L\oplus\ldots \oplus
E_{m-1}^L$. Thus the supremum above is finite.
\end{proof}
\begin{remark}\mbox{}
\begin{enumerate}
\item Given two points $a, b\in \mathbb R^d$ let $$\hat d(a,b)=\text{distance}(W_{m-1}^L(h_f(a)),W_{m-1}^L(h_f(b))).$$ It is
clear from the proof that constants $c_1$ and $c_2$ can be chosen in
such a way that they depend only on $\hat d(a,b)$.
\item In the proof above we do not use the fact that both $W_{m-1}^f$ and $V_m^f$ are expanding. We only need them to
be transverse. Thus, if we substitute for the weak unstable
foliation $W_{m-1}^f$ some weak stable foliation $\mathcal F$, the
statement still holds.
\item As mentioned earlier the assumption~(\ref{angles}) is crucial only for Lemmas~\ref{integr1}, \ref{integr2} and
\ref{weak_match}. We used this assumption in the proof above only
for convenience. A slightly more delicate argument goes through
without using assumption~(\ref{angles}).
\end{enumerate}
\end{remark}

\begin{proof}[Proof of the induction step.]
We will be working inside of the leaves of\/ $W^f(i,m)$. Assume that
$U$ does not subfoliate $W^f(i+1,m)$. Then there exists a point
$x_0$ and $x_1\in U(x_0)$ close to $x_0$ such that $x_1\notin
W^f(i+1,m)(x_0)$.

We fix an orientation $\mathcal O$ of\/ $U$ and $V_i^f$ that is
defined on pairs of points $(x,y)$, $y\in U(x)$ and $(x,y)$, $y\in
V_i^f(x)$. Although we denote these orientations by the same symbol
it will not cause any confusion since $U$ and $V_i^f$ are
topologically transverse.

For every $(x,y)$, $y\in U(x)$ with $\mathcal O(x,y)=\mathcal
O(x_0,x_1)$, define $$[x,y]=W^f({i+1},m)(x)\cap V_i^f(y).$$ For
instance, in Lemma~\ref{abcd}, $\tilde b=[a,b]$, $\tilde d=[c, d]$.


\begin{figure}[htbp]
\begin{center}

\begin{picture}(0,0)%
\includegraphics{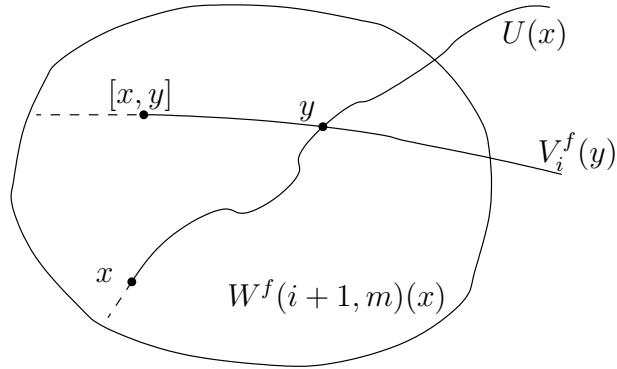}%
\end{picture}%
\setlength{\unitlength}{3947sp}%
\begingroup\makeatletter\ifx\SetFigFont\undefined%
\gdef\SetFigFont#1#2#3#4#5{%
\reset@font\fontsize{#1}{#2pt}%
\fontfamily{#3}\fontseries{#4}\fontshape{#5}%
\selectfont}%
\fi\endgroup%
\begin{picture}(4181,2296)(56,-1579)
\put(601,-1036){\makebox(0,0)[lb]{\smash{{\SetFigFont{12}{14.4}{\rmdefault}{\mddefault}{\updefault}{\color[rgb]{0,0,0}$x$}%
}}}}
\put(1876, 14){\makebox(0,0)[lb]{\smash{{\SetFigFont{12}{14.4}{\rmdefault}{\mddefault}{\updefault}{\color[rgb]{0,0,0}$y$}%
}}}}
\put(676, 89){\makebox(0,0)[lb]{\smash{{\SetFigFont{12}{14.4}{\rmdefault}{\mddefault}{\updefault}{\color[rgb]{0,0,0}$[x,y]$}%
}}}}
\put(1426,-1186){\makebox(0,0)[lb]{\smash{{\SetFigFont{12}{14.4}{\rmdefault}{\mddefault}{\updefault}{\color[rgb]{0,0,0}$W^f(i+1,m)(x)$}%
}}}}
\put(3376,-286){\makebox(0,0)[lb]{\smash{{\SetFigFont{12}{14.4}{\rmdefault}{\mddefault}{\updefault}{\color[rgb]{0,0,0}$V_i^f(y)$}%
}}}}
\put(3151,464){\makebox(0,0)[lb]{\smash{{\SetFigFont{12}{14.4}{\rmdefault}{\mddefault}{\updefault}{\color[rgb]{0,0,0}$U(x)$}%
}}}}
\end{picture}%

\end{center}
\caption{Definition of\/ $[x,y]$.}\label{6_bracket}
\end{figure}


\begin{lemma}
\label{monotonicity} For every $(x,y)$ as above, either $[x,y]=y$ or
$\mathcal O([x,y],y)=O^+\stackrel{\mathrm{def}}{=}\mathcal
O([x_0,x_1],x_1)$.
\end{lemma}
\begin{proof}
Let $a_0=\hat d(x_0,x_1)$ (for definition of\/ $\hat d$ see the
remark after the proof of Lemma~\ref{distance}). The number $a_0$ is
positive since $U(x)$ is transverse to $W_{m-1}^f$.

For any $y\in\mathbb T^d$, there is a unique point $sh(y)\in U(y)$
such that $\hat d(y, sh(y))=a_0$ and $\mathcal O(y,sh(y))=\mathcal
O(x_0,x_1)$.

The leaves of all the foliations that we consider depend
continuously on the point. Therefore we can find a small ball $B$
centered at $x_0$ such that $\forall y\in B$ $[y,sh(y)]\neq sh(y)$
and $\mathcal O([y,sh(y)],sh(y))=O^+$.

Next, let us fix $y\in B$ and choose any $z\in V_i^f(y)$. Apply
Lemma~\ref{abcd} for $a=y$, $b=sh(y)$, $c=z$, $d=sh(z)$ to get that
$[z,sh(z)]\neq sh(z)$ and $\mathcal O([z,sh(z)],sh(z))=\mathcal
O([y,sh(y)],sh(y))=O^+$ as shown on the Figure~\ref{7_shift_pres}.

By Property $\mathcal A$,
$$
\bigcup_{y\in B}V_i^f=\mathbb T^d.
$$
Thus
\begin{equation}
\label{first} \forall z\in \mathbb T^d\;\;\;[z,sh(z)]\neq
sh(z)\;\;\text{and}\;\;\mathcal O([z,sh(z)],sh(z))=O^+.
\end{equation}


\begin{figure}[htbp]
\begin{center}

\begin{picture}(0,0)%
\includegraphics{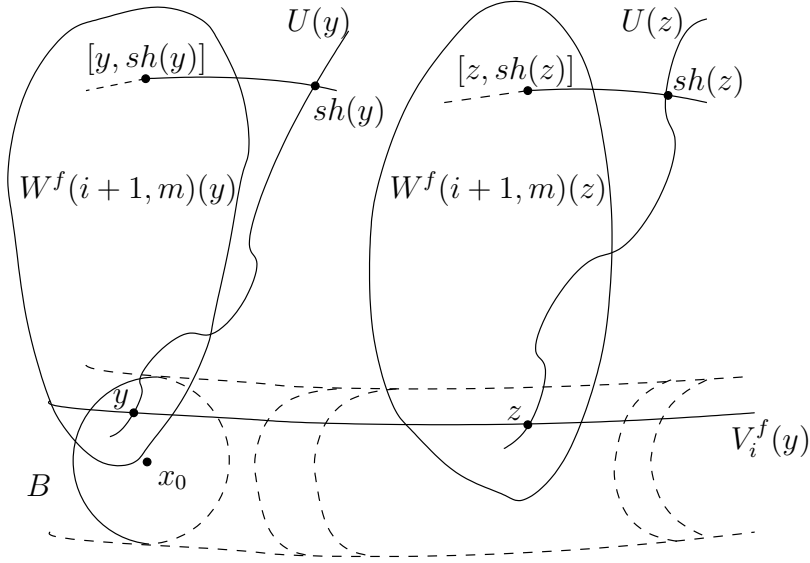}%
\end{picture}%
\setlength{\unitlength}{3947sp}%
\begingroup\makeatletter\ifx\SetFigFont\undefined%
\gdef\SetFigFont#1#2#3#4#5{%
\reset@font\fontsize{#1}{#2pt}%
\fontfamily{#3}\fontseries{#4}\fontshape{#5}%
\selectfont}%
\fi\endgroup%
\begin{picture}(5414,3513)(-202,-4651)
\put(-74,-4261){\makebox(0,0)[lb]{\smash{{\SetFigFont{12}{14.4}{\rmdefault}{\mddefault}{\updefault}{\color[rgb]{0,0,0}$B$}%
}}}}
\put(751,-4186){\makebox(0,0)[lb]{\smash{{\SetFigFont{12}{14.4}{\rmdefault}{\mddefault}{\updefault}{\color[rgb]{0,0,0}$x_0$}%
}}}}
\put(301,-1561){\makebox(0,0)[lb]{\smash{{\SetFigFont{12}{14.4}{\rmdefault}{\mddefault}{\updefault}{\color[rgb]{0,0,0}$[y,sh(y)]$}%
}}}}
\put(2626,-1636){\makebox(0,0)[lb]{\smash{{\SetFigFont{12}{14.4}{\rmdefault}{\mddefault}{\updefault}{\color[rgb]{0,0,0}$[z,sh(z)]$}%
}}}}
\put(4351,-3961){\makebox(0,0)[lb]{\smash{{\SetFigFont{12}{14.4}{\rmdefault}{\mddefault}{\updefault}{\color[rgb]{0,0,0}$V_i^f(y)$}%
}}}}
\put(1726,-1897){\makebox(0,0)[lb]{\smash{{\SetFigFont{12}{14.4}{\rmdefault}{\mddefault}{\updefault}{\color[rgb]{0,0,0}$sh(y)$}%
}}}}
\put(2941,-3781){\makebox(0,0)[lb]{\smash{{\SetFigFont{12}{14.4}{\rmdefault}{\mddefault}{\updefault}{\color[rgb]{0,0,0}$z$}%
}}}}
\put(-119,-2386){\makebox(0,0)[lb]{\smash{{\SetFigFont{12}{14.4}{\rmdefault}{\mddefault}{\updefault}{\color[rgb]{0,0,0}$W^f(i+1,m)(y)$}%
}}}}
\put(2209,-2386){\makebox(0,0)[lb]{\smash{{\SetFigFont{12}{14.4}{\rmdefault}{\mddefault}{\updefault}{\color[rgb]{0,0,0}$W^f(i+1,m)(z)$}%
}}}}
\put(1561,-1336){\makebox(0,0)[lb]{\smash{{\SetFigFont{12}{14.4}{\rmdefault}{\mddefault}{\updefault}{\color[rgb]{0,0,0}$U(y)$}%
}}}}
\put(469,-3673){\makebox(0,0)[lb]{\smash{{\SetFigFont{12}{14.4}{\rmdefault}{\mddefault}{\updefault}{\color[rgb]{0,0,0}$y$}%
}}}}
\put(3661,-1336){\makebox(0,0)[lb]{\smash{{\SetFigFont{12}{14.4}{\rmdefault}{\mddefault}{\updefault}{\color[rgb]{0,0,0}$U(z)$}%
}}}}
\put(3985,-1705){\makebox(0,0)[lb]{\smash{{\SetFigFont{12}{14.4}{\rmdefault}{\mddefault}{\updefault}{\color[rgb]{0,0,0}$sh(z)$}%
}}}}
\end{picture}%

\end{center}
\caption{Orientation of\/ $([z,sh(z)],sh(z))$ is positive for any
$z$ in the $V_i^f$-tube through the ball $B$. Foliation $W^f(i+1,m)$
is two-di\-men\-si\-on\-al on the picture.}\label{7_shift_pres}
\end{figure}

Now let us assume the contrary of the statement of the lemma.
Namely, assume that there exists $\tilde x_0$ and $\tilde x_1$,
$\tilde x_1\in U(\tilde x_0)$, $\mathcal O(\tilde x_0, \tilde
x_1)=\mathcal O(x_0, x_1)$, such that $[\tilde x_0,\tilde x_1]\neq
\tilde x_1$ and $\mathcal O([\tilde x_0,\tilde x_1], \tilde
x_1)\stackrel{\mathrm{def}}{=}O^-\neq O^+$. By perturbing $\tilde
x_1$ infinitesimally along $U(\tilde x_0)$ we can ensure that
$N_1a_0=N_2\hat d(\tilde x_0,\tilde x_1)$, where $N_1$ and $N_2$ are
some large integer numbers.

For any $y\in\mathbb T^d$ there is a unique point
$\widetilde{sh}(y)\in U(y)$ such that $\hat d(y,
\widetilde{sh}(y))=\hat d(\tilde x_0,\tilde x_1)$ and $\mathcal
O(y,\widetilde{sh}(y))=\mathcal O(\tilde x_0,\tilde x_1)$. Then by
the same argument we show an analog of~(\ref{first}):
\begin{equation}
\label{second} \forall z\in \mathbb
T^d\;\;\;[z,\widetilde{sh}(z)]\neq
\widetilde{sh}(z)\;\;\text{and}\;\;\mathcal
O([z,\widetilde{sh}(z)],\widetilde{sh}(z))=O^-.
\end{equation}

Pick a point $x\in \mathbb T^d$ and $y,z\in U(x)$, $\mathcal
O(x,y)=\mathcal O(y,z)$. Assume that $\mathcal O([x,y],y)=\mathcal
O([y,z],z)$. Then $\mathcal O([x,z],z)=\mathcal O([x,y],y)$. This
obvious property allows us to ``iterate" $sh$ and $\tilde {sh}$.

Choose any $z$ and ``iterate"~(\ref{first}) and~(\ref{second}) $N_1$
and $N_2$ times correspondingly as shown on the
Figure~\ref{8_shift_iterate}.


\begin{figure}[htbp]
\begin{center}

\begin{picture}(0,0)%
\includegraphics{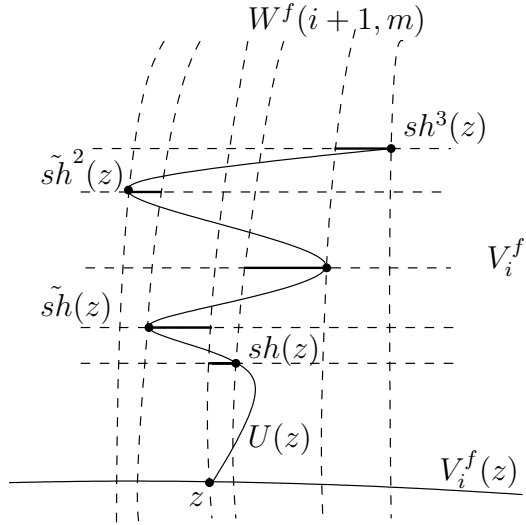}%
\end{picture}%
\setlength{\unitlength}{3947sp}%
\begingroup\makeatletter\ifx\SetFigFont\undefined%
\gdef\SetFigFont#1#2#3#4#5{%
\reset@font\fontsize{#1}{#2pt}%
\fontfamily{#3}\fontseries{#4}\fontshape{#5}%
\selectfont}%
\fi\endgroup%
\begin{picture}(3627,3318)(214,-2248)
\put(1351,-2086){\makebox(0,0)[lb]{\smash{{\SetFigFont{12}{14.4}{\rmdefault}{\mddefault}{\updefault}{\color[rgb]{0,0,0}$z$}%
}}}}
\put(2701,239){\makebox(0,0)[lb]{\smash{{\SetFigFont{12}{14.4}{\rmdefault}{\mddefault}{\updefault}{\color[rgb]{0,0,0}$sh^3(z)$}%
}}}}
\put(3226,-586){\makebox(0,0)[lb]{\smash{{\SetFigFont{12}{14.4}{\rmdefault}{\mddefault}{\updefault}{\color[rgb]{0,0,0}$V_i^f$}%
}}}}
\put(1726,914){\makebox(0,0)[lb]{\smash{{\SetFigFont{12}{14.4}{\rmdefault}{\mddefault}{\updefault}{\color[rgb]{0,0,0}$W^f(i+1,m)$}%
}}}}
\put(2929,-1921){\makebox(0,0)[lb]{\smash{{\SetFigFont{12}{14.4}{\rmdefault}{\mddefault}{\updefault}{\color[rgb]{0,0,0}$V_i^f(z)$}%
}}}}
\put(421,-61){\makebox(0,0)[lb]{\smash{{\SetFigFont{12}{14.4}{\rmdefault}{\mddefault}{\updefault}{\color[rgb]{0,0,0}$\tilde {sh}^2(z)$}%
}}}}
\put(1726,-1153){\makebox(0,0)[lb]{\smash{{\SetFigFont{12}{14.4}{\rmdefault}{\mddefault}{\updefault}{\color[rgb]{0,0,0}$sh(z)$}%
}}}}
\put(421,-886){\makebox(0,0)[lb]{\smash{{\SetFigFont{12}{14.4}{\rmdefault}{\mddefault}{\updefault}{\color[rgb]{0,0,0}$\tilde {sh}(z)$}%
}}}}
\put(1726,-1711){\makebox(0,0)[lb]{\smash{{\SetFigFont{12}{14.4}{\rmdefault}{\mddefault}{\updefault}{\color[rgb]{0,0,0}$U(z)$}%
}}}}
\end{picture}%

\end{center}
\caption{Illustration to the argument with shifts along $U(z)$.
Foliation $W^f(i+1,m)$ is one-di\-men\-si\-on\-al here, $N_1=3$,
$N_2=2$. The black segments of\/ $V_i^f$ carry known information
about the orientations of\/ $([\cdot,sh(\cdot)],sh(\cdot))$ and
$([\cdot,\widetilde{sh}(\cdot)],\widetilde{sh}(\cdot))$. This
picture is clearly impossible if\/
$sh^{N_1}=\widetilde{sh}^{N_2}$.}\label{8_shift_iterate}
\end{figure}


We get that
$$
\mathcal O([z,sh^{N_1}(z)],sh^{N_1}(z))=O^+\;\;\text{and}\;\;
\mathcal O([z,\widetilde{sh}^{N_2}(z)],\widetilde{sh}^{N_2}(z))=O^-.
$$
To get a contradiction it remains to notice that
$sh^{N_1}=\widetilde{sh}^{N_2}$. Hence the lemma is proved.
\end{proof}

From~(\ref{first}) we see that for any $z\in\mathbb T^d$,
$d_i^f([z,sh(z)],sh(z))>0$. Hence, due to the compactness and
continuity of the function $d_i^f([\cdot,sh(\cdot)],sh(\cdot))$, we
have $\delta<d_i^f([z,sh(z)],sh(z))<\Delta$ for some positive
$\delta$ and $\Delta$. Lemma~\ref{monotonicity} guarantees even
more:
\begin{equation}
\begin{split}
\label{bounded_shift} \forall x\in\mathbb T^d \;\;\text{and}\;\;
y\in U(x),\mathcal O(x,y)=\mathcal
O(x_0,x_1)\;\;\text{such that}\;\; \hat d(x,y)\le a_0,\\
\;\;\text{we have}\;\;\;
d_i^f([x,y],y)<\Delta.\end{split}\end{equation}

From now on it is more convenient to work on the universal cover,
although formally we do not have to do it since we are working
inside of the leaves of\/ $W^f(i,m)$, which are isometric to their
lifts.

Let $x_n=sh^n(x_0)$, $n>0$. For every $n\ge0$, we have $\mathcal
O([x_n,x_{n+1}],x_{n+1})=O^+$ and
$d_i^f([x_n,x_{n+1}],x_{n+1})>\delta$. Lemma~\ref{monotonicity} also
tells us that the leaves of $U$ are monotone with respect to
foliation $W^f(i+1,m)$. Namely, for any $x\in \mathbb T^d$, the
intersection $U(x)\cap {W^f(i+1,m)(x)}$ is a connected piece of\/
$U(x)$.

Denote by $\overline{x_n,x_{n+1}}$ the piece of\/ $U(x_0)$ that lies
between $x_n$ and $x_{n+1}$. We know that for any $n\ge 0$,
$\overline{x_n,x_{n+1}}$ is confined between the leaves
${W^f(i,m-1)(x_n)}$ and ${W^f(i,m-1)(x_{n+1})}$.
Lemma~\ref{monotonicity} guarantees that $\overline{x_n,x_{n+1}}$ is
also confined between ${W^f(i+1,m)(x_n)}$ and
${W^f(i+1,m)(x_{n+1})}$, as shown on Figure~\ref{9_rectangle}. Thus,
it makes sense to measure two different ``dimensions" of\/
$\overline{x_n,x_{n+1}}$. Namely, let $a_n=\hat d(x_n, x_{n+1})$ and
$b_n=d_i^f([x_n,x_{n+1}],x_{n+1})$. As we have remarked earlier
$b_n>\delta>0$ and $a_n=a_0$ by the definition of\/ $\hat d$ and
$sh$.


\begin{figure}[htbp]
\begin{center}

\begin{picture}(0,0)%
\includegraphics{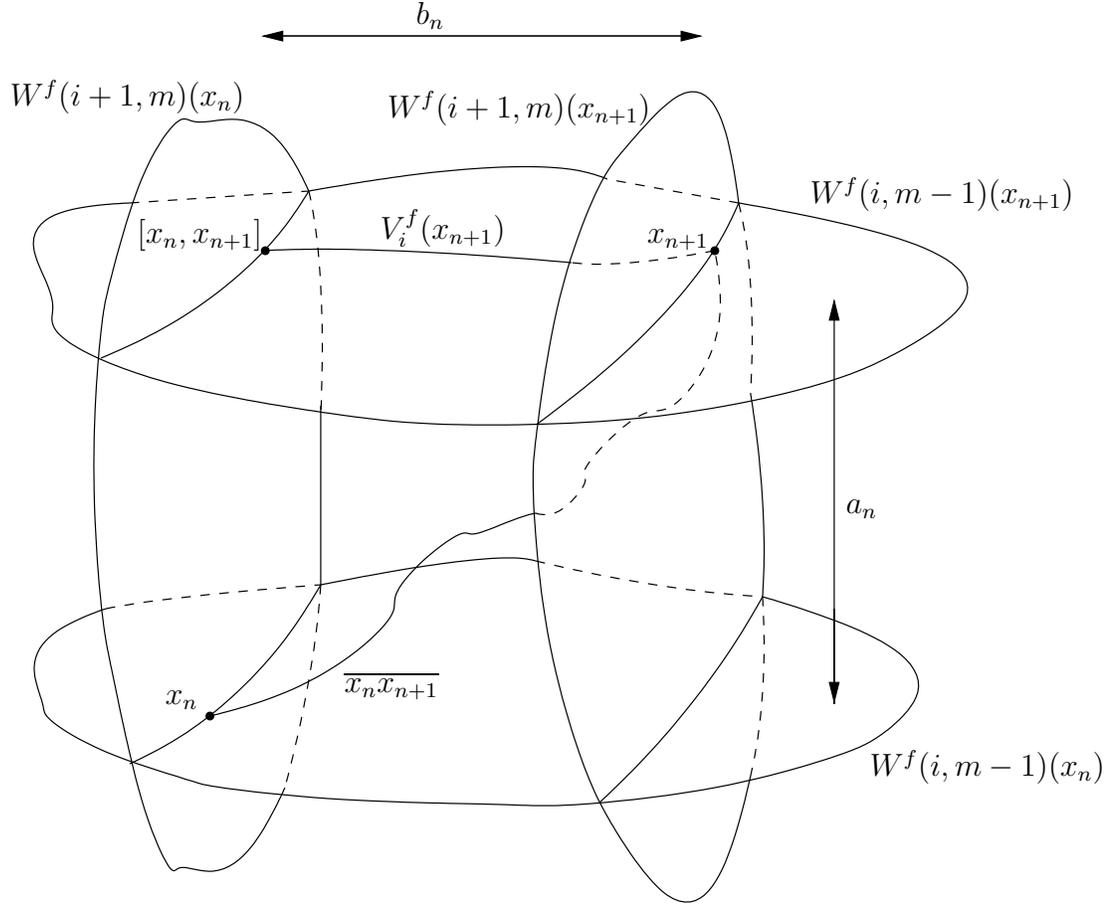}%
\end{picture}%
\setlength{\unitlength}{3947sp}%
\begingroup\makeatletter\ifx\SetFigFont\undefined%
\gdef\SetFigFont#1#2#3#4#5{%
\reset@font\fontsize{#1}{#2pt}%
\fontfamily{#3}\fontseries{#4}\fontshape{#5}%
\selectfont}%
\fi\endgroup%
\begin{picture}(6901,5687)(-449,-3267)
\put(2101,2264){\makebox(0,0)[lb]{\smash{{\SetFigFont{12}{14.4}{\rmdefault}{\mddefault}{\updefault}{\color[rgb]{0,0,0}$b_n$}%
}}}}
\put(4801,-811){\makebox(0,0)[lb]{\smash{{\SetFigFont{12}{14.4}{\rmdefault}{\mddefault}{\updefault}{\color[rgb]{0,0,0}$a_n$}%
}}}}
\put(526,-2011){\makebox(0,0)[lb]{\smash{{\SetFigFont{12}{14.4}{\rmdefault}{\mddefault}{\updefault}{\color[rgb]{0,0,0}$x_n$}%
}}}}
\put(3553,887){\makebox(0,0)[lb]{\smash{{\SetFigFont{12}{14.4}{\rmdefault}{\mddefault}{\updefault}{\color[rgb]{0,0,0}$x_{n+1}$}%
}}}}
\put(1651,-1936){\makebox(0,0)[lb]{\smash{{\SetFigFont{12}{14.4}{\rmdefault}{\mddefault}{\updefault}{\color[rgb]{0,0,0}$\overline{x_nx_{n+1}}$}%
}}}}
\put(1876,914){\makebox(0,0)[lb]{\smash{{\SetFigFont{12}{14.4}{\rmdefault}{\mddefault}{\updefault}{\color[rgb]{0,0,0}$V_i^f(x_{n+1})$}%
}}}}
\put(4951,-2461){\makebox(0,0)[lb]{\smash{{\SetFigFont{12}{14.4}{\rmdefault}{\mddefault}{\updefault}{\color[rgb]{0,0,0}$W^f(i,m-1)(x_n)$}%
}}}}
\put(4576,1139){\makebox(0,0)[lb]{\smash{{\SetFigFont{12}{14.4}{\rmdefault}{\mddefault}{\updefault}{\color[rgb]{0,0,0}$W^f(i,m-1)(x_{n+1})$}%
}}}}
\put(1921,1679){\makebox(0,0)[lb]{\smash{{\SetFigFont{12}{14.4}{\rmdefault}{\mddefault}{\updefault}{\color[rgb]{0,0,0}$W^f(i+1,m)(x_{n+1})$}%
}}}}
\put(-449,1751){\makebox(0,0)[lb]{\smash{{\SetFigFont{12}{14.4}{\rmdefault}{\mddefault}{\updefault}{\color[rgb]{0,0,0}$W^f(i+1,m)(x_n)$}%
}}}}
\put(349,887){\makebox(0,0)[lb]{\smash{{\SetFigFont{12}{14.4}{\rmdefault}{\mddefault}{\updefault}{\color[rgb]{0,0,0}$[x_n,x_{n+1}]$}%
}}}}
\end{picture}%

\end{center}
\caption{Piece $\overline{x_nx_{n+1}}$ is ``monotone" with respect
to foliation $W^f(i,m-1)$. By Lemma~\ref{monotonicity}
$\overline{x_nx_{n+1}}$ is also ``monotone" with respect to
$W^f(i+1,m)$: the intersections of\/ $\overline{x_nx_{n+1}}$ with
local leaves of\/ $W^f(i+1,m)$ are points or connected components
of\/ $\overline{x_nx_{n+1}}$. On this picture foliations
$W^f(i,m-1)$ and $W^f(i+1,m)$ are
two-di\-men\-si\-on\-al.}\label{9_rectangle}
\end{figure}

These ``dimensions" behave nicely under the dynamics. Namely,
\begin{equation*}
\forall N>0 \;\;
(f_*)^{-N}(b_n)\stackrel{\mathrm{def}}{=}d_i^f([f^{-N}(x_n),f^{-N}(x_{n+1})],f^{-N}(x_{n+1}))\ge\delta\beta_i^{-N}
\end{equation*}
and
$$
\forall N>0 \;\; (f_*)^{-N}(a_n)\stackrel{\mathrm{def}}{=}\hat
d(f^{-N}(x_n),f^{-N}(x_{n+1}))= a_0\lambda_m^{-N}.
$$
Recall that $\lambda_m>\beta_i$.

The idea now is to show that the leaf\/ $U(f^{-N}(x_0))$ is lying
``too close" to $W^f(i,m-1)(x_0)$ for $N$ large, which would lead to
a contradiction.

Take $N$ large and let $M=\lfloor \lambda_m^N\rfloor$. Then
\begin{align}
\label{v_size} \hat d(f^{-N}(x_0),f^{-N}(x_M))&=\sum_{j=0}^{M-1}\hat d(f^{-N}(x_j),f^{-N}(x_{j+1}))\\
&= \sum_{j=0}^{M-1}(f_*)^{-N}(a_j)= M a_0 \lambda_m^{-N}\le
a_0.\nonumber
\end{align}
The first equality holds since the holonomy along $W^f(i,m-1)$ is
isometric with respect to $\hat d$.

To estimate $d_i^f([f^{-N}(x_0),f^{-N}(x_M)],f^{-N}(x_M))$ in a
similar way we need to have control over holonomies along
$W^f(i+1,m)$.

Fix two small one-di\-men\-si\-on\-al transversals $T(x)\subset
V_i^f(x)$ and $T(y)\subset V_i^f(y)$, $y\in U(x)$ with $\hat
d(x,y)\le a_0$. This condition ensures that the distance between $x$
and $y$ along $W^f(i,m)(x)$ is uniformly bounded from above. To see
this we only need to bound the distance between $h(x)$ and $h(y)$
along $W^g(i,m)(h(x))$. This, in turn, is a direct consequence of
Lemma~\ref{distance} applied to $g$ since $h(y)\in V_m^g(h(x))$.

Consider the holonomy map along $W^f(i+1,m)\;$ $H\colon T(x)\to
T(y)$. This holonomy can be viewed as the holonomy along
$W^f(i+1,k)$. Recall that  $W^f(i+1,k)$ is the fast unstable
foliation. Since $f$ is at least $C^2$-differentiable, $W^f(i+1,k)$
is Lipschitz inside of\/ $W^f(i,k)$. Moreover, since the distance
between $x$ and $y$ is bounded from above, the Lipschitz constant
$C_{Hol}$ of\/ $H$ is uniform in $x$ and $y$. For a proof,
see~\cite{LY}, Section 4.2. They prove that the unstable foliation
is Lipschitz within center-unstable leaves but the proof goes
through for $W^f(i+1,k)$ within the leaves of\/ $W^f(i,k)$.

Let $\tilde x_j=W^f(i+1,m)(f^{-N}(x_j))\cap V_i^f(f^{-N}(x_M))$,
$j=1,\ldots M$. Then
\begin{align*}
d_i^f([f^{-N}(x_0),f^{-N}(x_M)],&f^{-N}(x_M))\\&=
\sum_{j=0}^{M-1}d_i^f(\tilde x_j,\tilde x_{j+1})\\&\ge
C_{Hol}\sum_{j=0}^{M-1}d_i^f([f^{-N}(x_j),f^{-N}(x_{j+1})],f^{-N}(x_{j+1}))\\&=C_{Hol}\sum_{j=0}^{M-1}
(f_*)^{-N}(b_j)\\  &\ge C_{Hol}M\delta \beta_i^{-N}.
\end{align*}


\begin{figure}[htbp]
\begin{center}

\begin{picture}(0,0)%
\includegraphics{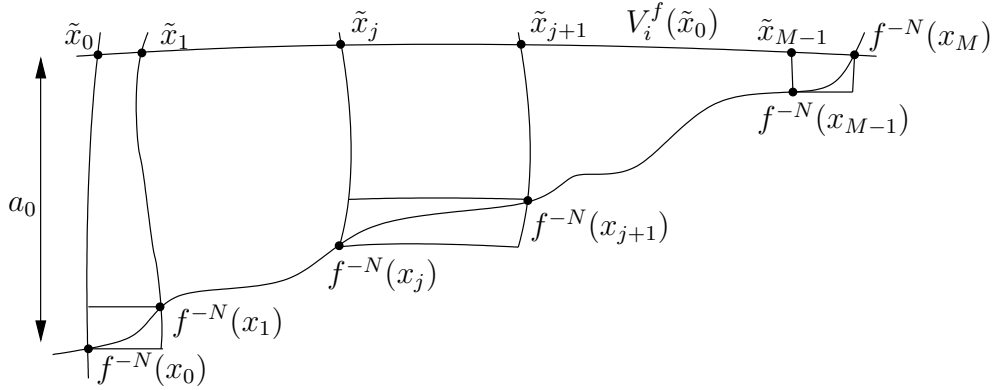}%
\end{picture}%
\setlength{\unitlength}{3947sp}%
\begingroup\makeatletter\ifx\SetFigFont\undefined%
\gdef\SetFigFont#1#2#3#4#5{%
\reset@font\fontsize{#1}{#2pt}%
\fontfamily{#3}\fontseries{#4}\fontshape{#5}%
\selectfont}%
\fi\endgroup%
\begin{picture}(6630,2392)(-503,-1469)
\put(526,-1111){\makebox(0,0)[lb]{\smash{{\SetFigFont{12}{14.4}{\rmdefault}{\mddefault}{\updefault}{\color[rgb]{0,0,0}$f^{-N}(x_1)$}%
}}}}
\put(1501,-811){\makebox(0,0)[lb]{\smash{{\SetFigFont{12}{14.4}{\rmdefault}{\mddefault}{\updefault}{\color[rgb]{0,0,0}$f^{-N}(x_j)$}%
}}}}
\put(2776,-511){\makebox(0,0)[lb]{\smash{{\SetFigFont{12}{14.4}{\rmdefault}{\mddefault}{\updefault}{\color[rgb]{0,0,0}$f^{-N}(x_{j+1})$}%
}}}}
\put(-149,689){\makebox(0,0)[lb]{\smash{{\SetFigFont{12}{14.4}{\rmdefault}{\mddefault}{\updefault}{\color[rgb]{0,0,0}$\tilde x_0$}%
}}}}
\put(451,689){\makebox(0,0)[lb]{\smash{{\SetFigFont{12}{14.4}{\rmdefault}{\mddefault}{\updefault}{\color[rgb]{0,0,0}$\tilde x_1$}%
}}}}
\put(-503,-361){\makebox(0,0)[lb]{\smash{{\SetFigFont{12}{14.4}{\rmdefault}{\mddefault}{\updefault}{\color[rgb]{0,0,0}$a_0$}%
}}}}
\put( 49,-1411){\makebox(0,0)[lb]{\smash{{\SetFigFont{12}{14.4}{\rmdefault}{\mddefault}{\updefault}{\color[rgb]{0,0,0}$f^{-N}(x_0)$}%
}}}}
\put(1651,767){\makebox(0,0)[lb]{\smash{{\SetFigFont{12}{14.4}{\rmdefault}{\mddefault}{\updefault}{\color[rgb]{0,0,0}$\tilde x_j$}%
}}}}
\put(2776,767){\makebox(0,0)[lb]{\smash{{\SetFigFont{12}{14.4}{\rmdefault}{\mddefault}{\updefault}{\color[rgb]{0,0,0}$\tilde x_{j+1}$}%
}}}}
\put(3376,767){\makebox(0,0)[lb]{\smash{{\SetFigFont{12}{14.4}{\rmdefault}{\mddefault}{\updefault}{\color[rgb]{0,0,0}$V_i^f(\tilde x_0)$}%
}}}}
\put(4201,719){\makebox(0,0)[lb]{\smash{{\SetFigFont{12}{14.4}{\rmdefault}{\mddefault}{\updefault}{\color[rgb]{0,0,0}$\tilde x_{M-1}$}%
}}}}
\put(4897,683){\makebox(0,0)[lb]{\smash{{\SetFigFont{12}{14.4}{\rmdefault}{\mddefault}{\updefault}{\color[rgb]{0,0,0}$f^{-N}(x_M)$}%
}}}}
\put(4201,161){\makebox(0,0)[lb]{\smash{{\SetFigFont{12}{14.4}{\rmdefault}{\mddefault}{\updefault}{\color[rgb]{0,0,0}$f^{-N}(x_{M-1})$}%
}}}}
\end{picture}%

\end{center}
\caption{Small rectangles along the leaf\/ $U(f^{-N}(x_0))$ are very
``flat" according to the estimates on $(f_*)^{-N}(b_n)$ and
$(f_*)^{-N}(a_n)$. Together with the Lipschitz property of the
foliation $W^f(i+1,m)$, this provides an estimate from below on the
horizontal size $d_i^f(\tilde x_0,
f^{-N}(x_M))$.}\label{10_flatness}
\end{figure}
The holonomy constant $C_{Hol}$ is uniform since
$$\hat d(f^{-N}(x_j),\tilde x_j)\le\hat d(f^{-N}(x_0),\tilde x_j)=\hat
d(f^{-N}(x_0),f^{-N}(x_M))\le a_0$$ by~(\ref{v_size}). Notice that
$C_{Hol}M\delta \beta_i^{-N}$ can be arbitrarily big when
$N\to\infty$, while $d(f^{-N}(x_0),f^{-N}(x_M))\le a_0$ which
contradicts to~(\ref{bounded_shift}). Hence the induction step is
established.
\end{proof}

\subsection{Induction step 2: proof of Lemma~\ref{step2} by transitive point argument}
\label{ind_step2} The proof of Lemma~\ref{step2} is carried out in a
way similar to the proofs of Lemmas~4 and~5 from~\cite{GG}. Here we
overview the scheme and deal with the complications that arise due
to higher dimension.

First, using the assumption on the p.~d., we argue that $h$ is
uniformly Lipschitz along $V_m^f$, \ie for any point $x$ the
restriction $h\big|_{V_m^f(x)}\colon V_m^f(x)\to V_m^g(x)$ is a
Lipschitz map with a Lipschitz constant  that does not depend on
$x$.  At this step, the assumption on the p.~d. along $V_m^f$ is
used.

The Lipschitz property implies differentiability at almost every
point with respect to the Lebesgue measure on the leaves of\/
$V_m^f$. The next step is to show that the differentiability of\/
$h$ along $V_m^f$ at a transitive point $x$ implies that $h$ is
$C^{1+\nu}$-differentiable along $V_m^f$. This is done by a direct
approximation argument (see Step~1 in Section 4.3 in~\cite{GG}). The
transitive point $x$ ``spreads differentiability" all over the
torus.

Last but not least, we need to find such a transitive point $x$.
Ideally, for that we would find an ergodic measure $\mu$ with full
support such that the foliation $V_m^f$ is absolutely continuous
with respect to $\mu$. Then, by the Birkhoff Ergodic Theorem, almost
every point would be transitive. Since $V_m^f$ is absolutely
continuous, we would then have that almost every point, with respect
to the Lebesgue measure on the leaves, is transitive. Hence we would
have a full measure set of the points that we are looking for.

Unfortunately, we cannot carry out the scheme described above. The
problem is that the foliation $V_m^f$ is not absolutely continuous
with respect to natural ergodic measures (see~\cite{GG} for detailed
discussion and~\cite{SX} for in-depth analysis of this phenomenon).
Instead, we construct a measure $\mu$ such that almost every point
is transitive and $V_m^f$ is absolutely continuous with respect to
$\mu$. This is clearly sufficient.

The construction follows the lines of Pesin--Sinai's~\cite{PS}
construction of\/ $u$-Gibbs measures. Given a partially hyperbolic
diffeomorphism, they construct a measure such that the unstable
foliation is absolutely continuous with respect to the measure. In
fact, this construction works well for any expanding foliation. We
apply this construction to  $m$-dimensional foliation $W_m^f$.

The construction is described as follows. Let $x_0$ be a fixed point
of\/ $f$. For any $y\in W_m^f(x_0)$, define
$$\rho(y)=\prod_{n\ge 0}\frac{J_m^f(f^{-n}(y))}{J_m^f(x_0)},$$
where $J_m^f$ is the Jacobian of\/ $f\big|_{W_m^f}$.

Let $\mathcal V_0$ be an open bounded neighborhood of\/ $x_0$ in
$W_m^f(x_0)$. Consider a probability measure $\eta_0$ supported on
$\mathcal V_0$ with density proportional to $\rho(\cdot)$. For
$n>0$, define
$$\mathcal V_n=f^n(\mathcal V_0),\; \eta_n=(f^n)_*\eta_0.$$
Let
$$\mu_n=\frac1n\sum_{i=0}^{n-1}\eta_i.$$
By the Krylov--Bogoljubov Theorem, $\{\mu_n; n\ge 0\}$ is weakly
compact and any of its limits is $f$-invariant. Let $\mu$ be an
accumulation point of\/ $\{\mu_n; n\ge 0\}$. This is the measure
that we are looking for.

Foliation $W_m^f$ is absolutely continuous with respect to $\mu$. We
refer to~\cite{PS} or~\cite{GG} for the proof. The proof
of~\cite{GG} requires some minimal modifications that are due to the
higher dimension of\/ $W_m^f$.

Since the foliation $W_m^f$ is conjugate to the linear foliation
$W_m^L$  we have that for any open ball $B$,
$$
\exists R>0\;\;\;{\bigcup_{y\in B} W_m^f(y, R)}=\mathbb T^d,
$$
where $W_m^f(y,R)$ is a ball of radius $R$ inside of the leaf\/
$W_m^f(y)$. Together with absolute continuity, this guarantees that
$\mu$ almost every point is transitive. See~\cite{GG}, Section~4.3,
Step~3 for the proof. We stress that we do not need to know that
$\mu$ has full support in that argument.

It is left to show that the conjugacy $h$ is
$C^{1+\nu}$-differentiable in the direction of\/ $V_m^f$ at $\mu$
almost every point. For this we need to argue that $V_m^f$ is
absolutely continuous with respect to $\mu$.

The foliation $W^f(m,k)$ is Lipschitz inside of a leaf of\/ $W^f$
(again we refer to~\cite{LY}, Section 4.2). Hence
$V_m^f=W^f(m,k)\cap W_m^f$ is Lipschitz inside of a leaf of\/
$W_m^f=W^f\cap W_m^f$. So  $V_m^f$ is absolutely continuous with
respect to the Lebesgue measure on a leaf of\/ $W_m^f$ while $W_m^f$
is absolutely continuous with respect to $\mu$. Therefore $V_m^f$ is
absolutely continuous with respect to $\mu$.


\subsection{Induction step 1 revisited}
\label{step1_revisited}

To carry out the proof of Lemma~\ref{step1} assuming Property
$\mathcal A$ only, we shrink the neighborhood $\mathcal U$ even
more. In addition to~(\ref{angles}) and~(\ref{phd}), we require
$f\in\mathcal U$ to have a narrow spectrum. Namely,
$$
\forall m,\; 1<m\le k \;\;\; \frac{\log \tilde\beta_m}{\log
\beta_m}>\frac{\log \beta_{m-1}}{\log\tilde\beta_m}
$$
and the analogous condition on the contraction rates $\alpha_j$,
$\tilde \alpha_j$. The following condition that we will actually use
is obviously a consequence of the above one.
\begin{equation}
\label{exponents} \forall i<k\;\;\;\text{and}\;\;\;\forall m,
\;i<m\le k, \;\;\;\rho\stackrel{\mathrm{def}}{=}\frac{\log
\tilde\beta_m}{\log \beta_m}>\frac{\log \beta_i}{\log\tilde\beta_m}.
\end{equation}
This inequality can be achieved by shrinking the size of\/ $\mathcal
U$ since $\beta_j$ and $\tilde \beta_j$ get arbitrarily close to
$\lambda_j$, $j=1,\ldots k$.

\begin{remark}
Condition~(\ref{exponents}) greatly simplifies the proof of
Lemma~\ref{step1}. We have yet another, longer, proof (but based on
the same idea) of Lemma~\ref{step1} that works for any $f$ with
Property $\mathcal A$ in $\mathcal U$ as defined in
Section~\ref{scheme}. It will not appear here.
\end{remark}

We start the proof as in Section~\ref{proof_step1}. The first place
where we use Property $\mathcal A'$ is the proof of
Lemma~\ref{monotonicity}. So we reprove induction step 1 with
Property $\mathcal A$ only assuming that we have got everything that
preceded Lemma~\ref{monotonicity}. With Property $\mathcal A$, the
proof of Lemma~\ref{monotonicity} still goes through, although
instead of~(\ref{first}), we get
\begin{equation*}
\label{first_new} \forall z\in \mathbb
T^d\;\;\;\text{either}\;\;\;[z,sh(z)]=
sh(z)\;\;\text{or}\;\;\mathcal O([z,sh(z)],sh(z))=O^+.
\end{equation*}
Thus we still have Lemma~\ref{monotonicity} and the upper
bound~(\ref{bounded_shift}) but not the lower bound
$d_i^f([z,sh(z)],sh(z))>\delta$. This is the reason why we cannot
proceed with the proof of the induction step as at the end of
Section~\ref{proof_step1}.

\begin{proof}[Proof of the induction step.]

As before, we need to show that $U$ subfoliates foliation
${W^f(i+1,m)}$.

Fix an orientation $\mathcal O$ on $V_m^f$ and $V_i^f$. Given
$x\in\mathbb T^d$ and $\varepsilon>0$ choose $\bar x\in V_m^f(x)$
such that $d_m^f(x,\bar x)=\varepsilon$ and $\mathcal O(x,\bar
x)=O^+$. Let $\bar y=U(x)\cap W^f(i,m-1)(\bar x)$ and
$y=V_i^f(x)\cap W^f(i+1,m)(\bar y)$. This way we define an
$\varepsilon$-``rectangle" $\EuScript R=\EuScript R(x,\bar x, y,
\bar y)$ with base point $x$, vertical size $d_m^f(x,\bar
x)=\varepsilon$, and horizontal size $d_i^f(x,y)=\bar\varepsilon$.
\begin{remark}
Notice that we measure vertical size in a way different from the one
in Section~\ref{proof_step1}.
\end{remark}
It is clear that this ``rectangle" is uniquely defined by its
``diagonal" $(x, \bar y)$ (Figure~\ref{9_rectangle} is the picture
of ``rectangle" with diagonal $(x_n,x_{n+1})$). Sometimes we will
use the notation $\EuScript R(x,\bar y)$. Note that by
Lemma~\ref{monotonicity}, $\mathcal O(x,y)$ does not depend on $x$
and $\varepsilon$. It also guarantees that the piece of\/ $U(x)$
between $x$ and $\bar y$ lies ``inside" of\/ $\EuScript R(x,\bar
y)$. The horizontal size $\bar \varepsilon$ might happen to be equal
to zero.

Next we define a set of base points $X_\varepsilon$ such that
$U(x)$, $x\in X_\varepsilon$, has big H\"older slope inside of
corresponding $\varepsilon$-rectangle,
$$
X_\varepsilon=\{x\in\mathbb T^d:\;\;
\bar\varepsilon\le\varepsilon^\delta\},
$$
with some $\delta$ satisfying inequality
$\rho>\delta>\log\beta_i/\log\tilde\beta_m$.

Let $\mu$ be the measure constructed in Section~\ref{ind_step2}.
Recall that $\mu$-almost every point is transitive. The foliation
$W^f(i,m)$ is absolutely continuous with respect to $\mu$. The
latter can be shown in the same way as absolute continuity of\/
$V_m^f$ is shown in Section~\ref{ind_step2}.

We consider two cases.

\smallskip
{\bfseries Case 1.} $\varlimsup\limits_{\varepsilon\to
0}\mu(X_\varepsilon)>0$.
\smallskip

Then choose $\{X_{\varepsilon_n}, n\ge 1\}$, $\varepsilon_n\to 0$ as
$n\to\infty$, such that $\varlimsup_{n\to\infty}
\mu(X_{\varepsilon_n})>0$.

The idea now is to iterate a rectangle with base point in
$X_{\varepsilon_n}$ and vertical size $\varepsilon_n$ until the
vertical size is approximately 1. Since the H\"older slope of the
initial rectangle was big, it will turn out that the horizontal size
of the iterated rectangle is extremely small. This argument will
show that for a set of base points of positive measure, the
horizontal size of rectangles is equal to zero. Hence the leaves
of\/ $U$ lie inside of the leaves of\/ $W^f(i+1,m)$.

Given $n$, let $N=N(n)$ be the largest number such that
$\frac1C\tilde\beta_m^N\varepsilon_n<1$ (constant $C$ here is from
Definition~(\ref{phd})). Take $x\in X_{\varepsilon_n}$ and the
corresponding $\varepsilon_n$-rectangle $\EuScript R(x,y,\bar x,\bar
y)$ and consider its image $\EuScript R(f^N(x),f^N(y),f^N(\bar
x),f^N(\bar y))$. The choice of\/ $N$ provides a lower bound on the
vertical size,
$$
VS\left(\EuScript R(f^N(x),f^N(y),f^N(\bar x),f^N(\bar y))\right)=
d_m^f(f^N(x),f^N(\bar x))\ge\frac{1}{\beta_m},
$$
while the horizontal size can be estimated as follows:
\begin{align*}
HS\left(\EuScript R(f^N(x),f^N(y),f^N(\bar x),f^N(\bar y))\right)&= d_i^f(f^N(x),f^N(y))\\
&\le C\beta_i^N\bar\varepsilon\le C\beta_i^N\varepsilon^\delta\\&\le
C\beta_i^N\left(\frac{C}{\tilde \beta_m^N}\right)^\delta=
C^{1+\delta}\left(\frac{\beta_i}{\tilde\beta_m^\delta}\right)^N.
\end{align*}

Rather than continuing to look at the rectangle  $\EuScript
R(f^N(x),f^N(y),f^N(\bar x),f^N(\bar y))$, we will now consider the
rectangle $\tilde {\EuScript R}(f^N(x))$ with base point $f^N(x)$
and fixed vertical size $1/\beta_m$. Lemma~\ref{monotonicity},
together with the estimate on the vertical size of\, $\EuScript
R(f^N(x),f^N(y),f^N(\bar x),f^N(\bar y))$, guarantees that
horizontal size of\/ $\tilde {\EuScript R}(f^N(x))$ is less than
$C^{1+\delta}\left({\beta_i}/{\tilde\beta_m^\delta}\right)^N$ as
well.

Thus, for every $x\in f^N(X_{\varepsilon_n})$, the horizontal size
of\/ $\tilde{\EuScript R}(x)=\tilde{\EuScript R}(x,z, \tilde x,
\tilde z)$ is less than
$C^{1+\delta}\left({\beta_i}/{\tilde\beta_m^\delta}\right)^N$. Note
that $\left({\beta_i}/{\tilde\beta_m^\delta}\right)^N\to 0$ as
$n\to\infty$ since ${\beta_i}/{\tilde\beta_m^\delta}<1$ and
$N\to\infty$ as $n\to\infty$.

Let $X=\varlimsup_{n\to\infty}f^N(X_{\varepsilon_n})$. Since any
$x\in X$ also belongs to $f^N(X_{\varepsilon_n})$ with arbitrarily
large $N$ we conclude that $\tilde{\EuScript R}(x)$ has zero
horizontal size, \ie $x=z$. Hence by Lemma~\ref{monotonicity} we
conclude that the piece of\/ $U(x)$ from $x$ to $\tilde z$ lies
inside of\/ $W^f(i+1,m)(x)$.

It is a simple exercise in measure theory to show that
$$
\mu(X)\ge\varlimsup_{n\to\infty}\mu(f^N(X_{\varepsilon_n}))=\varlimsup_{n\to\infty}\mu(X_{\varepsilon_n})>0.
$$
Finally recall that $\mu$-almost every point is transitive,
($\overline{\{f^j(x),\;j\ge1\}}=\mathbb T^d$). Hence by taking a
transitive point $x\in X$ and applying a straightforward
approximation argument, we get that $\forall y$ $\;U(y)\subset
W^f(i+1,m)(y)$.

\smallskip
{\bfseries Case 2.} $\varlimsup\limits_{\varepsilon\to
0}\mu(X_\varepsilon)=0$.
\smallskip

In this case, the idea is to use the assumption above to find a
leaf\/ $U(x)$ which is ``flat", \ie arbitrarily close to
$W^f(i,m-1)(x)$.  Since the leaf\/ $U(x)$ has to ``feel" the measure
$\mu$, we need to take it together with a small neighborhood. The
choice of this neighborhood is done by multiple applications of the
pigeonhole principle.

Given a point $\bar y\in U(x)$, denote by $U_{x\bar y}$ the piece
of\/ $U(x)$ between $x$ and $\bar y$. As before, by $\EuScript
R(x,\bar y)$ we denote the rectangle spanned by $x$ and $\bar y$.
Recall that $HS(\EuScript R(x,\bar y))$ and $VS(\EuScript R(x,\bar
y))$ stand for the horizontal and vertical sizes of\/ $\EuScript
R(x,\bar y)$. We will also need to measure the sizes of\/ $U_{x\bar
y}$. Let $HS(U_{x\bar y})=HS(\EuScript R(x,\bar y))$ and
$VS(U_{x\bar y})=VS(\EuScript R(x,\bar y))$.


{\bfseries Iterating Pigeonhole Principle.} Divide $\mathbb T^d$
into finite number of tubes $\EuScript T_1$, $\EuScript T_2,\ldots$
$\EuScript T_q$ foliated by $U$ such that any connected component
of\/ $U(x)\cap\EuScript T_j$, $j=1,\ldots q$, has vertical size
between $S_0$ and $S_1$. The numbers $S_0$ and $S_1$ are fixed,
$0<S_0<S_1$. We also require every tube $\EuScript T_j$ to be
$W^f(i,m-1)$-foliated so it can be represented as
$$
\EuScript T_j=\bigcup_{y\in \text{Transv}}\text{Plaque}(y),
$$
where $\text{Transv}$ is a plaque of\/ $U$ and $\text{Plaque}(y)$
are plaques of\/ $W^f(i,m-1)$.

Given a small number $\tau>0$, we can find an $\varepsilon>0$ such
that $\mu(X_{\varepsilon})<\tau$. Then by the pigeonhole principle
we can choose a tube $\EuScript T_j$ such that $\mu(\EuScript
T_j)\neq 0$ and
$$
\frac{\mu(\EuScript T_j\cap X_\varepsilon)}{\mu(\EuScript
T_j)}<\tau.
$$

The tube $\EuScript T_j$ can be represented as $\EuScript
T_j=\bigcup_{z\in \hat{\EuScript T}_j}W(z)$, where $\hat{\EuScript
T}_j$ is a transversal to $W^f(i,m)$ and $W(z)$, $z\in
\hat{\EuScript T}_j$, are connected plaques of\/ $W^f(i,m)$. By
absolute continuity,
$$
\mu(\EuScript T_j)=\int\limits_{\hat{\EuScript
T}_j}d\hat\mu(z)\int\limits_{W(z)}d\mu_{W(z)},
$$
where $\hat\mu$ is the factor measure on $\hat{\EuScript T}_j$ and
$\mu_{W(z)}$ is the conditional measure on $W(z)$.

Applying the pigeonhole principle again, we choose $W=W(z)$ such
that
$$
\mu_W(W\cap X_\varepsilon)<\tau.
$$
Recall that $\mu_W(W)=1$ by definition of the conditional measure
and $\mu_W$ is equivalent to the induced Riemannian volume on $W$ by
the absolute continuity of\/ $W^f(i,m)$.

The plaque $W$ is subfoliated by plaques of\/ $U$ of sizes between
$S_0$ and $S_1$. Unfortunately, we do not know if\/ $U$ is
absolutely continuous with respect to $\mu_W$. So we construct a
finite partition of\/ $W$ into smaller plaques of\/ $W^f(i,m)$ which
are  thin $U$-foliated tubes.

To construct this partition, we switch to $h(W)$, which is a plaque
of\/ $W^g(i,m)$ subfoliated by the plaques of\/ $h(U)=V_m^g$. The
partition $\{\tilde{\EuScript T}_1, \tilde{\EuScript T}_2,\ldots
\tilde{\EuScript T}_p\}$ will consist of\/ $V_m^g$-tubes inside of\/
$h(W)$ that can be represented as
$$
\tilde{\EuScript T}_j=\bigcup_{z\in\hat{\EuScript
T}_j}V(z),\;\;j=1,\ldots p,
$$
where $\hat{\EuScript T}_j$ is a transversal to $V_m^g$ inside of\/
$h(W)$ and $V(z)$ are plaques of\/ $V_m^g$. For every $j=1,\ldots
p$, choose $z_j\in\hat{\EuScript T}$. Then the tube
$\tilde{\EuScript T}_j$ can also be represented as
$$
\tilde{\EuScript T}_j=\bigcup_{y\in V(z_j)}\tilde P_j(y),
$$
where $\tilde P_j(y)\subset W^g(i,m-1)(y)$ are connected plaques.

Recall that $V_m^g$ is Lipschitz inside of\/ $W^g(i,m)$. Hence for
any $\xi>0$ it is possible to find a partition $\{\tilde{\EuScript
T}_1, \tilde{\EuScript T}_2,\ldots \tilde{\EuScript T}_p\}$,
$p=p(\xi)$, such that
\begin{multline}
\label{tilde_tube_size} \forall j=1,\ldots p\;\;\forall y\in
V(z_j)\;\;\; \exists B_j(\tilde C_1\xi), B_j(\tilde
C_2\xi)\subset W^g(i,m-1)(y)\\
\;\;\;\text{such that}\;\;\;\; B_j(\tilde C_1\xi)\subset\tilde
P_j(y)\subset B_j(\tilde C_2\xi),
\end{multline}
where $B_j(\tilde C_1\xi)$ and  $B_j(\tilde C_2\xi)$ are balls
inside of\/ $(W^g(i,m-1)(y)$, induced Riemannian distance) of radii
$\tilde C_1\xi$ and $\tilde C_2\xi$ respectively. The constants
$\tilde C_1$ and $\tilde C_2$ are independent of\/ $\xi$. Since we
are working in a bounded plaque $h(W)$ they also do not depend on
any other choices but $S_1$.

In the sequel we will need to take $\xi$ to be much smaller than
$\varepsilon$.

Now we pool this partition back into a partition of\/ $W$.
$$
\{{\EuScript T}_1, {\EuScript T}_2,\ldots {\EuScript T}_p\}=\{
h^{-1}(\tilde{\EuScript T}_1), h^{-1}(\tilde{\EuScript T}_2),\ldots
h^{-1}(\tilde{\EuScript T}_p)\}. $$ Although we use the same
notation for this partition, it is clearly different from the
initial partition of\/ $\mathbb T^d$.

Each tube $\EuScript T_j$ can be represented as
\begin{equation}
\label{P_plaque} {\EuScript T_j}=\bigcup_{y\in U(h^{-1}(z_j))}
P_j(y),
\end{equation}
where $P_j(y)=h^{-1}(\tilde P_j(y))\subset W^f(i,m-1)(y)$.


\begin{figure}[htbp]
\begin{center}

\begin{picture}(0,0)%
\includegraphics{11_tubes.pstex}%
\end{picture}%
\setlength{\unitlength}{3947sp}%
\begingroup\makeatletter\ifx\SetFigFont\undefined%
\gdef\SetFigFont#1#2#3#4#5{%
\reset@font\fontsize{#1}{#2pt}%
\fontfamily{#3}\fontseries{#4}\fontshape{#5}%
\selectfont}%
\fi\endgroup%
\begin{picture}(3999,1495)(-236,-794)
\put(226,-736){\makebox(0,0)[lb]{\smash{{\SetFigFont{12}{14.4}{\rmdefault}{\mddefault}{\updefault}{\color[rgb]{0,0,0}$W$}%
}}}}
\put(2701,-736){\makebox(0,0)[lb]{\smash{{\SetFigFont{12}{14.4}{\rmdefault}{\mddefault}{\updefault}{\color[rgb]{0,0,0}$h(W)$}%
}}}}
\put(301, 14){\makebox(0,0)[lb]{\smash{{\SetFigFont{12}{14.4}{\rmdefault}{\mddefault}{\updefault}{\color[rgb]{0,0,0}$\EuScript T_1$}%
}}}}
\put(826, 14){\makebox(0,0)[lb]{\smash{{\SetFigFont{12}{14.4}{\rmdefault}{\mddefault}{\updefault}{\color[rgb]{0,0,0}$\EuScript T_j$}%
}}}}
\put(2401,314){\makebox(0,0)[lb]{\smash{{\SetFigFont{12}{14.4}{\rmdefault}{\mddefault}{\updefault}{\color[rgb]{0,0,0}$\tilde{\EuScript T}_1$}%
}}}}
\put(3151,314){\makebox(0,0)[lb]{\smash{{\SetFigFont{12}{14.4}{\rmdefault}{\mddefault}{\updefault}{\color[rgb]{0,0,0}$\tilde{\EuScript T}_j$}%
}}}}
\put(3751,314){\makebox(0,0)[lb]{\smash{{\SetFigFont{12}{14.4}{\rmdefault}{\mddefault}{\updefault}{\color[rgb]{0,0,0}$\tilde{\EuScript T}_p$}%
}}}}
\put(1693,-169){\makebox(0,0)[lb]{\smash{{\SetFigFont{12}{14.4}{\rmdefault}{\mddefault}{\updefault}{\color[rgb]{0,0,0}$h$}%
}}}}
\put(1426, 71){\makebox(0,0)[lb]{\smash{{\SetFigFont{12}{14.4}{\rmdefault}{\mddefault}{\updefault}{\color[rgb]{0,0,0}$\EuScript T_p$}%
}}}}
\end{picture}%

\end{center}
\caption{We construct the partition $\{{\EuScript T}_1, {\EuScript
T}_2,\ldots {\EuScript T}_p\}$ as a pullback of the partition of\/
$h(W)$ by $V_m^g$-tubes. The foliation $V_m^g$ is Lipschitz and $h$
is continuously differentiable along $W^f(i,m-1)$. This guarantees
that the ``width" of a tube $\EuScript T_j$ is of the same order as
we move along $\EuScript T_j$~(\ref{tube_size}). Hence $\mu_W$ is
``uniformly distributed" along $\EuScript T_j$.}\label{11_tubes}
\end{figure}


By Lemma~\ref{step2}, $h$ is continuously differentiable along
$W^f(i,m-1)$. Moreover, the derivative depends continuously on the
points in $W$. Hence property~(\ref{tilde_tube_size}) persists:
\begin{multline}
\label{tube_size} \forall j=1,\ldots p\;\;\forall y\in
U(h^{-1}(z_j))\;\;\; \exists B_j(C_1\xi), B_j(
C_2\xi)\subset W^f(i,m-1)(y)\\
\;\;\;\text{such that}\;\;\;\; B_j(C_1\xi)\subset P_j(y)\subset
B_j(C_2\xi).
\end{multline}
The constants $C_1$ and $C_2$ differ from $\tilde C_1$ and $\tilde
C_2$ by a finite factor due to the bounded distortion along
$W^f(i,m-1)$ by the differential of\/ $h$.

Applying the pigeonhole principle for the last time, we find
$\EuScript T\in \{{\EuScript T}_1, {\EuScript T}_2,\ldots {\EuScript
T}_p\}$ such that
\begin{equation}
\label{tau_measure} \frac{\mu_W(\EuScript T\cap
X_\varepsilon)}{\mu_W(\EuScript T)}<\tau.
\end{equation}
Take a plaque $U_{x\bar y}$ inside of\/ $\EuScript T$. By
construction,
$$
S_0<VS(U_{x\bar y})<S_1.
$$


{\bfseries Estimating horizontal size of\/ $U_{x\bar y}$ from
below.} We have constructed $U_{x\bar y}$ such that a lot of points
in the neighborhood of\/ $U_{x\bar y}$ $\EuScript T$ lie outside
of\/ $X_\varepsilon$. The corresponding $\varepsilon$-rectangles
$\EuScript R(x)$ have vertical size greater than
$\varepsilon^\delta$. It is clear that we can use this fact to show
that $VS(U_{x\bar y})$ is large.

Choose a sequence $\{x_0=x, x_1,\ldots x_N\}\subset U_{x\bar y}$
such that
$$
VS(\EuScript R(x_0,x_N))\ge S_0 \;\;\; \text{and} \;\;\;
VS(\EuScript R(x_j,x_{j+1}))=\varepsilon,\; j=0,\ldots N-1.
$$

First we estimate the number of rectangles $N$.
\begin{lemma}
\label{holder_holonomy} The holonomy map $\text{Hol}\colon T(a)\to
T(b)$, $b\in W^f(i,m)(a)$, $T(a)\subset V_m^f(a)$, $T(b)\subset
V_m^f(b)$, along $W^f(i,m-1)$ is H\"older-continuous with exponent
$$
\rho\stackrel{\mathrm{def}}{=}\frac{\log \tilde\beta_m}{\log
\beta_m}.
$$
\end{lemma}
We postpone the proof until the end of the current section.

Let $\tilde x_j=W^f(i,m-1)(x_j)\cap V_m^f(x_0)$, $j=0,\ldots N$.
Then according to the lemma above,
$$
d_m^f(\tilde x_{j-1},\tilde x_j)\le C_{\text{Hol}} VS(\EuScript
R(x_{j-1},x_j))^\rho=C_{\text{Hol}}\varepsilon^\rho,\;\;\;
j=1,\ldots N,
$$
which allows us to estimate $N$
$$
S_0\le VS(\EuScript R(x_0,x_N))=\sum_{j=1}^N d_m^f(\tilde
x_{j-1},\tilde x_j)\le NC_{\text{Hol}}\varepsilon^\rho.
$$
Hence
\begin{equation}
\label{n1_estimate} N\ge \frac{S_0}{C_{Hol}\varepsilon^\rho}.
\end{equation}

Along with the rectangles $\EuScript R(x_j,x_{j+1})$, let us
consider sets $A(x_j,x_{j+1})\subset\EuScript T$, $j=0,\ldots N-1$,
given by the formula
$$
A(x_j,x_{j+1})=\bigcup_{y\in U_{x_jx_{j+1}}} P(y),
$$
where $P(y)$ are the plaques of\/ $W^f(i,m-1)$ from the
representation~(\ref{P_plaque}) for $\EuScript T$. The sets
$A(x_j,x_{j+1})$ have the same vertical size. The following property
of these sets is a direct consequence of~(\ref{tube_size}) and the
fact that $\mu_W$ is equivalent to the Riemannian volume on $W$.
\begin{equation}
\label{volume_distortion} \exists C_{\text{univ}}\;\;\;\text{such
that}\;\;\; \forall j,\tilde j=1,\ldots N-1\;\;
\frac{1}{C_{\text{univ}}}<\frac{\mu_W(A(x_j,x_{j+1}))}{\mu_W(A(x_{\tilde
j},x_{\tilde j+1}))}<C_{\text{univ}}.
\end{equation}
The constant $C_{\text{univ}}$ depends on $C_1$, $C_2$ and size of\/
$W$, but is independent of\/ $\varepsilon$ and~$\xi$.

Let
$$
A_1=\bigcup_{\substack{j=1\\ j\;\; \text{\scriptsize is
odd}}}^{N-1}A(x_j,x_{j+1})\;\;\;\text{and}\;\;\;
A_2=\bigcup_{\substack{j=1\\ j\;\; \text{\scriptsize is
even}}}^{N-1}A(x_j,x_{j+1})
$$
It follows from~(\ref{tau_measure}) that either
$$
\frac{\mu_W(A_1\cap
X_\varepsilon)}{\mu_W(A_1)}<\tau\;\;\;\text{or}\;\;\;\frac{\mu_W(A_2\cap
X_\varepsilon)}{\mu_W(A_2)}<\tau.
$$
For concreteness, assume that the first possibility holds.

The bounds~(\ref{volume_distortion}) allow us to estimate the number
$N_1$ of sets $A(x_j,x_{j+1})\subset A_1$ that have a point $q_j\in
A(x_j,x_{j+1})$ such that $q_j\notin X_\varepsilon$.
$$
N_1\ge\left\lfloor\frac {N}{2}\right\rfloor-\lfloor
C_{\text{univ}}\tau N\rfloor.
$$
Here $\lfloor N/2\rfloor$ is the total number of sets
$A(x_j,x_{j+1})$ in $A_1$ and $\lfloor C_{\text{univ}}\tau N\rfloor$
is the maximal possible number of sets $A(x_j,x_{j+1})$ in $A_1\cap
X_\varepsilon$. Clearly we can choose $\tau$ and $\varepsilon$
accordingly so $N_1\ge N/3$.

For every $A(x_j,x_{j+1})$ as above, fix $q_j\in A(x_j,x_{j+1})$,
$q_j\notin X_\varepsilon$, and consider rectangle $\EuScript R(q_j)$
of vertical size $\varepsilon$. Then
$$
HS(\EuScript R(q_j))\ge\varepsilon^\delta.
$$

Consider two rectangles $\EuScript R(q_j)$ and $\EuScript
R(q_{\tilde j})$ as above. Since $|j-\tilde j|\ge 2$, they do not
``overlap" vertically if\/ $\xi$ is sufficiently small (although
this is not important to us). They might happen to ``overlap"
horizontally as shown on the Figure~\ref{12_ladder_hard} but the
size of the overlap cannot exceed the diameter of the tube
$\EuScript T$, which, according to~(\ref{tube_size}), is bounded by
$C_2\xi$.

The above considerations result in the following estimate:
\begin{align}
\label{hz_estimate}
HS(U_{x\bar y})&\ge HS(U_{x_0x_N})\ge \frac{1}{C_H}\sum_{j=1}^{N_1} HS(\EuScript R(q_j))-C_HN_1C_2\xi\\
&\ge\frac{1}{C_H} N_1 \varepsilon^\delta-C_HNC_2\xi \ge\frac
{N}{3C_H}\varepsilon^\delta-NC_HC_2\xi\nonumber\\
&\ge\frac{S_0}{3
C_HC_{Hol}}\varepsilon^{\delta-\rho}-NC_HC_2\xi,\nonumber
\end{align}
where $C_H$ is the Lipschitz constant of the holonomy map along
$W^f(i+1,m)$. We used estimate on $N_1$ and
estimate~(\ref{n1_estimate}) on $N$.


\begin{figure}[htbp]
\begin{center}

\begin{picture}(0,0)%
\includegraphics{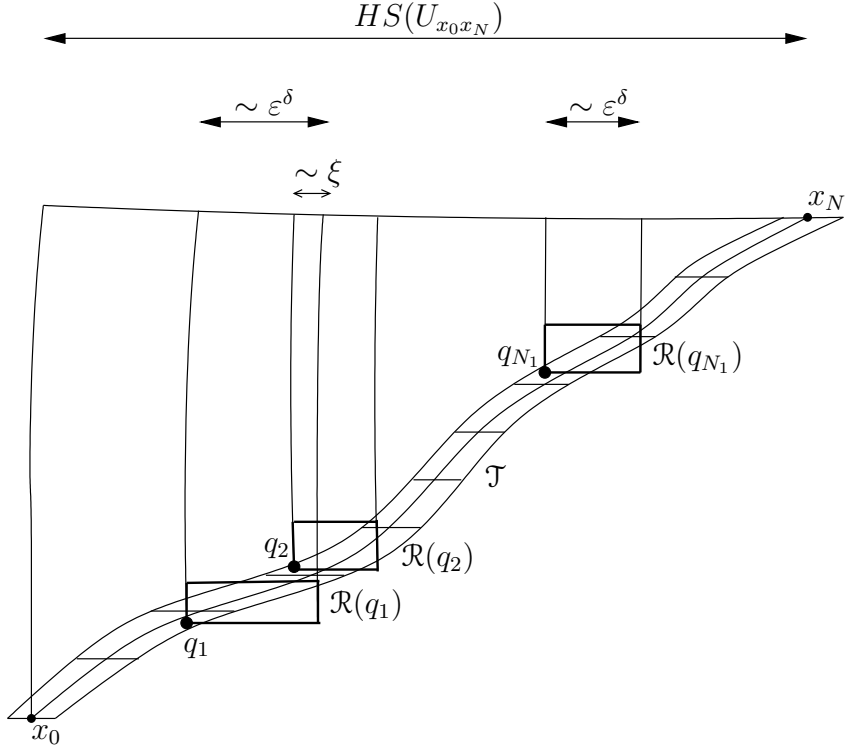}%
\end{picture}%
\setlength{\unitlength}{3947sp}%
\begingroup\makeatletter\ifx\SetFigFont\undefined%
\gdef\SetFigFont#1#2#3#4#5{%
\reset@font\fontsize{#1}{#2pt}%
\fontfamily{#3}\fontseries{#4}\fontshape{#5}%
\selectfont}%
\fi\endgroup%
\begin{picture}(5544,4696)(-86,-2051)
\put(1726,1514){\makebox(0,0)[lb]{\smash{{\SetFigFont{12}{14.4}{\rmdefault}{\mddefault}{\updefault}{\color[rgb]{0,0,0}$\sim\xi$}%
}}}}
\put(4951,1364){\makebox(0,0)[lb]{\smash{{\SetFigFont{12}{14.4}{\rmdefault}{\mddefault}{\updefault}{\color[rgb]{0,0,0}$x_N$}%
}}}}
\put(2101,2489){\makebox(0,0)[lb]{\smash{{\SetFigFont{12}{14.4}{\rmdefault}{\mddefault}{\updefault}{\color[rgb]{0,0,0}$HS(U_{x_0x_N})$}%
}}}}
\put(1351,1931){\makebox(0,0)[lb]{\smash{{\SetFigFont{12}{14.4}{\rmdefault}{\mddefault}{\updefault}{\color[rgb]{0,0,0}$\sim\varepsilon^\delta$}%
}}}}
\put(3451,1931){\makebox(0,0)[lb]{\smash{{\SetFigFont{12}{14.4}{\rmdefault}{\mddefault}{\updefault}{\color[rgb]{0,0,0}$\sim\varepsilon^\delta$}%
}}}}
\put(3001,395){\makebox(0,0)[lb]{\smash{{\SetFigFont{12}{14.4}{\rmdefault}{\mddefault}{\updefault}{\color[rgb]{0,0,0}$q_{N_1}$}%
}}}}
\put(3976,359){\makebox(0,0)[lb]{\smash{{\SetFigFont{12}{14.4}{\rmdefault}{\mddefault}{\updefault}{\color[rgb]{0,0,0}$\EuScript R(q_{N_1})$}%
}}}}
\put(1951,-1201){\makebox(0,0)[lb]{\smash{{\SetFigFont{12}{14.4}{\rmdefault}{\mddefault}{\updefault}{\color[rgb]{0,0,0}$\EuScript R(q_1)$}%
}}}}
\put(1033,-1441){\makebox(0,0)[lb]{\smash{{\SetFigFont{12}{14.4}{\rmdefault}{\mddefault}{\updefault}{\color[rgb]{0,0,0}$q_1$}%
}}}}
\put( 76,-1993){\makebox(0,0)[lb]{\smash{{\SetFigFont{12}{14.4}{\rmdefault}{\mddefault}{\updefault}{\color[rgb]{0,0,0}$x_0$}%
}}}}
\put(1537,-817){\makebox(0,0)[lb]{\smash{{\SetFigFont{12}{14.4}{\rmdefault}{\mddefault}{\updefault}{\color[rgb]{0,0,0}$q_2$}%
}}}}
\put(2401,-901){\makebox(0,0)[lb]{\smash{{\SetFigFont{12}{14.4}{\rmdefault}{\mddefault}{\updefault}{\color[rgb]{0,0,0}$\EuScript R(q_2)$}%
}}}}
\put(2926,-409){\makebox(0,0)[lb]{\smash{{\SetFigFont{12}{14.4}{\rmdefault}{\mddefault}{\updefault}{\color[rgb]{0,0,0}$\EuScript T$}%
}}}}
\end{picture}%

\end{center}
\caption{This picture illustrates the key
estimate~(\ref{hz_estimate}). Since the holonomy along $W^f(i+1,m)$
is Lipschitz, the horizontal size of\/ $U_{x_0x_N}$ can be estimated
from below by the sum of horizontal sizes of ``flat" rectangles with
base points $q_j\in A_1\subset\EuScript T$, $j=1,\ldots N_1$. They
might overlap horizontally as shown, but the overlap is of order
$\xi\lll\varepsilon$.}\label{12_ladder_hard}
\end{figure}


Finally, recall that $\delta-\rho<0$, while $\xi$ can be chosen
arbitrarily small independently of\/ $\varepsilon$ (and hence $N$).
Hence by choosing $\varepsilon$ small, we can find $U_{x\bar y}$
with arbitrarily big horizontal size, which contradicts to the
uniform upper bound~(\ref{bounded_shift}) that follows from
compactness. Hence Case 2 is impossible.
\end{proof}

\begin{remark}
Note that we do not need to take $\tau$ arbitrarily small. The
constant $\tau$ just needs to be small enough to provide the
estimate on $N_1$.
\end{remark}

\begin{proof}[Proof of Lemma~\ref{holder_holonomy}]
Take points $x$ and $y\in V_m^f(x)$ such that
\begin{equation}
\label{zvezda} 1\le d_m^f(x,y)\le C\beta_m
\end{equation}
By Lemma~\ref{distance}, there exist constants $c_1$ and $c_2$ such
that
\begin{multline}
\label{lemma_distance} \forall \tilde x,\tilde y, \;\;\tilde y\in
V_m^f(\tilde x), \;\;\tilde x\in W^f(i,m-1)(x), \;\;\tilde y\in
W^f(i,m-1)(y)\;\;\;\;\\
c_1<d_m^f(\tilde x,\tilde y)<c_2.
\end{multline}
Moreover, since $c_1$ and $c_2$ depend only on $\hat d(x,y)$ (see
the remark after the proof of Lemma~\ref{distance}), they can be
chosen independently of\/ $x$ and $y$ as long as $x$ and $y$
satisfy~(\ref{zvezda}).

Take $x, y\in T(a)$ close to each other. Let $N$ be the smallest
integer such that $d_m^f(f^N(x),f^N(y))\ge 1$. Then
\begin{equation}
\label{11} d_m^f(f^N(x),f^N(y))\ge \frac 1C \tilde\beta_m^N
d_m^f(x,y),
\end{equation}
and, obviously,
\begin{equation}
\label{22} d_m^f(f^N(x),f^N(y))\le C\beta_m.
\end{equation}
Hence by taking in~(\ref{lemma_distance}) $\tilde
x=f^N(\text{Hol}(x))$  and $\tilde y=f^N(\text{Hol}(y))$, we get
\begin{equation}
\label{33}
d_m^f\left(f^N(\text{Hol}(x)),f^N(\text{Hol}(y))\right)>c_1.
\end{equation}
On the other hand,
\begin{equation}
\label{44}
d_m^f\left(f^N(\text{Hol}(x)),f^N(\text{Hol}(y))\right)\le
C\beta_m^N d_m^f(\text{Hol}(x),\text{Hol}(y)).
\end{equation}

Combining~(\ref{11}), (\ref{22}), (\ref{33}) and~(\ref{44}), we
finish the proof
\begin{align*}
d_m^f(x,y)&\le \frac{C}{\tilde\beta_m^N}d_m^f(f^N(x),f^N(y))\le\frac{C^2\beta_m}{c_1^\rho\tilde\beta_m^N}\cdot c_1^\rho\\
&<\frac{C^2\beta_m}{c_1^\rho}\cdot\frac
{1}{\tilde\beta_m^N}d_m^f\left(f^N(\text{Hol}(x)),f^N(\text{Hol}(y))\right)^\rho\\&\le
C_{\text{Hol}}\frac{\beta_m^{\rho N}}{\tilde\beta_m^N}d_m^f(\text{Hol}(x),\text{Hol}(y))^\rho\\
&=C_{\text{Hol}}d_m^f(\text{Hol}(x),\text{Hol}(y))^\rho.
\end{align*}
We used~(\ref{exponents}) for the last equality.
\end{proof}

\section{Proof of Theorem C}
\subsection{Scheme of the proof of Theorem C}
\label{schemeC}

The way we choose the neighborhood $\mathcal U$ is the same as in
Theorem A. We look at the $L$-invariant splitting
$$
T\mathbb T^4=E_L^{ss}\oplus E_L^{ws}\oplus E_L^{wu}\oplus E_L^{su},
$$
where $E_L^{ws}$, $E_L^{wu}$ are eigendirections with eigenvalues
$\lambda^{-1}<\lambda$ and $E_L^{ss}\oplus E_L^{su}$ is the Anosov
splitting of\/ $g$. We choose $\mathcal U$ in such a way that for
any $f\in \mathcal U$ the invariant splitting survives,
\begin{equation}
\label{T4splitting} T\mathbb T^4=E_f^{ss}\oplus E_f^{ws}\oplus
E_f^{wu}\oplus E_f^{su},
\end{equation}
with
\begin{equation}
\label{angles2} \max_{x\in\mathbb T^4, \sigma=ss, ws, wu,
su}\left(\measuredangle(E_f^{\sigma}(x),E_L^{\sigma}(x))\right)<\frac\pi
2
\end{equation}
and $f$ is partially hyperbolic in the strongest sense~(\ref{phd})
with respect to the splitting~(\ref{T4splitting}).

Lemma~\ref{integr1}  works for $f\in \mathcal U$. Hence the
distributions $E_f^{ss}$, $E_f^{ws}$, $E_f^{wu}$ and $E_f^{su}$
integrate uniquely to foliations $W_f^{ss}$, $W_f^{ws}$, $W_f^{wu}$
and $W_f^{su}$. Also, as usual, $W_f^s$ and $W_f^u$ stand for
two-dimensional stable and unstable foliations.

Fix $f\in \mathcal U$ and let $H$ be the conjugacy with the model,
$H\circ f=L\circ H$. The distribution $E_L^{ws}\oplus E_L^{wu}$
obviously integrates to the foliation $W_L^c$, which is subfoliated
by $W_L^{ws}$ and $W_L^{wu}$. Applying Lemma~\ref{weak_match} to the
weak foliations, we get $H(W_f^{ws})=W_L^{ws}$ and
$H(W_f^{wu})=W_L^{wu}$. Hence the distribution  $E_f^{ws}\oplus
E_f^{wu}$ integrates to the foliation $W_f^c$, which is subfoliated
by $W_f^{ws}$ and $W_f^{wu}$.

Note that the leaves of\/ $W_f^c$ are embedded two-dimensional tori.

\begin{lemma}
\label{HisC1} The conjugacy $H$ is\, $C^{1+\nu}$along $W_f^{ws}$ and
$W_f^{wu}$. Hence, by the Regularity Lemma, $H$ is $C^{1+\nu}$ along
$W_f^c$.
\end{lemma}
Proposition~\ref{central_smoothness} is a more general statement
which we prove in Section~\ref{moduli}. So we omit the proof of
Lemma~\ref{HisC1} here.

We establish smoothness of central holonomies.
\begin{lemma}
\label{central_holonomy} Let $T_1$ and $T_2$ be open
$C^{1+\nu}$-disks transverse to $W_f^c$. Then the holonomy map along
$W_f^c$, $H_f^c\colon T_1\to T_2$, is $C^{1+\nu}$-differentiable.
\end{lemma}

Next we introduce a distance on the leaves of\/  $W_f^{ws}$ and
$W_f^{wu}$ by simply letting $d^\sigma(x, y)=d^\sigma(H(x),H(y))$,
$y\in W_f^\sigma(x)$, $\sigma=ws, wu$. Notice that by
Lemma~\ref{HisC1}, $d^{ws}$ and $d^{wu}$ are induced by a
H\"older-continuous Riemannian metric --- the pullback by
$DH^{-1}|_{W_L^c}$ of the Riemannian metric on $W_L^c$.

Let $x_0$ be the fixed point of\/ $f$ and let $S_0$ be the
two-dimensional torus passing through $x_0$ and tangent to
$E_L^{ss}\oplus E_L^{su}$. Assumption~(\ref{angles2}) guarantees
that $S_0$ is transverse to $W_f^c$.

Now we construct a foliation $S$ that is transverse to $W_f^c$. For
any point $x\in\mathbb T^4$ let $x_1=W_f^c(x)\cap S_0$ and $x_2$ be
some point of intersection of\/ $W_f^{ws}(x_1)$ and $W_f^{wu}(x)$.
Fix $\tilde x\in \mathbb T^4$ and define
\begin{multline*}
S(\tilde x)=\{x: \text{such that}
(x_1, x_2)\;\; \text{and} \;\;(\tilde x_1, \tilde x_2)\;\; \text{have the same orientation in}\;\; W_f^{ws};\\
(x_2, x)\;\; \text{and} \;\;(\tilde x_2, x) \;\;\text{have the same orientation in}\;\; W_f^{wu};\\
d^{ws}(x_1, x_2)=d^{ws}(\tilde x_1, \tilde x_2); \;\; d^{wu}(x_2,
x)=d^{wu}(\tilde x_2, \tilde x) \}.
\end{multline*}


\begin{figure}[htbp]
\begin{center}

\begin{picture}(0,0)%
\includegraphics{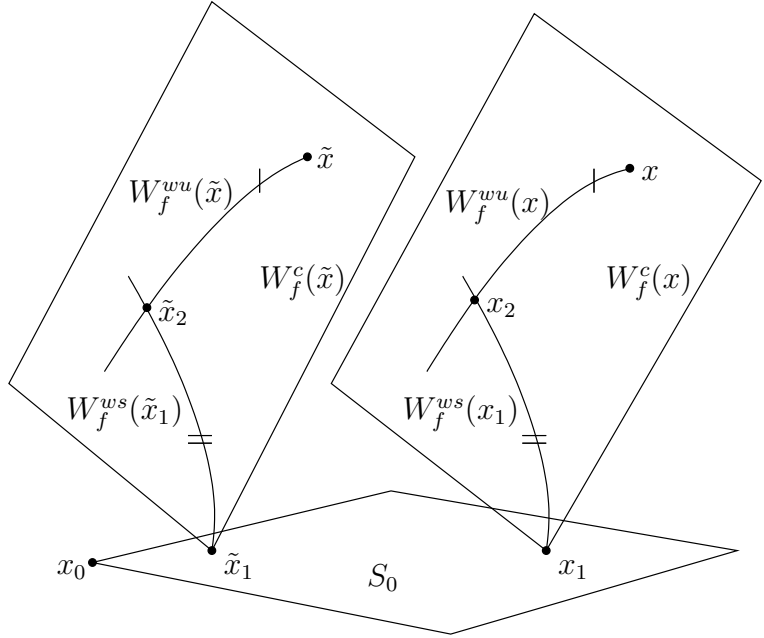}%
\end{picture}%
\setlength{\unitlength}{3947sp}%
\begingroup\makeatletter\ifx\SetFigFont\undefined%
\gdef\SetFigFont#1#2#3#4#5{%
\reset@font\fontsize{#1}{#2pt}%
\fontfamily{#3}\fontseries{#4}\fontshape{#5}%
\selectfont}%
\fi\endgroup%
\begin{picture}(4749,3999)(-86,-2773)
\put(226,-2386){\makebox(0,0)[lb]{\smash{{\SetFigFont{12}{14.4}{\rmdefault}{\mddefault}{\updefault}{\color[rgb]{0,0,0}$x_0$}%
}}}}
\put(2176,-2461){\makebox(0,0)[lb]{\smash{{\SetFigFont{12}{14.4}{\rmdefault}{\mddefault}{\updefault}{\color[rgb]{0,0,0}$S_0$}%
}}}}
\put(1276,-2386){\makebox(0,0)[lb]{\smash{{\SetFigFont{12}{14.4}{\rmdefault}{\mddefault}{\updefault}{\color[rgb]{0,0,0}$\tilde x_1$}%
}}}}
\put(3376,-2386){\makebox(0,0)[lb]{\smash{{\SetFigFont{12}{14.4}{\rmdefault}{\mddefault}{\updefault}{\color[rgb]{0,0,0}$x_1$}%
}}}}
\put(3901, 89){\makebox(0,0)[lb]{\smash{{\SetFigFont{12}{14.4}{\rmdefault}{\mddefault}{\updefault}{\color[rgb]{0,0,0}$x$}%
}}}}
\put(3676,-586){\makebox(0,0)[lb]{\smash{{\SetFigFont{12}{14.4}{\rmdefault}{\mddefault}{\updefault}{\color[rgb]{0,0,0}$W_f^c(x)$}%
}}}}
\put(1501,-586){\makebox(0,0)[lb]{\smash{{\SetFigFont{12}{14.4}{\rmdefault}{\mddefault}{\updefault}{\color[rgb]{0,0,0}$W_f^c(\tilde x)$}%
}}}}
\put(2926,-736){\makebox(0,0)[lb]{\smash{{\SetFigFont{12}{14.4}{\rmdefault}{\mddefault}{\updefault}{\color[rgb]{0,0,0}$x_2$}%
}}}}
\put(865,-793){\makebox(0,0)[lb]{\smash{{\SetFigFont{12}{14.4}{\rmdefault}{\mddefault}{\updefault}{\color[rgb]{0,0,0}$\tilde x_2$}%
}}}}
\put(1861,167){\makebox(0,0)[lb]{\smash{{\SetFigFont{12}{14.4}{\rmdefault}{\mddefault}{\updefault}{\color[rgb]{0,0,0}$\tilde x$}%
}}}}
\put(289,-1411){\makebox(0,0)[lb]{\smash{{\SetFigFont{12}{14.4}{\rmdefault}{\mddefault}{\updefault}{\color[rgb]{0,0,0}$W_f^{ws}(\tilde x_1)$}%
}}}}
\put(2401,-1411){\makebox(0,0)[lb]{\smash{{\SetFigFont{12}{14.4}{\rmdefault}{\mddefault}{\updefault}{\color[rgb]{0,0,0}$W_f^{ws}(x_1)$}%
}}}}
\put(676,-49){\makebox(0,0)[lb]{\smash{{\SetFigFont{12}{14.4}{\rmdefault}{\mddefault}{\updefault}{\color[rgb]{0,0,0}$W_f^{wu}(\tilde x)$}%
}}}}
\put(2665,-109){\makebox(0,0)[lb]{\smash{{\SetFigFont{12}{14.4}{\rmdefault}{\mddefault}{\updefault}{\color[rgb]{0,0,0}$W_f^{wu}(x)$}%
}}}}
\end{picture}%

\end{center}
\caption{Definition of\/ $S$. Point $x\in S(\tilde x)$.}\label{13_S}
\end{figure}


According to this definition, $S(\tilde x)$ intersects each leaf
of\/ $W_f^c$ exactly once. Also note that, since the distances came
from the model $L$, the definition above does not depend on the
choice of\/ $\tilde x_2$. It is clear that $S$ is a topological
foliation into topological two-dimensional tori. We show that these
tori are in fact regular.
\begin{lemma}
\label{SisC1} Leaves of\/ $S$ are $C^{1+\nu}$-embedded
two-dimensional tori.
\end{lemma}

Let $f_0\colon S_0\to S_0$ be the factor map of\/ $f$,
$f_0(x)=W_f^c(f(x))\cap S_0$. Lemma~\ref{central_holonomy}
guarantees that $f_0$ is a $C^{1+\nu}$-diffeomorphism. Every
periodic point of\/ $f_0$ lifts to a periodic point of\/ $f$.
Applying Lemma~\ref{central_holonomy} again, we see that the p.~d.
of\/ $f_0$ are the same as the strong stable and unstable p.~d. of\/
$f$ which is the same as the p.~d. of\/ $g$. Hence there is a
$C^{1+\nu}$-diffeomorphism $h_0$ homotopic to identity such that
$h_0\circ f_0=g\circ h_0$.

Let $f_c\colon W_f^c(x_0)\to W_f^c(x_0)$ be the restriction of\/ $f$
to $W_f^c(x_0)$. Obviously the p.~d. of\/ $f_c$ and $A$ are the
same. Hence there is a $C^{1+\nu}$-diffeomorphism $h_c$ homotopic to
identity such that $h_c\circ f_c=A\circ h_c$.

We are ready to construct the conjugacy $h\colon\mathbb T^4\to
\mathbb T^2\times\mathbb T^2$.
$$
h(x)=\left(h_c(S(x)\cap W_f^c(x_0)), h_0(W_f^c(x)\cap S_0)\right).
$$
The homeomorphism $h$ maps the central foliation into the vertical
foliation and the foliation $S$ into the horizontal foliation.

\begin{remark}
Notice that at this point we do not know if\/ $h$ is a
$C^{1+\nu}$-diffeomorphism, although $h_c$ and $h_0$ are
$C^{1+\nu}$-differentiable.
\end{remark}

\begin{lemma}
\label{hisWC1} Homeomorphism $h$ is $C^{1+\nu}$-differentiable along
$W_f^c$.
\end{lemma}
\begin{proof}
The projection $x\mapsto S(x)\cap
W_f^c(x_0)\stackrel{\mathrm{def}}{=}pr(x)$ projects a weak stable
leaf\/ $W_f^{ws}(x)$ into $W_f^{ws}(pr(x))$. Moreover, it is clear
from the definition of\/ $S$ that the restriction of this projection
to $W_f^{ws}(x)$ is an isometry with respect to the distance
$d^{ws}$. According to the formula for the first component of\/ $h$,
we compose this projection with $h_c$, which is an isometry when
restricted to the leaf\/ $W_f^{ws}(pr(x))$ by the definition of\/
$d^{ws}$. The diffeomorphism $h_c$ straightens the weak stable
foliation into a foliation by straight lines $W_L^{ws}$. Hence
$h(W_f^{ws})=W_L^{ws}$ and $h$ is an isometry as a map
$(W_f^{ws}(x), d^{ws})\mapsto (W_L^{ws}(h(x)), \;\text{Riemannian
metric})$. Thus $h$ is $C^{1+\nu}$ along $W_f^{ws}$.

Everything above can be repeated for the weak unstable foliation.
Applying the Regularity Lemma, we get the desired statement.
\end{proof}

\begin{lemma}
\label{hisSC1} The homeomorphism $h$ is $C^{1+\nu}$-differentiable
along $S$.
\end{lemma}

\begin{proof}
The restriction of\/ $h$ to $S_0$ is just $h_0$. The restriction
of\/ $h$ to some other leaf\/ $S(x)$ can be viewed as composition of
the holonomy $H_f^c$, $h_0$ and the holonomy $H_L^c$. Hence this
restriction is $C^{1+\nu}$-differentiable as well. We need to make
sure that the derivative of\/ $h$ along $S$ is H\"older-continuous
on $\mathbb T^4$. For this we need only show that the derivative
of\/ $H_f^c\colon S(x)\to S_0$ depends H\"older-continuously on $x$.
This assertion will become clear in the proof of Lemma~\ref{SisC1}.
\end{proof}
Now by the Regularity Lemma, we conclude that $h$ is a
$C^{1+\nu}$-diffeomorphism.

Let $\tilde L=h\circ f\circ h^{-1}$. Clearly the foliations
$W_L^{ws}$ and $W_L^{wu}$ are $\tilde L$-invariant. By construction,
$h$ and $h^{-1}$ are isometries when restricted to the leaves of the
weak foliations. Recall that $f$ stretches by a factor $\lambda$ the
distance $d^{wu}$ on $W_f^{wu}$ and contracts by a factor
$\lambda^{-1}$ the distance $d^{ws}$ on $W_f^{ws}$. Hence, if we
consider the restriction of\/ $\tilde L$ on a fixed vertical
two-torus $W_L^c(x)\mapsto W_L^c(\tilde L(x))$, then it acts by a
hyperbolic automorphism $A$.

Also, it is obvious from the construction of\/ $h$ that the factor
map of\/ $\tilde L$ on a horizontal two-torus is $g$. These
observations show that $\tilde L$ is of the form
$$
\tilde L=(Ax+\vec \varphi(y), g(y)). \eqno{(4)}
$$
Note that we do not have to additionally argue that $\vec\varphi$ is
smooth since we know that $\tilde L$ is a
$C^{1+\nu}$-diffeomorphism.

\begin{remark}
An observant reader would notice that our choice of\/ $h$ and hence
$\tilde L$ is far from being unique. The starting point of the
construction of\/ $h$ is the torus $S_0$. Although we have chosen a
concrete $S_0$, in fact, the only thing we need from $S_0$ is
transversality to $W_f^c$. This is not surprising. Many
diffeomorphisms of type~(\ref{tildeLgeneral}) are $C^1$-conjugate to
each other. In the linear case this is controlled by the
invariants~(\ref{condition}).
\end{remark}

In the rest of this section we prove Lemmas~\ref{central_holonomy}
and~\ref{SisC1}.

\subsection{A technical Lemma}

Before we proceed with proofs of Lemmas~\ref{central_holonomy}
and~\ref{SisC1}, we establish a crucial technical lemma which is a
corollary of Lemma~\ref{HisC1}.

Let $U^{\sigma}=H(W_f^{\sigma})$, $\sigma=ss, su$. These are
foliations by H\"older-continuous curves.

\begin{lemma}
\label{technical} Fix $x\in \mathbb T^4$ and $y\in W_L^c(x)$. Let
$\vec v$ be a vector connecting $x$ and $y$ inside of\/ $W_L^c(x)$.
Then
$$
U^{\sigma}(y)=U^{\sigma}(x)+\vec v.
$$
In other words, the foliation $U^{\sigma}$ is invariant under
translations along $W_L^c$, $\sigma=ss, su$.
\end{lemma}

\begin{proof}
For concreteness, we take $\sigma=ss$. The proof in the case where
$\sigma=su$ is the same.

First let us assume that $y\in W_L^{ws}(x)$. This allows us to
restrict our attention to the stable leaf\/ $W_L^s(x)$, since
$U^{ss}(x)$ and $U^{ss}(y)$ lie inside of\/ $W_L^s(x)$. Pick a point
$z\in U^{ss}(x)$ and let $\tilde z= W_L^{ws}(z)\cap U^{ss}(y)$. We
only need to show that $d(x,y)=d(z,\tilde z)$, where $d$ is the
Riemannian distance along weak stable leaves. The simple idea of the
proof of Claim 1 from~\cite{GG} works here. We briefly outline the
argument.

Let $c=d(z,\tilde z)/d(x,y)$. Obviously
\begin{equation}
\label{num1} \forall n \;\;\;\;\; \frac{d(L^n(z), L^n(\tilde
z))}{d(L^n(x), L^n(y))}=c.
\end{equation}
Since $H^{-1}(z)\in W_f^{ss}(x)$, $H^{-1}(\tilde z)\in W_f^{ss}(y)$,
and strong stable leaves contract exponentially faster than weak
stable leaves, we have
\begin{multline}
\label{num2} \forall \varepsilon>0\;\;\; \exists N:\forall
n>N:\;\;\left|\frac{d\left(H^{-1}(L^n(z)), H^{-1}(L^n(\tilde
z))\right)}{d\left(H^{-1}(L^n(x)),
H^{-1}(L^n(y))\right)}-1\right|\\=
\left|\frac{d\left(f^{n}(H^{-1}(z)), f^{n}(H^{-1}(\tilde
z))\right)}{d\left(f^{n}(H^{-1}(x)),
f^{n}(H^{-1}(y))\right)}-1\right|<\varepsilon.
\end{multline}
On the other hand, since the derivative of\/ $H$ along $W_f^{ws}$ is
continuous, the ratios
$$
\frac{d(L^n(z), L^n(\tilde z))}{d\left(H^{-1}(L^n(z)),
H^{-1}(L^n(\tilde z))\right)}\quad\text{and}\quad\frac{d(L^n(x),
L^n(y))}{d\left(H^{-1}(L^n(x)), H^{-1}(L^n(y))\right)}
$$
are arbitrarily close when $n\to+\infty$. Together
with~(\ref{num2}), this shows that the constant $c$
from~(\ref{num1}) is arbitrarily close to 1. Hence $c=1$.

Finally, recall that for any $x$ the leaf\/ $W_L^{ws}(x)$ is dense
in $W_L^c(x)$. Hence by continuity, we get the statement of the
lemma for any $y\in W_L^c(x)$.
\end{proof}
Lemma~\ref{technical} leads to some nontrivial structural
information about $f$ which is of interest on its own.
\begin{proposition}
\label{ws_su_integrability} The distributions $E_f^{wu}\oplus
E_f^{ss}$ and $E_f^{ws}\oplus E_f^{su}$ are integrable.
\end{proposition}
\begin{proof}
It follows from Lemma~\ref{technical} that the foliations $W_L^{wu}$
and $U^{ss}$ integrate together. Thus the foliations $W_f^{wu}$ and
$W_f^{ss}$ integrate to a foliation with tangent distribution
$E_f^{wu}\oplus E_f^{ss}$.
\end{proof}


\subsection{Smoothness of central holonomies}
We assume that the holonomy map $H_f^c\colon T_1\to T_2$ is a
bijection. It can be represented as a composition of holonomies
along $W_f^{ws}$ and $W_f^{wu}$. Indeed, let us work on the
universal cover and consider two open three-dimensional submanifolds
of\/ $\mathbb R^4$: $\;M_1=\bigcup_{x\in T_1}W_f^{wu}(x)$ and
$M_2=\bigcup_{x\in T_2}W_f^{ws}(x)$. Let $T_3=M_1\cap M_2$.
Obviously, $T_3$ is a smooth two-dimensional open submanifold. Also,
it is easy to see that $T_3$ is connected since we are working on
the universal cover. Then $H_f^c\colon T_1\to T_2$ is the
composition of\/ $H_f^{wu}\colon T_1\to T_3$ and $H_f^{ws}\colon
T_3\to T_2$.

So, it is sufficient to study the holonomy map along $W_f^{wu}$
$H_f^{wu}\colon T_1\to T_2$. The study of holonomies along
$W_f^{ws}$ is the same.

First we make a reduction that will allow us to work with
one-dimensional transversals instead of two-dimensional
transversals. Let $\tilde W_f$ and $\tilde W_L$ be the integral
foliations of\/ $E_f^{ws}\oplus E_f^{wu}\oplus E_f^{su}$ and
$E_L^{ws}\oplus E_L^{wu}\oplus E_L^{su}$, respectively. Also, let
$\bar W_f$ and $\bar W_L$ be the integral foliations of\/
$E_f^{ss}\oplus E_f^{ws}\oplus E_f^{wu}$ and $E_L^{ss}\oplus
E_L^{ws}\oplus E_L^{wu}$, respectively.

Any transversal $T$ to $W_f^c$ can be foliated by connected
components of intersections with leaves of\/ $\tilde W_f$. Call this
foliation $\tilde T$. This is a well-defined one-dimensional
foliation since $T$ is two-dimensional while the leaves of\/ $\tilde
W_f$ are three-dimensional and both $T$ and $\tilde W_f$ are
transverse to $W_f^{ss}$. The holonomy map $H_f^{wu}\colon T_1\to
T_2$ maps $\tilde T_1$ into $\tilde T_2$ since $W_f^{wu}$
subfoliates $\tilde W_f$.

Analogously, any transversal $T$ can be foliated by connected
components of intersections with leaves of\/ $\bar W_f$. Call this
foliation $\bar T$.  Then $H_f^{wu}(\bar T_1)=\bar T_2$ since
$W_f^{wu}$ subfoliates $\bar W_f$.

Hence we can consider restrictions of\/ $H_f^{wu}$ to the leaves
of\/ $\tilde T_1$ and $\bar T_1$.

\begin{lemma}
The restriction of the holonomy $H_f^{wu}$ to a leaf of\/ $\tilde
T_1$, $H_f^{wu}\colon\tilde T_1(x)\to \tilde T_2(H_f^{wu}(x))$, is
$C^{1+\nu}$-differentiable. \label{holonomy1}
\end{lemma}

\begin{lemma}
The restriction of the holonomy $H_f^{wu}$ to a leaf of\/ $\bar
T_1$, $H_f^{wu}\colon\bar T_1(x)\to \bar T_2(H_f^{wu}(x))$, is
$C^{1+\nu}$-differentiable. \label{holonomy2}
\end{lemma}

Note that $\tilde T_i$ and $\bar T_i$ are transverse since $T_i$ is
transverse to $W_f^{c}$, $i=1, 2$. Hence, by the Regularity Lemma,
the holonomy $H_f^{wu}\colon T_1\to T_2$ is
$C^{1+\nu}$-differentiable.

To prove Lemmas~\ref{holonomy1} and~\ref{holonomy2} we need to
establish regularity of\/ $H$ in the strong unstable direction.

Given $x\in \mathbb T^4$, define $H_x\colon W_f^{su}(x)\to
W_L^{su}(H(x))$ by the following composition:
$$
W_f^{su}(x)\stackrel{H}{\longrightarrow}
U^{su}(H(x))\stackrel{H_L^{wu}}{\longrightarrow}W_L^{su}(H(x)).
$$
First, we map $W_f^{su}(x)$ into a H\"older-continuous curve
$U^{su}(H(x))\subset W_L^u(H(x))$ and then we project it on
$W_L^{su}(H(x))$ along the linear foliation $W_L^{wu}$, as shown on
the Figure~\ref{14_Hx}.


\begin{figure}[htbp]
\begin{center}

\begin{picture}(0,0)%
\includegraphics{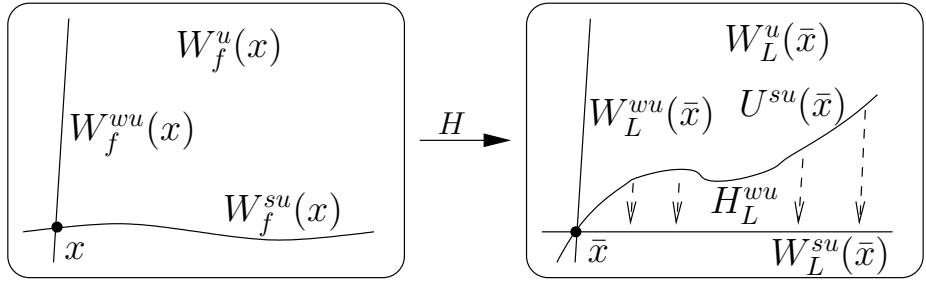}%
\end{picture}%
\setlength{\unitlength}{3947sp}%
\begingroup\makeatletter\ifx\SetFigFont\undefined%
\gdef\SetFigFont#1#2#3#4#5{%
\reset@font\fontsize{#1}{#2pt}%
\fontfamily{#3}\fontseries{#4}\fontshape{#5}%
\selectfont}%
\fi\endgroup%
\begin{picture}(5804,1749)(64,-598)
\put(459,277){\makebox(0,0)[lb]{\smash{{\SetFigFont{14}{16.8}{\rmdefault}{\mddefault}{\updefault}{\color[rgb]{0,0,0}$W_f^{wu}(x)$}%
}}}}
\put(3717,-490){\makebox(0,0)[lb]{\smash{{\SetFigFont{14}{16.8}{\rmdefault}{\mddefault}{\updefault}{\color[rgb]{0,0,0}$\bar x$}%
}}}}
\put(1130,803){\makebox(0,0)[lb]{\smash{{\SetFigFont{14}{16.8}{\rmdefault}{\mddefault}{\updefault}{\color[rgb]{0,0,0}$W_f^u(x)$}%
}}}}
\put(3721,372){\makebox(0,0)[lb]{\smash{{\SetFigFont{14}{16.8}{\rmdefault}{\mddefault}{\updefault}{\color[rgb]{0,0,0}$W_L^{wu}(\bar x)$}%
}}}}
\put(4576,839){\makebox(0,0)[lb]{\smash{{\SetFigFont{14}{16.8}{\rmdefault}{\mddefault}{\updefault}{\color[rgb]{0,0,0}$W_L^u(\bar x)$}%
}}}}
\put(4669,431){\makebox(0,0)[lb]{\smash{{\SetFigFont{14}{16.8}{\rmdefault}{\mddefault}{\updefault}{\color[rgb]{0,0,0}$U^{su}(\bar x)$}%
}}}}
\put(4867,-505){\makebox(0,0)[lb]{\smash{{\SetFigFont{14}{16.8}{\rmdefault}{\mddefault}{\updefault}{\color[rgb]{0,0,0}$W_L^{su}(\bar x)$}%
}}}}
\put(433,-469){\makebox(0,0)[lb]{\smash{{\SetFigFont{14}{16.8}{\rmdefault}{\mddefault}{\updefault}{\color[rgb]{0,0,0}$x$}%
}}}}
\put(1429,-229){\makebox(0,0)[lb]{\smash{{\SetFigFont{14}{16.8}{\rmdefault}{\mddefault}{\updefault}{\color[rgb]{0,0,0}$W_f^{su}(x)$}%
}}}}
\put(4484,-181){\makebox(0,0)[lb]{\smash{{\SetFigFont{14}{16.8}{\rmdefault}{\mddefault}{\updefault}{\color[rgb]{0,0,0}$H_L^{wu}$}%
}}}}
\put(2776,314){\makebox(0,0)[lb]{\smash{{\SetFigFont{12}{14.4}{\rmdefault}{\mddefault}{\updefault}{\color[rgb]{0,0,0}$H$}%
}}}}
\end{picture}%

\end{center}
\caption{The definition of\/ $H_x$. Here $\bar
x\stackrel{def}{=}H(x)$.}\label{14_Hx}
\end{figure}


\begin{lemma}
\label{HxisC1} For any $x\in\mathbb T^4$, the map $H_x$ is
$C^{1+\nu}$-differentiable.
\end{lemma}
\begin{proof}
Let us first show that $H_x$ is uniformly Lipshitz with a constant
that does not depend on $x$. Denote by $d$, $d_f^{su}$, $d_L^u$ and
$d_L^{su}$ the Riemannian distances on the universal cover $\mathbb
R^4$ along the leaves of\/ $W_f^{su}$, along the leaves of\/
$W_L^u$, and along the leaves of\/ $W_L^{su}$, respectively. First,
we show that $H_x$ is Lipshitz if the points are far enough apart.
Assume that $y, z \in W_f^{su}(x)$ and $d_f^{su}(y,z)\ge 1$. Then on
the universal cover,
\begin{align*}
d_L^{su}(H_x(y), H_x(z))&\stackrel{1}{\le} c_1d_L^{u}(H_x(y), H_x(z))\\
&\stackrel{2}{\le} c_1 c_2\inf\{d_L^u(\tilde y, \tilde z): \tilde
y\in W_L^{wu}(H_x(y)), \tilde z \in W_L^{wu}(H_x(z))
\}\\&\stackrel{3}{\le}
c_1 c_2 d_L^u(H(x), H(y))\\
&\stackrel{4}{\le} c_1 c_2 c_3 d(H(x), H(y))\\&\stackrel{5}{\le}
c_1c_2c_3c_4 d(y,z)\\&\stackrel{6}{\le} c_1c_2c_3c_4 d_f^{su}(y,z).
\end{align*}
The first and fourth inequalities hold since $W_L^{su}$ and $W_L^u$
are quasi-isometric. The second inequality holds with a universal
constant $c_2$ due to the uniform transversality of\/ $W_L^{wu}$ and
$W_f^{su}$. Inequalities 3 and 6 are obvious. The fifth inequality
holds since $d_f^{su}(y,z)\ge 1$ and the lift of the conjugacy
satisfies
$$
H(x+\vec m)=H(x)+\vec m, \;\; x\in \mathbb R^4,\;\; \vec m\in
\mathbb Z^4.
$$
Here we slightly abuse notation by denoting the lift and the map
itself by the same letter.

Now we need to show that $H_x$ is Lipschitz if\/ $y$ and $z$ are
close on the leaf. Notice that $H_x$ is the composition of\/ $H_y$
and the holonomy $H_L^{wu}\colon W_L^{su}(H(y))\to W_L^{su}(H(x))$,
which is just a translation. Hence, to show that $H_x$ is Lipschitz
at $y$ we only need to show that $H_y$ is Lipshitz at $y$.

So we fix $x$ and $y$ on $W_L^{su}(x)$ close to $x$ and show that
$d_L^{su}(H_x(x), H_x(y))\le c\, d_f^{su}(x, y)$. The argument here
is an adapted argument from the proof of Lemma 4 from~\cite{GG}. The
two major tools here are the Livshitz Theorem and the affine
distance-like functions $\tilde d_f^{su}$ and $\tilde d_L^{su}$ on
$W_f^{su}$ and $W_L^{su}$, respectively. We used the same
distance-like function on the foliation $V_i^f$ in the proof of
Lemma~\ref{abcd}. Recall the properties of\/ $\tilde d_f^{su}$:
\begin{itemize}
\item[(D1)] $\tilde d_f^{su}(x,y)=d_f^{su}(x,y)+o(d_f^{su}(x,y))$,
\item[(D2)] $\tilde d_f^{su}(f(x),f(y))=D_f^{su}(x)\tilde d_f^{su}(x,y)$,
\item[(D3)] $\forall K>0 \;\exists C>0$ such that
\begin{equation*}
\frac1C\tilde d_f^{su}(x,y)\le d_f^{su}(x,y)\le C\tilde
d_f^{su}(x,y)
\end{equation*}
whenever $d_f^{su}(x, y)<K$.
\end{itemize}

Consider the H\"older-continuous functions $D_f^{su}(\cdot)$ and
$D_L^{su}(H(\cdot))$. The assumption on the p.~d. of\/ $f$ and $L$
guarantee that the products of these derivatives along periodic
orbits coincide. Thus we can apply the Livshitz Theorem and get the
H\"older-continuous positive transfer function $P$ such that
\begin{equation*}
\forall n>0 \;\;\;\;\;
\prod_{i=0}^{n-1}\frac{D_L^{su}\left(H(f^i(x))\right)}{D_f^{su}(f^i(x))}=\frac{P(x)}{P(f^n(x))}.
\end{equation*}
Choose the smallest $N$ such that $d_f^{su}(f^N(x), f^N(y))\ge 1$.
Then
\begin{align*}
\frac{\tilde d_L^{su}(H_x(x),H_x(y))}{\tilde d_f^{su}(x,y)}&=
\prod_{i=0}^{N-1}\frac{D_L^{su}\left(L^i(H_x(x))\right)}{D_f^{wu}(f^i(x))}\cdot
\frac {\tilde d_L^{su}\left(L^N(H_x(x)),L^N(H_x(y))\right)} {\tilde d_f^{su}(f^N(x),f^N(y))}\\
&=\frac{P(x)}{P(f^N(x))}\cdot\frac {\tilde
d_L^{su}\left(H_{f^N(x)}(f^N(x)), H_{f^N(x)}(f^N(y))\right)} {\tilde
d_f^{su}(f^N(x),f^N(y))}\\&\le \frac{P(x)}{P(f^N(x))}\cdot
c_1c_2c_3c_4.
\end{align*}
The function $P$ is uniformly bounded away from zero and infinity.
Hence, together with (D3), this shows that $H_x$ is Lipschitz at $x$
uniformly in $x$ and hence is uniformly Lipschitz.

Next we apply the transitive point argument. Consider the SRB
measure $\mu^u$ which is the equilibrium state for the potential
minus the logarithm of the unstable jacobian of\/ $f$. It is well
known that $W_f^u$ is absolutely continuous with respect to $\mu^u$.
On a fixed leaf of\/ $W_f^u$, the foliation $W_f^{su}$ is absolutely
continuous with respect to the Lebesgue measure on the leaf (for
proof see~\cite{LY}, Section 4.2; they prove that the unstable
foliation is Lipschitz with center-unstable leaves, but the proof
goes through for strong unstable foliation within unstable leaves).
Hence $W_f^{su}$ is absolutely continuous with respect to $\mu^u$.

We know that $H_x$ is Lipschitz and hence almost everywhere
differentiable on $W_f^{su}(x)$. It is clear from the definition
that $H_x$ is differentiable at $y$ if and only if\/ $H_y$ is
differentiable at $y$. Thus it makes sense to speak about
differentiability at a point on a strong unstable leaf without
referring to a particular map $H_x$. The absolute continuity of\/
$W_f^{su}$ allows to conclude that $H_x$ is differentiable at $x$
for $\mu^u$-almost every $x$.

Since $\mu^u$ is ergodic and has full support we can consider a
transitive point $\bar x$ such that $H_{\bar x}$ is differentiable
at $\bar x$. Now $C^1$-differentiability of\/ $H_x$ for any
$x\in\mathbb T^4$ can be shown by an approximation argument: we
approximate the target point by iterates of\/ $\bar x$. The argument
is the same as the proof of Step~1, Lemma~5 from~\cite{GG} with
minimal modifications, so we omit it. This argument shows even more,
namely,
$$
D(H_x)(x)=\frac{P(x)}{P(\bar x)}D(H_{\bar x})(\bar x).
$$
Note that $D(H_x)(y)=D(H_y)(y)$. Hence $H_x$ maps the Lebesgue
measure on the leaf\/ $W_f^{su}(x)$ into an absolutely continuous
measure, $dy\mapsto \frac{P(y)}{P(\bar x)}d\text{Leb}$. Recall that
$P$ is H\"older-continuous. Hence $H_x$ is
$C^{1+\nu}$-differentiable.
\end{proof}

\begin{proof}[Proof of Lemma~\ref{holonomy1}]
We work in a ball $B$ inside of the leaf\/ $\tilde W_f(x)$ that
contains $\tilde T_1(x)$ and $\tilde T_2(H_f^{wu}(x))$. Recall that
$B$ is subfoliated by $W_f^{c}$ and $W_f^{su}$. We apply the
conjugacy map $H$  to the ball $B$. It maps $W_f^{su}$ and $W_f^{c}$
into $U^{su}$ and $W_L^{c}$, respectively. We  construct a shift map
$sh\colon H(B)\to \tilde W_L(H(x))$ in such a way that, for any $z$,
the leaf\/ $W_L^{c}(z)$ is $sh$-invariant and the action of\/ $sh$
on the leaf is a rigid translation.

Given a point $z\in H(B)$, let $y(z)=W_L^c(H(x))\cap U^{su}(z)$.
Define
$$
sh(z)=W_L^{su}(y(z))\cap W_L^{wu}(z).
$$
Clearly $sh(U^{su}(H(x)))=W_L^{su}(H(x))$. Moreover, by
Lemma~\ref{technical}, $sh(U^{su})=W_L^{su}$.


\begin{figure}[htbp]
\begin{center}

\begin{picture}(0,0)%
\includegraphics{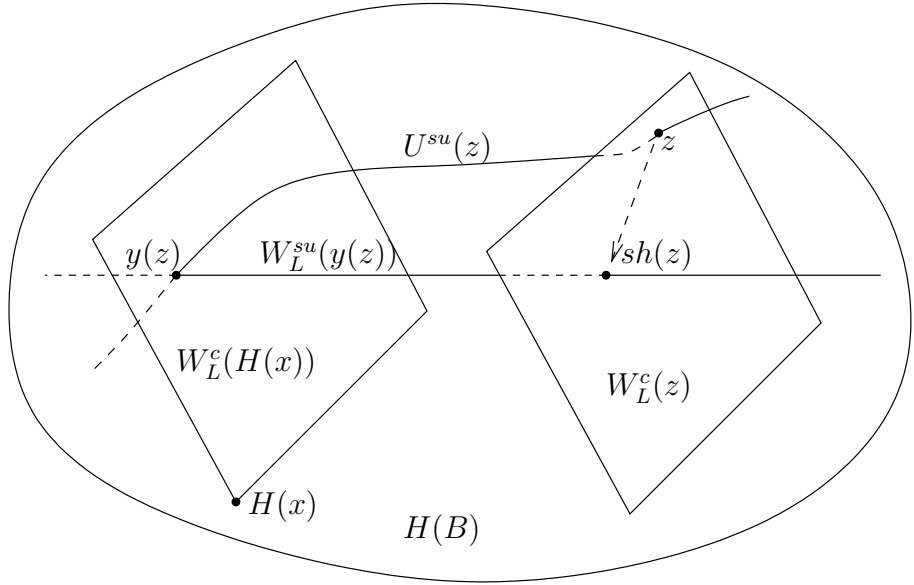}%
\end{picture}%
\setlength{\unitlength}{3947sp}%
\begingroup\makeatletter\ifx\SetFigFont\undefined%
\gdef\SetFigFont#1#2#3#4#5{%
\reset@font\fontsize{#1}{#2pt}%
\fontfamily{#3}\fontseries{#4}\fontshape{#5}%
\selectfont}%
\fi\endgroup%
\begin{picture}(5688,3660)(-311,-3051)
\put(2176,-2761){\makebox(0,0)[lb]{\smash{{\SetFigFont{12}{14.4}{\rmdefault}{\mddefault}{\updefault}{\color[rgb]{0,0,0}$H(B)$}%
}}}}
\put(1201,-2611){\makebox(0,0)[lb]{\smash{{\SetFigFont{12}{14.4}{\rmdefault}{\mddefault}{\updefault}{\color[rgb]{0,0,0}$H(x)$}%
}}}}
\put(751,-1711){\makebox(0,0)[lb]{\smash{{\SetFigFont{12}{14.4}{\rmdefault}{\mddefault}{\updefault}{\color[rgb]{0,0,0}$W_L^c(H(x))$}%
}}}}
\put(3451,-1861){\makebox(0,0)[lb]{\smash{{\SetFigFont{12}{14.4}{\rmdefault}{\mddefault}{\updefault}{\color[rgb]{0,0,0}$W_L^c(z)$}%
}}}}
\put(433,-1036){\makebox(0,0)[lb]{\smash{{\SetFigFont{12}{14.4}{\rmdefault}{\mddefault}{\updefault}{\color[rgb]{0,0,0}$y(z)$}%
}}}}
\put(2176,-361){\makebox(0,0)[lb]{\smash{{\SetFigFont{12}{14.4}{\rmdefault}{\mddefault}{\updefault}{\color[rgb]{0,0,0}$U^{su}(z)$}%
}}}}
\put(1273,-1036){\makebox(0,0)[lb]{\smash{{\SetFigFont{12}{14.4}{\rmdefault}{\mddefault}{\updefault}{\color[rgb]{0,0,0}$W_L^{su}(y(z))$}%
}}}}
\put(3793,-325){\makebox(0,0)[lb]{\smash{{\SetFigFont{12}{14.4}{\rmdefault}{\mddefault}{\updefault}{\color[rgb]{0,0,0}$z$}%
}}}}
\put(3541,-1036){\makebox(0,0)[lb]{\smash{{\SetFigFont{12}{14.4}{\rmdefault}{\mddefault}{\updefault}{\color[rgb]{0,0,0}$sh(z)$}%
}}}}
\end{picture}%

\end{center}
\caption{The definition of the shift.}\label{15_sh}
\end{figure}


The shift $sh$ is designed such that the composition $sh\circ H$
maps the foliation $W_f^{c}$ into $W_L^{c}$ and the foliation
$W_f^{su}$ into $W_L^{su}$. According to Lemma~\ref{HisC1}, $sh\circ
H$ is $C^{1+\nu}$-differentiable along $W_f^{c}$. Also notice that
the restriction of\/ $sh\circ H$ to a strong unstable leaf\/
$W_f^{su}$ is nothing but $H_y$ composed with constant parallel
transport along $W_L^{wu}$. Recall that $H_y$ is
$C^{1+\nu}$-differentiable by Lemma~\ref{HxisC1}. Hence, by the
Regularity Lemma, we conclude that $sh\circ H$ is
$C^{1+\nu}$-diffeomorphism.

Therefore $\hat T_1=sh\circ H(\tilde T_1(x))$ and  $\hat T_2=sh\circ
H(\tilde T_2(H_f^{wu}(x)))$ are smooth curves inside of\/ $H(B)$ and
the holonomy map $H_f^{wu}$ can be represented as a composition as
shown on the commutative diagram
$$
\begin{CD}
\tilde T_1(x) @>H_f^{wu}>> \tilde T_2(H_f^{wu}(x))\\
@Vsh\circ HVV @Vsh\circ HVV\\
\hat T_1 @>H_L^{wu}>> \hat T_2
\end{CD}
\medskip
$$
The holonomy $H_L^{wu}$ is smooth since $W_L^{wu}$ is a foliation by
straight lines. Hence $H_f^{wu}$ is $C^{1+\nu}$-differentiable.
\end{proof}

\begin{remark}
Notice that this argument completely avoids dealing with the
geometry of transversals, \ie their relative position to the
foliations.
\end{remark}

\begin{proof}[Proof of Lemma~\ref{holonomy2}]
We use exactly the same argument as in the previous proof. Notice
that the picture is not completely symmetric compared to the picture
in Lemma~\ref{holonomy1} since we are dealing with the weak unstable
holonomy. Nevertheless the argument goes through by looking at
transversals $\bar T_1(x)$ and $\bar T_2(H_f^{wu}(x))$ on the leaf
of\/ $\bar W_f$. The shift map must be constructed in such a way
that it maps $U^{ss}$ into $W_L^{ss}$.
\end{proof}

\begin{proof}[Proof of Lemma~\ref{SisC1}]
In this proof we exploit the same idea of composing $H$ with some
shift map. We fix $S_1=S(x_1)\in S$ which is, a priori, just an
embedded topological torus. We assume that $x_1\in W_f^{wu}(x_0)$.
It is easy to see that this is not restrictive.

Foliate $S_0$ and $S_1$ by $\tilde T_0$, $\bar T_0$ and $\tilde
T_1$, $\bar T_1$, respectively, by taking intersections with leaves
of\/ $\tilde W_f$ and $\bar W_f$. To prove the lemma we only have to
show that the leaves of\/ $\tilde T_1$ and $\bar T_1$ are
$C^{1+\nu}$-differentiable curves.

We restrict our attention to a leaf of\/ $\tilde W_f$. Construct the
shift map $sh$ in the same way as in Lemma~\ref{holonomy1}. Fix an
$x\in S_0$ and let $\hat T_0=sh\circ H(\tilde T_0(x))$, $\hat
T_1=sh\circ H(\tilde T_1(H_f^{wu}(x)))$.

$\hat T_0$ is a $C^{1+\nu}$-curve since $sh\circ H$ is
$C^{1+\nu}$-diffeomorphism. By the definition of\/ $S_1$,
$$\forall y\in \tilde T_0 \;\;\;d^{wu}(y, H_f^{wu}(y))=d^{wu}(x, H_L^{wu}(x)).$$
Recalling the definition of\/ $d^{wu}$, we see that the conjugacy
$H$ acts as an isometry on a weak unstable leaf. Obviously $sh$ is
an isometry when restricted to a weak unstable leaf as well.
Therefore
$$
\forall y\in \hat T_0 \;\;\;d(y, H_f^{wu}(y))=d(sh\circ H(x),
H_L^{wu}(sh\circ H(x))),
$$
where $d$ is the Riemannian distance along $W_L^{wu}$. Hence $\hat
T_1$ is smooth as a parallel translation of\/ $\hat T_0$. We
conclude that $\tilde T_1(H_f^{wu}(x))=(sh\circ H)^{-1}(\hat T_1)$
is $C^{1+\nu}$-curve.

Repeating the same argument for $\bar T_0(x)$ and $\bar
T_1(H_f^{wu}(x))$, we can show that  $\bar T_1(H_f^{wu}(x))$  is
$C^{1+\nu}$-curve. Hence the lemma is proved.
\end{proof}

\section{Proof of Theorem D}
\label{moduli}
\subsection{Scheme of the proof of Theorem D}
We choose $\mathcal U$ in the same way as in~\ref{schemeC}. The only
difference is that $L$ is given by~(\ref{L}) not
by~(\ref{Lgeneral}).

Given $f\in\mathcal U$ we denote by $W_f^c$ the two-dimensional
central foliation. Take $f$ and $g$ in $\mathcal U$. Then they are
conjugate, $h\circ f=g\circ h$.
\begin{proposition}
\label{central_smoothness} Assume that $f$ and  $g$ have the same
p.~d. Then $h(W_f^c)=W_g^c$ and the conjugacy $h$ is
$C^{1+\nu}$-differentiable along $W_f^c$.
\end{proposition}

\begin{remark}
In the proof we only need coincidence of the p.~d. in the central
direction.
\end{remark}

After we have differentiability along the central foliation, the
strong stable and unstable foliation moduli come into the picture.

\begin{lemma}
\label{strong_into_strong} Assume that $f$ and $g$ have the same
p.~d. and the same strong unstable foliation moduli. Then
$h(W_f^{su})=W_g^{su}$.
\end{lemma}

Now the proof of Theorem D follows immediately. Coincidence of the
p.~d. in the strong unstable direction guarantees
$C^{1+\nu}$-differentiability of\/ $h$ along $W_f^{su}$. This can be
done by a transitive-point argument with the SRB-measure in the same
way as the proof of Lemma~\ref{step2}. Then we repeat everything for
the strong stable foliation. After this, we apply the Journ\'e
Regularity Lemma twice to conclude that $h$ is
$C^{1+\nu}$-differentiable.

In particular, this argument shows that in the counterexample of de
la Llave, the strong stable and unstable foliations are not
preserved by the conjugacy. We can make use of this fact by
extending the counterexample for diffeomorphisms of the form
$(x,y)\mapsto(Ax+\vec \varphi(y), g(y))$. Namely, take $L=(Ax,By)$
and $\tilde L=(Ax+\vec \varphi(y), By)$ as in~(\ref{L})
and~(\ref{tildeL}), respectively. We know that the strong foliations
of\/  $L$ and $\tilde L$ do not match. The strong foliations depend
continuously on the diffeomorphism in the $C^1$-topology. If we
consider diffeomorphisms $L'(x,y)=(Ax, g(y))$ and $\tilde
L'(x,y)=(Ax+\vec\varphi(y), g(y))$ with $g$ sufficiently $C^1$-close
to $B$, then the conjugacy between $L'$ and $\tilde L'$ is $C^0$
close to the conjugacy between $L$ and $\tilde L$. Therefore the
strong foliations of\/ $L'$ and $\tilde L'$ do not match as well. We
conclude that $L'$ and $\tilde L'$ are not $C^1$-conjugate as well.

We do not know how to show that the counterexample extends to the
whole neighborhood $\mathcal U$.

\begin{con}
For any $f\in\mathcal U$ there exists $g\in\mathcal U$ with the same
p.~d. which is not $C^1$-conjugate to $f$.
\end{con}

\begin{proof}[Proof of Lemma~\ref{strong_into_strong}]
Let $U=h^{-1}(W_g^{su})$. We need to show that $U=W_f^{su}$. The
main tool is the following statement.
\begin{lemma}
\label{last} Consider a point $a\in\mathbb T^4$. Suppose that there
is a point $b\ne a$, $b\in W_f^{su}(a)\cap U(a)$. Let $c\in
W_f^{wu}(a)$ and $d=W_f^{wu}(b)\cap W_f^{su}(c)$, $e=W_f^{wu}(b)\cap
U(c)$. Then $d=e$.
\end{lemma}
This means that the ``intersection structure" of\/ $U$ and
$W_f^{su}$ is invariant under the shifts along $W_f^{wu}$. We refer
to~\cite{GG} for the proof. Claim 1 in~\cite{GG} is exactly the same
statement in the context of\/ $\mathbb T^3$. The proof uses
Proposition~\ref{central_smoothness}.

According to the definition of strong unstable foliation moduli, we
have to distinguish two cases.

First assume~(\ref{moduli2}). It follows that there is a curve
$\EuScript C\subset W_f^{su}(x)$ that corresponds to the interval
$I$ such that $\EuScript C\subset U$ as well. Let
$$
\EuScript S=\bigcup_{a\in\EuScript C} W_f^{wu}(a).
$$
Obviously $\EuScript S\subset W_f^u(x)$. It follows from
Lemma~\ref{last} that $W_f^{su}=U$ when restricted to $\EuScript S$.
Then $W_f^{su}=U$ when restricted to $f^n(\EuScript S)$, $n>0$ as
well. It remains to notice that $\bigcup_{n>0} f^n(\EuScript S)$ is
dense in $\mathbb T^4$ since $\text{length}(f^n(\EuScript
C))\to\infty$ as $n\to\infty$. Hence $W_f^{su}=U$.

Now let us consider the second case. Namely, assume~(\ref{moduli1}).
Let $x_0$ be a fixed point. Define $x_1=\EuScript
I^{su}(x_0)^{-1}(t)$. Then by~(\ref{moduli1})  $x_1\in
W_f^{su}(x_0)\cap U(x_0)$. We continue to define a sequence $\{x_k;
k\ge 0\}$ inductively. Given $x_k$, define $x_{k+1}=\EuScript
I^{su}(x_k)^{-1}(t)$. Then for any $k$, $\;x_{k+1}\in
W_f^{su}(x_k)\cap U(x_k)= W_f^{su}(x_0)\cap U(x_0)$. Obviously
$f^{-n}(x_k)\in W_f^{su}(x_0)\cap U(x_0)$  as well.

The map $\EuScript I^{su}(x_0)$ is an isometry, hence
$d_f^{su}(x_k,x_{k+1})$ does not depend on $k$. Therefore the set
$\{f^{-n}(x_k); n\ge 0, k\ge 0\}$ is dense in $W_f^{su}(x_0)$, which
guarantees that $W_f^{su}(x_0)=U(x_0)$. We can proceed as in the
first case now to conclude that $W_f^{su}=U$.
\end{proof}

\subsection{Smoothness along the central foliation: proof of Proposition~\ref{central_smoothness}}
We apply the transitive point argument as in the proof of
Lemma~\ref{step2}. The technical difficulty that we have to deal
with is that the leaves of\/ $W^c$ are not dense in $\mathbb T^4$.

The conjugacy $h$ preserves the weak stable and unstable foliations.
By the Regularity Lemma we only need to show
$C^{1+\nu}$-differentiability of\/ $h$ along these one-dimensional
foliations. For concreteness, we work with the weak unstable
foliation $W_f^{wu}$.

For the transitive point argument to work, we have to find an
invariant measure $\mu$ such that $\mu$-a.~e. point is transitive
($\overline{\{f^n(x); n\ge 0\}}=\mathbb T^4$) and $W_f^{wu}$ is
absolutely continuous with respect to $\mu$. Provided that we have
such a measure $\mu$, $\;C^{1+\nu}$-differentiability of\/ $h$ along
$W_f^{wu}$ is proved in the same way as Lemma~5 from~\cite{GG}.

We modify the construction from the proof of Lemma~\ref{step2}.
Consider the space $\mathcal T$ of the leaves of\/ $W_f^c$. Clearly
this is a topological space homeomorphic to a two-torus. Let $\tilde
f\colon\mathcal T\to\mathcal T$ be the factor dynamics of\/ $f$.
Since the conjugacy to the linear model $L$ maps the central leaves
to the central leaves, $\tilde f$ is conjugate to the automorphism
$B\colon\mathbb T^2\to\mathbb T^2$, $\tilde h\circ B=\tilde
f\circ\tilde h$. Then the measure $\tilde \mu=\tilde
h_*(\text{Lebesgue})$ is $\tilde f$-invariant and ergodic.

Pick a point $x_0$ on a $\tilde\mu$-typical central leaf. Let
$\mathcal V_0$ be an open bounded neighborhood of\/ $x_0$ in
$W_f^{wu}(x_0)$. Given $x$ and $y\in W_f^{wu}(x)$, let
$$
\rho(x,y)=\prod_{n\ge
0}\frac{D_f^{wu}(f^{-n}(y))}{D_f^{wu}(f^{-n}(x))}.
$$
Consider a probability measure $\eta_0$ supported on $\mathcal V_0$
with density proportional to $\rho(x_0,\cdot)$. For $n>0$ define
$$\mathcal V_n=f^n(\mathcal V_0),\; \eta_n=(f^n)_*\eta_0.$$
Let
$$\mu_n=\frac1n\sum_{i=0}^{n-1}\eta_i.$$
An accumulation point of\/ $\{\mu_n; n\ge 0\}$ is the measure $\mu$
that we are looking for.

By the choice of\/ $x_0$ the projection of\/ $\mu$ to $\mathcal T$
is $\tilde\mu$.

The foliation $W_f^{wu}$ is absolutely continuous with respect to
$\mu$. We refer to~\cite{PS} or~\cite{GG} for the proof.
In~\cite{GG}, $x_0$ is a fixed point but we do not use it in the
proof of absolute continuity.

Now we have to argue that $\mu$-a.~e. point is transitive. We fix a
ball in $\mathbb T^4$ and we show that a.~e. point visits the ball
infinitely many times. Then to prove transitivity, we only need to
cover $\mathbb T^4$ by a countable collection of balls such that
every point is contained in an arbitrarily small ball.

So let us fix a ball $B'$ and a slightly smaller ball $B$, $B\subset
B'$. Let $\psi$ be a nonnegative continuous function supported on
$B'$  and equal to $1$ on $B$. By the Birkhoff Ergodic Theorem,
\begin{equation}
\label{ergodictheorem} E(\psi|\mathcal
I)=\lim_{n\to\infty}\frac1n\sum_{i=0}^{n-1}\psi\circ f^i
\end{equation}
where $\mathcal I$ is the $\sigma$-algebra of\/ $f$-invariant sets.

Let $A=\{x: E(\psi|\mathcal I)(x)=0\}$. Then $\mu(A\cap B)=0$ since
$\int_A\psi d\mu=\int_A E(\psi|\mathcal I)d\mu=0$. Hence
\begin{equation*}
E(\psi|\mathcal I)(x)>0\;\; \text{for} \;\;\mu-\;\text{a.~e.}\;\;
x\in B.
\end{equation*}

Let $\tilde B\subset B$ be a slightly smaller ball and let
$W^c(\tilde B)=\bigcup_{x\in \tilde B}W_f^c(x)$. Since weak unstable
leaves are dense in the corresponding central leaves it is possible
to find $R>0$ such that
$$
W^c(\tilde B)\subset \bigcup_{x\in B} W_f^{wu}(x, R).
$$
Applying the standard Hopf argument,  for $\mu$-a.~e. $x$, the
function $E(\psi|\mathcal I)$ is constant on $W(x, R)$. Now the
absolute continuity of\/ $W_f^{wu}$ together with the above
observations show that
\begin{equation*}
E(\psi|\mathcal I)(x)>0\;\; \text{for} \;\;\mu-\text{a.~e.}\;\; x\in
W^c(\tilde B).
\end{equation*}
Obviously
\begin{equation*}
\forall n\;\;\;\; E(\psi|\mathcal I)(x)>0\;\; \text{for}
\;\;\mu-\text{a.~e.}\;\; x\in f^n(B).
\end{equation*}
Repeat the same argument to get
\begin{equation*}
\forall n\;\;\;\; E(\psi|\mathcal I)(x)>0\;\; \text{for}
\;\;\mu-\text{a.~e.}\;\; x\in W^c(f^n(\tilde B)).
\end{equation*}
Let $\EuScript O(\tilde B)=\bigcup_{n\in\mathbb Z}f^n(\tilde B)$ and
$W^c(\EuScript O(\tilde B))=\bigcap_{x\in\EuScript O(\tilde B)}
W_f^c(x)$. Then
\begin{equation*}
E(\psi|\mathcal I)(x)>0\;\; \text{for} \;\;\mu-\text{a.~e.}\;\; x\in
W^c(\EuScript O(\tilde B)).
\end{equation*}
Set $W^c(\EuScript O(\tilde B))$ is $W_f^c$-saturated. Hence
$\mu(W^c(\EuScript O(\tilde B)))$ is equal to the
$\tilde\mu$-measure of its projection $\text{proj}(W^c(\EuScript
O(\tilde B)))=\text{proj}(\EuScript O(\tilde B))$ on $\mathcal T$.
Set $\text{proj}(\EuScript O(\tilde B))$ is an open $\tilde
f$-invariant set. By ergodicity of\/ $\tilde f$, it has full
measure. Hence $\mu(W^c(\EuScript O(\tilde B)))=1$ and
\begin{equation*}
E(\psi|\mathcal I)(x)>0\;\; \text{for} \;\;\mu-\text{a.~e.}\;\;
x\in\mathbb T^4.
\end{equation*}
According to~(\ref{ergodictheorem}) this means that $\mu$-a.~e. $x$
visits $B'$ infinitely many times.

 \end{document}